\documentclass[12pt]{amsart}
\usepackage{geometry,verbatim}
\usepackage[shadow]{todonotes}
\usepackage{amssymb,latexsym,amsfonts,mathtools}
\usepackage{mathrsfs} 
\usepackage{enumitem} 
\usepackage{microtype} 
\usepackage[shadow]{todonotes}
\usepackage{tikz}
\usetikzlibrary{arrows,backgrounds,patterns,shapes,calc,decorations.pathreplacing}
\newenvironment{enumerate-(a)}{\begin{enumerate}[label={\upshape (\alph*)}, leftmargin=2pc]}{\end{enumerate}}
\newenvironment{enumerate-(A)}{\begin{enumerate}[label={\upshape (\Alph*)}, leftmargin=2pc]}{\end{enumerate}}
\newenvironment{enumerate-(i)}{\begin{enumerate}[label={\upshape (\roman*)}, leftmargin=2pc]}{\end{enumerate}}
\newenvironment{itemizenew}{\begin{itemize}[leftmargin=2pc]}{\end{itemize}}
\theoremstyle{plain}
\newtheorem{theorem}{Theorem}[section]

\newtheorem{proposition}[theorem]{Proposition}
\newtheorem{lemma}[theorem]{Lemma}
\newtheorem{corollary}[theorem]{Corollary}
\newtheorem{claim}{Claim}[theorem]

\theoremstyle{definition}
\newtheorem{definition}[theorem]{Definition}
\theoremstyle{remark}
\newtheorem{remark}[theorem]{Remark}
\newtheorem{remarks}[theorem]{Remarks}
\newtheorem{example}[theorem]{Example}

\newcommand{\AND}{\mathbin{\, \wedge \,}}
\newcommand{\OR}{\mathbin{\, \vee \,}}
\renewcommand{\implies}{\Rightarrow}	
\newcommand{\IMPLIES}{\mathbin{\, \Rightarrow \,}}
\renewcommand{\iff}{\mathbin{\Leftrightarrow }}
\newcommand{\IFF}{\mathbin{\, \Leftrightarrow \,}}
\newcommand{\FORALL}[1]{\forall {#1} \, }
\newcommand{\FORALLS}[2]{\forall^{#1} {#2} \, }
\newcommand{\EXISTS}[1]{\exists {#1} \, }
\newcommand{\EXISTSONE}[1]{\exists ! {#1} \, }
\newcommand{\EXISTSS}[2]{\exists^{#1} {#2} \, }
\newcommand{\R}{\mathbb{R}}	
\newcommand{\Q}{\mathbb{Q}}	
\newcommand{\Mid}{\boldsymbol\mid}			
\newcommand{\equalsdef}{\stackrel{\text{\tiny\rm def}}{=}} 	
\newcommand{\setof}[2]{\mathopen \{{#1}\Mid{#2} \mathclose\}} 
\newcommand{\setofLR}[2]{\left \{{#1} \Mid {#2} \right\}} 
\newcommand{\setLR}[1]{\left \{ {#1} \right \}} 
\newcommand{\Pow}{\mathscr{P}}		
\newcommand{\seqofLR}[2]{\left \langle #1 \Mid #2 \right \rangle} 
\newcommand{\seqLR}[1]{\left \langle #1 \right \rangle}	

\newcommand{\eq}[1]{{\boldsymbol [}{#1} {\boldsymbol ]}}	
\newcommand{\symdif}{\mathop{\triangle}}	
\newcommand{\LOC}[2]{{#1}_{\lfloor {#2}\rfloor}} 
\newcommand{\pre}[2]{\prescript{#1}{}{#2}}				
\DeclareMathOperator{\dom}{dom} 		
\DeclareMathOperator{\ran}{ran} 		
\newcommand{\conc}{{}^\smallfrown}		
\DeclareMathOperator{\lh}{lh}				
\newcommand{\onto}{\twoheadrightarrow}
\newcommand{\lelex}{<_{\mathrm{lex}}}		
\DeclareMathOperator{\Int}{Int} 		
\DeclareMathOperator{\Cl}{Cl} 		
\DeclareMathOperator{\Fr}{Fr} 		
\newcommand{\Nbhd}{{\boldsymbol N}\!} 
\DeclareMathOperator{\Ball}{B}
\newcommand{\bSigma}{\boldsymbol{\Sigma}}
\newcommand{\bPi}{\boldsymbol{\Pi}}
\newcommand{\bGamma}{\boldsymbol{\Gamma}}

\newcommand{\bDelta}{\boldsymbol{\Delta}}
\newcommand{\Fsigma}{\mathbf{F}_\sigma}	

\newcommand{\NULL}{\textrm{\scshape Null}}	
\newcommand{\Bor}{\textrm{\scshape Bor}}	
\newcommand{\MEAS}{\textrm{\scshape Meas}}	
\newcommand{\MALG}{\textrm{\scshape Malg}}	
\newcommand{\ZF}{\mathsf{ZF}} 		
\newcommand{\ZFC}{\mathsf{ZFC}}		
\newcommand{\DC}{\mathsf{DC}}			
\DeclareMathOperator{\Rake}{Rake}

\newcommand{\Rakep}{\Rake^+}
\DeclareMathOperator{\Plus}{Plus}
\DeclareMathOperator{\Sum}{Sum}
\DeclareMathOperator{\NATURAL}{Nat}
\DeclareMathOperator{\FLAT}{Flat}
\newcommand{\uu}{\boldsymbol u}
\newcommand{\hh}{\boldsymbol{h}}
\newcommand{\mm}{\boldsymbol m}

\DeclareMathOperator{\cof}{cof} 			
\newcommand{\AD}{\mathsf{AD}}		
\newcommand{\I}{\mathbf{I}}
\newcommand{\II}{\mathbf{II}}
\newcommand{\GL}{G_{\mathrm{L}}}
\newcommand{\GW}{G_{\mathrm{W}}}
\newcommand{\leqW}{\leq_{\mathrm{W}}}
\newcommand{\leql}{\leq_{\mathrm{L}}}
\newcommand{\nleqW}{\nleq_{\mathrm{W}}}
\newcommand{\leW}{<_{\mathrm{W}}}
\newcommand{\lel}{<_{\mathrm{L}}}

\newcommand{\Wequiv}{\equiv_{\mathrm{W}}}
\newcommand{\lequiv}{\equiv_{\mathrm{L}}}

\newcommand{\Wdeg}[1]{\eq{#1}_{\mathrm{W}}}

\newcommand{\Wrank}[1]{\RANK{#1}_{\mathrm{W}}}
\newcommand{\Wsum}{\mathbin{\boldsymbol{+}} }
\newcommand{\Wtree}{\boldsymbol{T}}
\DeclareMathOperator{\supt}{supt}
\newcommand{\RANK}[1]{\left \Vert #1 \right \Vert}
\newcommand{\body}[1]{\left [ {#1} \right ]} 
\newcommand{\markdef}[1]{\textbf{#1}}
\newcommand{\densitytree}{\boldsymbol{D}}

\usepackage[naturalnames]{hyperref}

\title{The Descriptive Set Theory of the Lebesgue Density Theorem}
\author{Alessandro Andretta}
\address{Dipartimento di Matematica, Universit\`a di Torino, via Carlo Alberto 10, 10123 Torino --- Italy}
\email{alessandro.andretta@unito.it}
\author{Riccardo Camerlo}
\address{Dipartimento di Matematica, Politecnico di Torino, Corso Duca degli Abruzzi 24, 10129 Torino --- Italy}
\email{camerlo@calvino.polito.it}
\dedicatory{In memory of Greg Hjorth}
\date{\today}                                           

\begin{document}
\begin{abstract}
Given an equivalence class \( \eq{A} \) in the measure algebra of the Cantor space, let 
\( \hat \Phi ( \eq{A} ) \) be the set of points having density \( 1 \) in \( A \).
Sets of the form \( \hat \Phi ( \eq{A} ) \) are called \( \mathcal{T} \)-regular.
We establish several results about \( \mathcal{T} \)-regular sets.
Among these, we show that \( \mathcal{T} \)-regular sets can have any complexity 
within \( \bPi^{0}_{3}\) (=\( \mathbf{F}_{ \sigma\delta}\)), 
that is for any \( \bPi^{0}_{3}\) subset \( X \) of the Cantor space
there is a \( \mathcal{T} \)-regular set that has the same topological complexity of \( X \).
Nevertheless, the generic \( \mathcal{T} \)-regular set is \( \bPi^{0}_{3}\)-complete, meaning that the classes 
\( \eq{A} \) such that \( \hat{ \Phi} ( \eq{A} ) \) is \( \bPi^{0}_{3}\)-complete form a comeagre 
subset of the measure algebra.
We prove that this set is also dense in the sense of forcing, as \( \mathcal{T} \)-regular sets with empty interior 
turn out to be \( \bPi^{0}_{3}\)-complete.
Finally we show that the generic \( \eq{A} \) does not contain a \( \bDelta^{0}_{2}\) set, 
i.e., a set which is in \(  \mathbf{F}_ \sigma \cap \mathbf{G}_ \delta  \).
\end{abstract}
\maketitle

\section{Introduction}
The measure algebra of a probability Borel measure 
\( \mu \) on a standard Borel space \( X \), is the quotient
\[
\MALG ( X , \mu )= \frac{\MEAS ( X ,  \mu ) }{\NULL ( X , \mu ) }
\] 
where \(\MEAS ( X , \mu ) \) is the \( \sigma \)-algebra of the \( \mu \)-measurable 
subsets of \( X \) and 
\( \NULL ( X ,  \mu ) \) is the \( \sigma \)-ideal of the sets of \( \mu \)-measure \( 0 \).
It  can be obtained by taking the quotient of \( \Bor ( X ) \), 
the \( \sigma \)-algebra of Borel subsets of \( X \), and it is canonical, in the 
sense that \( \MALG ( X , \mu ) \) is unique, up to isomorphism, for
any continuous probability measure \( \mu \) on a standard Borel 
space~\cite[p.~116]{Kechris:1995kc}.
The function \( \left ( \eq{A} , \eq{B} \right ) \mapsto \mu \left ( A \symdif B  \right ) \)
is a separable complete metric that turns \( \MALG \) into a Polish space.

In order to state our results in a convenient way, we will take 
the measure space to be the Cantor set \( \pre{\omega}{2} \) 
with the Lebesgue measure \( \mu \), also known as the Bernoulli or 
coin-tossing measure.

A point \( x \in \pre{\omega}{2} \) is said to have \markdef{density \( r \in [ 0 ; 1 ] \) 
in a measurable set \( A \subseteq \pre{\omega}{2} \)} if
\begin{equation}\label{eq:Lebesguedensitylimit}
\mathcal{D}_A ( x )  \equalsdef  \lim_{n \to \infty} \frac{\mu \left ( A \cap \Nbhd_{x \restriction n} \right )}
{ \mu \left ( \Nbhd_{x \restriction n} \right )} = r ,
\end{equation}
where \( \Nbhd_s = \setofLR{ x \in \pre{\omega}{2} }{s \subset x}\) is the basic open neighborhhod
determined by the finite sequence \( s \).
The map \( \mathcal{D}_A \) is called the \markdef{density function for the set \( A \)}.
Note that \( \mathcal{D}_A ( x ) \) does not necessarily exist for all \( x \), since the limit might not converge.
However, for all \( x \in \pre{\omega}{2} \)
\begin{equation*}\label{eq:densitycomplement0}
\mathcal{D}_A ( x ) = 1 - \mathcal{D}_{\neg A} ( x ) 
\end{equation*}
where \( \neg A \equalsdef \pre{\omega}{2} \setminus A \) is the complement of \( A \),
meaning that if one of the two limits exists, so does the other, and equality holds.
The following result, known as the Lebesgue Density Theorem
says that almost every \( x \in A \) has density \( 1 \) in \( A \).
\begin{theorem}\label{th:Lebesguedensity}
Let \( A \subseteq \pre{\omega}{2} \) be Lebesgue measurable.
Then 
\[
\Phi  ( A ) = \setofLR{ x \in \pre{\omega}{2} }{ x \text{ has density \( 1 \) in } A }
\]
is Lebesgue measurable, and \( \mu ( A \symdif \Phi  ( A ) ) = 0 \).
In other words, \( \mathcal{D}_A ( x ) \) agrees with the characteristic function of \( A \),
for almost every \( x \in \pre{\omega}{2} \).
\end{theorem}

If \( A , B \subseteq \pre{\omega}{2} \) are measurable sets and \( \mu ( A\symdif  B ) = 0 \), then 
\( \mathcal{D}_A ( x ) = \mathcal{D}_B ( x ) \) for all \( x \) hence \( \Phi ( A ) = \Phi ( B ) \).
The map \( \Phi  \colon \MEAS \to \MEAS \) induces a function 
\[ 
\hat{\Phi}  \colon \MALG \to \MEAS 
\]
selecting a representative in each \( \equiv \)-equivalence class, where \( \equiv \) is equality up to a null set,
\[
A \equiv B \IFF A \symdif B \in \NULL.
\]

In the literature (see e.g.~\cite[Theorem 3.21, p.~17]{Oxtoby:1980kr},
or~\cite[Corollary 6.2.6, p.~184]{Cohn:1993fk}, or~\cite[Corollary 223B, p.~63]{Fremlin:2002mf}), 
the Lebesgue Density Theorem is stated (and proved) for  
\( \R^k\), rather than the Cantor space, with the density of a point \( x \in \R^k \)
in a measurable set \( A \subseteq \R^k \) defined as the limit
\[
\lim_{ \varepsilon \to 0} \frac{ \lambda^k ( A \cap \Ball_d ( x ; \varepsilon ) )}
{ \lambda^k ( \Ball_d ( x ; \varepsilon ) )},
\]
where \(  \Ball_d ( x ; \varepsilon ) = \setofLR{ y \in \R^k }
{ d ( y , x ) < \varepsilon }\) 
is the open ball centered around \( x \) of radius \( \varepsilon \),
and \( d \) and  \( \lambda^k \) are, respectively, the Euclidean metric and the Lebesgue measure on \( \R^k \).
Density functions can be defined for every Borel measure \( \mu \) on a metric space \( ( X , d ) \),
but the Lebesgue Density Theorem might not hold 
even when \( ( X , d ) \) is Polish (D.H.~Fremlin, personal communication).
On the other hand, for every Borel probability measure \( \mu \) on a standard Borel space \( X \), 
the algebra \( \MALG ( X , \mu ) \) admits a Borel selector, being isomorphic 
to the measure algebra on the Cantor set (see Proposition~\ref{prop:Pi03}).
This paper focuses on the Cantor space, so for the reader's benefit we include a proof 
of Theorem~\ref{th:Lebesguedensity} in Section~\ref{sec:proofofdensitythm}.

The sets of the form \( \Phi ( A ) \) are known to be \( \bPi^{0}_{3} \) 
--- see e.g.~\cite[p.~681]{Wilczynski:2002rz}.
In this paper we shall follow the logicians' notation and write \( \bSigma^{0}_{1}\) 
for the family of open sets, \( \bSigma^{0}_{n + 1} \) for the family of countable 
unions of  \( \bPi^{0}_{n}\) sets, \( \bPi^{0}_{n}\) for the family of
complements of \( \bSigma^{0}_{n}\) sets, and \( \bDelta^{0}_{n} \) for
\( \bSigma^{0}_{n} \cap \bPi^{0}_{n}\).
Therefore \( \bPi^{0}_{3}\) is simply the collection of all \( \mathbf{F}_{ \sigma \delta }\) sets, 
and, by a theorem of Wadge, an \( \mathbf{F}_{ \sigma \delta }\) set which is not \(  \mathbf{G}_{ \delta \sigma }\)
is complete \( \bPi^{0}_{3}\) (see~\cite[Section22.B]{Kechris:1995kc}).

We shall prove some results on the complexity of \( \Phi ( A ) \).

\begin{theorem}\label{th:Pi03complete}
There is an open set \(U \) and a closed set \( C \) such that 
\( \Phi  ( U ) = \Phi ( C ) \) is complete \( \bPi^{0}_{3} \).
\end{theorem}

In fact there are many sets of the form \( \Phi  ( A ) \) which are complete \( \bPi^{0}_{3} \).

\begin{theorem}\label{th:emptyinteriorPi03}
If \( \emptyset \neq \Phi ( A ) \) has empty interior, then \( \Phi ( A  ) \) is complete \( \bPi^{0}_{3} \).
\end{theorem}

Not  every set \( \Phi  ( A ) \) is complete \( \bPi^{0}_{3} \) --- in fact the opposite is true.
In order to formulate the next result in a convenient form, recall that two subsets
\( A , B \subseteq \pre{\omega}{2} \) are Wadge equivalent \( A \Wequiv B \)
just in case each one is the continuous preimage of the other.
A \( \Wequiv \)-equivalence class is called a Wadge degree.

\begin{theorem}\label{th:belowDelta03}
For each \( \bDelta^{0}_{3} \)  
set \( A \subseteq \pre{\omega}{2}  \) there are
an open set \( U \) and a closed set \( C \) such that 
\( \Phi  ( U ) = \Phi ( C ) \Wequiv A \).
\end{theorem}

Although Theorems~\ref{th:Pi03complete} and~\ref{th:belowDelta03} could be merged into a single statement 
\begin{theorem}\label{th:global}
For every Wadge degree \( \boldsymbol{d} \subseteq \bPi^{0}_{3}\) there are
\( U \in \bSigma^{0}_{1}\) and \( C \in \bPi^{0}_{1}\) such that \( \Phi ( U ) = \Phi ( C ) \in \boldsymbol{d}\).
\end{theorem}
\noindent
the proofs of the two results are different enough to warrant distinct statements.
Theorem~\ref{th:global} asserts that applying \( \Phi \) to very simple sets (like open or closed sets)
every conceivable complexity below \( \bPi^{0}_{3}\) can be attained.
This does not mean that every \( \Phi ( A ) \) is of the form  \( \Phi ( U ) \) or \( \Phi ( C ) \) 
with \( U \) open and \( C \) closed, since this would imply that every measurable 
set \( A \) is equivalent (up to a null set) to a closed or an open set,
which is far from being true.
Every measurable set is equivalent to a \( \bSigma^{0}_{2}\) (=\( \mathbf{F}_ \sigma \))
and to a \( \bPi^{0}_{2}\) (=\(  \mathbf{G}_ \delta  \)) and these are the least pointclasses 
that intersect every equivalence class in \( \MALG \).

\begin{theorem}\label{th:new}
\( \setofLR{\eq{A} \in \MALG}{ \eq{A} \cap \bDelta^{0}_{2} = \emptyset } \) is comeager in \( \MALG \).
\end{theorem}
In other words, for the generic \( A \) there is no set \( D \) which is simultaneously 
\( \mathbf{F}_ \sigma \) and \( \mathbf{G}_ \delta \), and such that \( \mu ( A \symdif D ) = 0 \).

The import of Theorem~\ref{th:global}
is that, arguing in \( \ZF + \DC \) alone, the family of sets
\[
\mathcal{S} \equalsdef \ran ( \Phi ) = \setofLR{ \Phi ( A )}{ A\in \MEAS}
\]
intersects every Wadge degree
inside \( \bPi^{0}_{3}\), yet \( \mathcal{S} \) has the size of the continuum, being in bijection
with the Polish space \( \MALG \).
This should be contrasted with the fact that under the Axiom of Determinacy (\( \AD\)) the size 
of \( \bPi^{0}_{3}\) is much larger than the continuum~\cite{Hjorth:1998sv,Andretta:2007ce}.
Families \( \mathcal{S} \) of size continuum intersecting every Wadge degree in \( \bGamma\)
can be constructed under \( \AD \)  for every Borel boldface pointclass \( \bGamma\),
as L.~Motto Ros pointed out to us;
however the case of \( \mathcal{S} = \ran ( \Phi )\) and \( \bGamma = \bPi^{0}_{3} \)
is the only nontrivial, natural example of this phenomenon we know of.

Theorem~\ref{th:global} implies that for every Wadge degree \( \boldsymbol{d} \subseteq \bPi^{0}_{3}\), the sets
\[
\mathscr{W}_{\boldsymbol{d}} = \setofLR{\eq{A} \in \MALG }{ \Phi ( A ) \in \boldsymbol{ d} }
\]
are nonempty, hence \( \setofLR{ \mathscr{W}_{\boldsymbol{d}}}{ \boldsymbol{d} \subseteq \bPi^{0}_{3}}\)
is a partition of \( \MALG \).
Since the length of the Wadge hierarchy of \( \bDelta^{0}_{3}\) sets is \( \omega _1^{ \omega _1}\),
this defines a canonical well-quasi-order \( \preceq\) on \( \MALG \) of length \( \omega _1^{ \omega _1} + 1 \),
which, by the Kunen-Martin theorem, cannot be \( \bSigma^{1}_{1}\).
Actually the complexity of \( \preceq\) is \( \bSigma^{1}_{2}\) and its equivalence classes 
\( \mathscr{W}_{\boldsymbol{d}}\) are provably \( \bDelta^{1}_{2}\)
(Section~\ref{subsec:wqoonMALG}), hence they have
the property of Baire.
Clearly all but countably many \( \mathscr{W}_{\boldsymbol{d}}\) must be meager ---
in fact all but one.
\begin{theorem}\label{th:Pi03comeagre}
Let \( \boldsymbol{d} = \bPi^{0}_{3} \setminus \bDelta^{0}_{3} \) be the Wadge degree of all 
sets which are complete \( \bPi^{0}_{3}\).
Then \( \mathscr{W}_{ \boldsymbol{d}} \) is comeager in \( \MALG \).
\end{theorem}

Although most of the \( \mathscr{W}_{\boldsymbol{d}} \) are meager, 
they are topologically dense:

\begin{theorem}\label{th:dense}
If \( \boldsymbol{d} \subseteq \bPi^{0}_{3}\) is a Wadge degree and 
\( \boldsymbol{d} \neq \setLR{ \emptyset} , \setLR{ \pre{\omega}{2} } \), then 
\( \mathscr{W} _{\boldsymbol{d}} \) is dense in the topological space \( \MALG \), i.e.
\[ 
 \FORALL{ A \in \MEAS} \FORALL{\varepsilon > 0 } \EXISTS{B \in \MEAS} 
\bigl (  \mu ( A \symdif B ) < \varepsilon \AND  \Phi ( B ) \in \boldsymbol{d} \bigr ) .
\]
In fact we can take \( B \) such that \( \Phi ( B ) = \Phi ( U ) = \Phi ( C ) \) with 
\( U \in \bSigma^{0}_{1}\) and \( C \in \bPi^{0}_{1}\).
\end{theorem}

If we look at \( \MALG \) as a Boolean algebra or, equivalently, as a forcing notion, 
there is a competing notion of ``dense set'': if \( \mathbb{B} \) is a Boolean algebra
then \( D \subseteq  \mathbb{B} \setminus \setLR{ 0_{\mathbb{B}}}\) is dense  iff 
\[
\FORALL{ b \in  \mathbb{B} \setminus \setLR{ 0_{\mathbb{B}}}}
\EXISTS{ d\in D } ( d \leq b ) .
\]
The set of all \( \eq{A} \in \MALG \) such that \( \Phi ( A ) \) has empty interior is dense 
in the sense of forcing, and from Theorem~\ref{th:emptyinteriorPi03} we shall obtain

\begin{theorem}\label{th:forcingdensePi03}
Let \( \boldsymbol{d} = \bPi^{0}_{3} \setminus \bDelta^{0}_{3} \) be the 
top Wadge degree in \( \bPi^{0}_{3}\), i.e., the Wadge degree of the complete 
\( \bPi^{0}_{3}\) sets.
Then \( \mathscr{W}_{ \boldsymbol{d}} \)
is dense in \( \MALG \) in the sense of forcing,
and it is the unique Wadge degree with this property.
\end{theorem}

Therefore when forcing with the measure algebra, it is enough to focus on 
conditions that are complete \( \bPi^{0}_{3}\) sets.

\subsection*{Plan of the paper.}
The paper is organized as follows.
In Sections~\ref{sec:notation} and~\ref{sec:easyfacts} we record some basic facts and the notations 
used throughout the paper, while Section~\ref{sec:Someexamples} is devoted to some examples and counterexamples.
The basics of the Wadge hierarchy of the Cantor space are developed in Section~\ref{sec:Wadgehierarchy},
where Theorem~\ref{th:dense} is deduced from Theorem~\ref{th:global}.
The main technical parts of the paper are Sections~\ref{sec:climbing} and~\ref{sec:Wadge-styleconstructions}
where measure-theoretic analogues of the Wadge constructions are developed, and 
Theorem~\ref{th:belowDelta03} is proved.
Finally, Theorems~\ref{th:Pi03complete}, \ref{th:emptyinteriorPi03}, \ref{th:new},
\ref{th:Pi03comeagre}, and~\ref{th:forcingdensePi03}
are proved in Section~\ref{sec:ClosedsetswithcomplicatedD}.

\subsection*{Acknowledgments}
Several people contributed with helpful discussion on the material of this paper.
In particular we wish to thank 
David Fremlin, Luca Motto Ros, and Asger T\"ornquist.
We owe a particular debt to Greg Hjorth --- to whom this paper is dedicated --- 
for most illuminating conversations at the early stages of this work.

\section{Notation}\label{sec:notation}
For the basics of descriptive set theory, measure theory, and the density topology, 
the reader is referred to~\cite{Kechris:1995kc,Oxtoby:1980kr,Wilczynski:2002rz}.

The length of \( x \in  \pre{ \leq\omega }{2}\equalsdef \pre{< \omega }{2} \cup \pre{\omega}{2} \)
is \( \dom ( x ) \) and it is usually denoted by \( \lh ( x ) \).
If \( s \in  \pre{< \omega }{2} \) and \( x \in   \pre{ \leq\omega }{2} \),
the concatenation of \( s \) with \( x \) is denoted with \( s \conc x \), or even \( s x \),
if there is no danger of confusion.
When \( x = \seqLR{i} \) and \( i \in \setLR{0, 1}\) we simply write \( s \conc i \),
while \( i^{(n)} \) denotes the sequence of length \(n \) and constant value \( i \).
Two sequences \( s , t \in \pre{< \omega }{2} \) are incompatible, in symbols 
\( s \perp t \), if \( s ( n ) \neq t ( n ) \) for some \( n < \lh ( s ) , \lh ( t ) \).
If \( A \subseteq \pre{\omega}{2} \) and \( s \in  \pre{< \omega }{2} \) then 
\[
\LOC{A}{s} = \setofLR{x \in \pre{\omega}{2} }{s \conc x \in A }
\]
is the \markdef{localization} of \( A \) at \( s \).
In particular~\eqref{eq:Lebesguedensitylimit} can be restated as
\[
\mathcal{D}_A ( x )  =  
\lim_{n \to \infty} \mu \left ( \LOC{A}{x \restriction n} \right ) = r .
\]
Similarly, if \( T \) is a tree on  \( 2 \), then 
\[
\LOC{T}{s} = \setofLR{u \in  \pre{< \omega }{2} }{s \conc u \in T}
\]
is the localization of \( T \) at \( s \).

The Lebesgue measure \( \mu \)  on the Cantor space is the unique Borel measure such that 
\( \mu ( \Nbhd_s ) = 2^{- \lh ( s ) }\), and for any measurable set \( A \), 
\[
 \mu ( \LOC{A}{s} ) = 
\frac{1}{2} ( \mu ( \LOC{A}{s \conc 0} ) + \mu ( \LOC{A}{s \conc 1 } ) ) 
\]
hence for every \( n \)
\[
 \mu ( A ) = 2^{-n} \bigl ( \sum_{\smash{s \in  \pre{n}{2} }} \mu ( \LOC{A}{s} ) \bigr ) .
\]
Therefore 
\begin{equation}\label{eq:series2}
\begin{split}
\mu ( A ) & = \sum_{n=0}^\infty 2^{-n - 1} \mu ( A )
\\
 & = \sum_{n=0}^\infty  2^{- 2n - 1}  \sum_{s \in  \pre{n}{2} } \mu ( \LOC{A}{s} )
 \\
 & =  \sum_{\smash{s \in  \pre{< \omega }{2} }} 2^{- 2 \lh ( s ) - 1} \mu ( \LOC{A}{s} )
\end{split} 
\end{equation}
and in particular, when \( A = \pre{\omega}{2} \)
\begin{equation}\label{eq:series1}
1= \sum_{s \in  \pre{< \omega }{2} } 2^{-  2\lh ( s ) - 1}  .
\end{equation}

Let  \( \mathcal{A} \subseteq  \pre{ < \omega }{2}\) be an antichain, i.e. a family of pairwise 
incompatible nodes. 
Then the \( \Nbhd_s \) (\( s \in  \pre{ < \omega }{2}\)) are pairwise disjoint hence  
\begin{equation}\label{eq:measureantichain}
\sum_{s \in \mathcal{A}} 2^{- \lh ( s ) } = \mu \bigl ( \bigcup_{s \in \mathcal{A}} \Nbhd_s \bigr ) \leq 1 . 
\end{equation}

\begin{lemma}\label{lem:approximation}
Let \( \mathcal{B} \) be a nonempty collection of measurable sets, closed under the operations 
\[
B \mapsto D \cup  t \conc B
\]
where \( D \in \bDelta^{0}_{1}\) and \( \Nbhd_t \cap D = \emptyset \). 
Then 
\[
 \FORALL{A\in \MEAS} \FORALL{ \varepsilon > 0} \EXISTS{B \in \mathcal{B}} 
 \bigl ( \mu ( A \symdif  B ) < \varepsilon \bigr ) .     
\]
In other words: \( \setofLR{ \eq{B}}{B \in \mathcal{B}} \) is topologically dense in \( \MALG \).
\end{lemma}

\begin{proof}
Let \( A \in \MEAS \), \( B \in \mathcal{B}\) and \( \varepsilon > 0 \) be given.
Fix a clopen set \( D \neq \pre{\omega}{2} \) such that \( \mu ( A \symdif D ) < \varepsilon / 2 \).
Let \( t \) be such that \( \Nbhd_t \cap D = \emptyset \) and \( 2^{- \lh ( t ) } < \varepsilon /2\).
Then \( D \cup  t \conc B \in \mathcal{B} \) by assumption, and 
\( \mu ( A \symdif ( D \cup  t  \conc B )  ) < \varepsilon \).
\end{proof}

The \markdef{interior} and \markdef{closure} of a set \( A \) are denoted by 
\( \Int  A \) and \( \Cl  A \), respectively, while the symbol \( \overline{A} \) 
is reserved for a different concept (see Section~\ref{subsec:Wadgesconstructions}).
The \markdef{frontier} of \( A \) is the set 
\( \Fr  A = \Cl  A \cap \Cl ( \neg A ) = \Cl A \setminus \Int  A \).

If \( \mu \) is a finite Borel measure on a second countable topological space \( X \), the 
\markdef{support} of \( \mu \) is the smallest co-null closed set, that is 
\[ 
X \setminus \bigcup \setofLR{U \subseteq X}{ U \text{ open and } \mu ( U ) = 0 } .
\]
This notion suggests the following definition.
If \( A \) is measurable, the \markdef{inner support} of \( A \) 
\[
\supt^- ( A ) = \bigcup \setofLR{U}{ \mu ( U ) = \mu ( U \cap A ) \AND U \text{ open}}
\]
is the largest open set \( V \) such that \( \mu ( V ) = \mu ( V \cap A ) \).
The \markdef{outer support} of \( A \) 
\begin{equation*}
\begin{split}
\supt^+ ( A ) & = \neg \supt^- ( \neg A )
\\
 & =  \bigcap \setofLR{C}{ \mu (A \setminus C ) = 0  \AND C \text{ closed}}
\end{split} 
\end{equation*}
is the smallest closed set \( C \) that contains \( A \) up to a null set.
It is immediate to check that \( \Int A \subseteq\supt^- ( A ) \) 
and \( \supt^+ ( A ) \subseteq \Cl  A  \),
but the inclusions can be strict as \( \supt^+ \)and \( \supt^-\)
are invariant up to null sets.

\section{Easy facts}\label{sec:easyfacts}
\subsection{A coding of \texorpdfstring{\( \bPi^{0}_{3}\)}{Pi03} sets}\label{subsec:codingofPi03}
A clopen \( D \subseteq \pre{\omega}{2} \) is completely determined 
by a finite tree \( T \) on \( \setLR{0,1}\) such that 
\( D = \bigcup \setofLR{ \Nbhd_t }{ t \text{ a terminal node of } T}\).
In order to have a unique such \( T \) we require that there is no \( t \)
such that both \( t \conc 0 \) and \( t \conc 1\) are  terminal nodes of \( T \).
Let \( \mathcal{T} \) be the collection of all such trees.
A clopen subset of \( \pre{\omega}{2} \times \omega \times \omega \times \omega \) 
--- where this space is endowed with the product topology, and \( \omega \) is taken to be discrete --- 
can be identified with a function \( ( k , m , n ) \mapsto T_{k , m , n} \in \mathcal{T} \).
By standard arguments, every such function can be coded as an element of the Cantor space, 
so let \( \mathcal{K} \subseteq \pre{\omega}{2} \) be the set of all such codes, and let
\[
\pi \colon \mathcal{K} \to \bDelta^{0}_{1} ( \pre{\omega}{2} \times \omega \times \omega \times \omega ) 
\]
be the decoding bijection.
The map 
\[ 
\psi \colon \bDelta^{0}_{1} ( \pre{\omega}{2} \times \omega \times \omega \times \omega )  
\to \bPi^{0}_{3} ( \pre{\omega}{2} ), 
\qquad D \mapsto \bigcap_n \bigcup_m \bigcap_k D_{k,m,n} ,
\]
where 
\[
D_{k,m,n} = \setofLR{x \in \pre{\omega}{2} }{ ( x , k , m , n ) \in D}
\]
is surjective, hence \( \pi \circ \psi \colon \mathcal{K} \onto \bPi^{0}_{3} ( \pre{\omega}{2} ) \) 
can be construed as a coding of the \( \bPi^{0}_{3} \) subsets of the Cantor space.

By the Lebesgue Density Theorem~\ref{th:Lebesguedensity},
for any measurable sets \( A , B \subseteq \pre{\omega}{2} \)
\[
A \equiv B \IFF \FORALL{s \in \pre{< \omega}{2} } \left ( \mu \left (\LOC{A}{s} \right ) = 
\mu \left (\LOC{B}{s} \right )  \right ),
\]
and for any \( m \in  \omega \) and \( r \in [ 0 ; 1 ) \) the set  
\[ 
\setofLR{x \in \pre{\omega}{2} }{ \mu ( \LOC{A}{x \restriction m} ) > r} 
\]
is clopen.
Therefore the set
\begin{equation}\label{eq:closedfordensity}
\tilde{A} = \setofLR{( x , m , n , k ) \in \pre{\omega}{2} \times  \omega \times \omega \times \omega }
{ m \geq n \implies \mu \left (\LOC{A}{x \restriction m} \right ) > 1 - 2^{-k - 1}} 
\end{equation}
is clopen.
(The reason for the extra coordinate \( n \) in the definition of \( \tilde{A} \) will be clear shortly.)
Moreover \( \tilde{A} \)  depends on the equivalence class \( \eq{A} \in \MALG \), rather than on the set \( A \), i.e.
\[ 
A \equiv B \implies \tilde{A} = \tilde{B} ,
\]
and the map \( \MALG \to \bDelta^{0}_{1} ( \pre{\omega}{2} \times \omega \times \omega \times \omega ) \),
\( \eq{A} \mapsto \tilde{A} \) is Borel --- 
in the sense that its composition with \( \pi^{-1} \) yields a Borel function 
\( \MALG \to \mathcal{K} \).
(In fact this map falls short of being continuous in that the preimage 
of an open set is a Boolean combination of open sets.)
Since
\begin{equation}\label{eq:folklore}
x\in \Phi (A) \IFF \forall k \exists  n \forall m \geq n
\left (\mu ( \LOC{A}{x \restriction m} ) > 1 - 2^{-k - 1} \right )  
\end{equation}
then
\begin{equation*}
\hat{\Phi} ( \eq{A} ) = \bigcap_k \bigcup_n \bigcap_m \tilde{A}_{m ,n , k} .
\end{equation*}

\begin{proposition}\label{prop:Pi03}
\begin{enumerate-(a)}
\item\label{prop:Pi03-a}
\( \Phi  ( A ) \in \bPi^{0}_{3} \) for all measurable \( A \), and
\item\label{prop:Pi03-b}
the map \( \hat{\Phi} \colon \MALG \to \bPi^{0}_{3} \) is Borel in the codes,~i.e.\ there is
a Borel map \( \mathcal{F} \colon \MALG \to \mathcal{K}\)
such that \( \mathcal{F} ( \eq{A} ) \) is a code for \( \hat{\Phi} ( \eq{A} ) \).
\end{enumerate-(a)}
\end{proposition}

\begin{proof}
Part~\ref{prop:Pi03-a} is folklore and it follows from~\eqref{eq:folklore} 
when the ambient space is \( \pre{\omega}{2} \) --- see~\cite{Wilczynski:2002rz} for a proof that
\( \Phi ( A ) \) is \( \bPi^{0}_{3} \) when the ambient space is \( [ 0 ; 1 ] \) and \( \mu \) is 
the Lebesgue measure.

For~\ref{prop:Pi03-b} just take \( \mathcal{F} ( \eq{A} ) = \tilde{A} \). 
\end{proof}

\subsection{Properties if \texorpdfstring{\( \Phi \)}{Phi}}\label{subsec:propertiesifPhi}
Let us list some easy facts about the density map \( \Phi  \).
\begin{subequations}
\begin{align}
& A \subseteq B \IMPLIES \Phi  ( A ) \subseteq \Phi ( B ) , \label{eq:densitymonotone}
\\
& \Phi ( \Phi ( A ) ) = \Phi ( A ) , \label{eq:idempotence}
\\
& \Phi  (  A  \cap B ) = \Phi  ( A ) \cap \Phi ( B ), \label{eq:densityintersection}
\\
& \Phi ( \emptyset ) = \emptyset \text{ and } \Phi ( \pre{\omega}{2} ) = \pre{\omega}{2} ,\label{eq:Phi(empty)}
\\
&\Phi  ( \neg A ) \subseteq \neg \Phi  ( A ) , \label{eq:densitycomplement}
\\
&\Phi  ( A \cup B ) \supseteq \Phi  ( A ) \cup \Phi ( B )  \text{, and more generally,}\label{eq:densityunion}
\\
& \Phi ( \bigcup_{i \in I} A_i ) \supseteq \bigcup_{i \in I} \Phi ( A_i ) . \label{eq:densityunion2}
\end{align}
\end{subequations}
The inclusions in~\eqref{eq:densitycomplement} and~\eqref{eq:densityunion} 
cannot be replaced by equalities, as can be seen by constructing appropriate 
counterexamples or by the following metamathematical argument.
If \( \Phi  ( \neg A ) = \neg \Phi  ( A ) \) for all \( A \) or, equivalently, 
\(  \Phi ( A \cup B ) = \Phi ( A ) \cup \Phi ( B )\) for all \( A , B \),
then \( \Phi \colon \MEAS \to \Bor \) would be a homomorphism of 
Boolean algebras such that \( \Phi ( A ) \equiv A \).
Such homomorphisms are called Borel liftings, and by work of Shelah~\cite{Shelah:1983fk} 
their existence is independent of \( \ZFC \).

By~\eqref{eq:densityintersection},~\eqref{eq:Phi(empty)}, and \eqref{eq:densityunion2} the family
\[
\mathcal{T} = \setofLR{A \in \MEAS}{ A \subseteq \Phi ( A ) }
\] 
is a topology on the Cantor set,
and it is called the \markdef{density topology}.
If \( A \) is open and \( x \in A \), then \( \Nbhd_{x \restriction n} \subseteq A \)
for all large enough \( n \), so 
\begin{equation}\label{eq:densityopen}
 A \in \bSigma^{0}_{1} \IMPLIES A \subseteq \Phi  ( A ) ,
\end{equation}
hence \( \mathcal{T} \) refines the standard topology.
Since \( \Phi ( A \setminus N ) = \Phi ( A ) \supseteq A \supseteq A \setminus N \), 
for any null set \( N \) and any open set \( A \), it follows that \( \mathcal{T} \) 
is strictly finer than the standard topology.
Note that the inclusion in~\eqref{eq:densityopen} can be strict by Example~\ref{xmp:densityopen} below.
Here is a list of the properties of \( \mathcal{T} \) --- for proofs see~\cite{Wilczynski:2002rz}
and the reference therein:
\begin{itemizenew}
\item
 For any \( A \subseteq \pre{\omega}{2} \)
 \[
\Int_ \mathcal{T}  A  = A \cap \Phi ( B )
 \]
where \( B \) is any measurable kernel of \( A \), that is: 
any measurable set \( B \subseteq A \) such that \( \mu ( B ) = \mu _* ( A ) \), with \( \mu _* \) the inner measure. 
\item
\( \mathcal{T} \) is neither first countable (hence not second countable) nor separable, but it is Baire.
\item
\( A \) is null if and only if it is meager in the topology \( \mathcal{T} \), if and only if 
it is nowhere dense in the topology \( \mathcal{T} \), if and only if it is closed
and discrete in the topology \( \mathcal{T} \).
\item
\( A = \Phi ( A ) \) if and only if \( A \) is a regular open set of the topology
\( \mathcal{T} \), that is 
\( A = \Int_\mathcal{T} \Cl_\mathcal{T}  A  \).
\end{itemizenew}
In view of this last property, a measurable set \( A \subseteq \pre{\omega}{2} \) such that \( A = \Phi ( A ) \) 
is called \markdef{\( \mathcal{T} \)-regular}.

Clearly
\begin{equation}\label{eq:densityclopen}
A \in \bDelta^{0}_{1} \IMPLIES  A  = \Phi ( A )
\end{equation}
but the converse implication does not hold ---
as we shall see below, there are sets \( A \) such that \( \Phi ( A ) = A \) of arbitrarily high
complexity in the pointclass \( \bPi^{0}_{3}\). 

By~\eqref{eq:densitycomplement} and~\eqref{eq:densityopen}
\begin{equation}\label{eq:densityclosed}
 A \in \bPi^{0}_{1} \IMPLIES \Phi  ( A ) \subseteq A.
\end{equation}
(Again by Example~\ref{xmp:densityopen} the inclusion can be strict.)
By monotonicity
\begin{equation}\label{eq:densityinteriorclosure0}
\Int  A  \subseteq \Phi ( A ) \subseteq \Cl  A   .
\end{equation}
Thus if \( A = \Phi  ( C ) \) for some closed \( C \), then
by~\eqref{eq:densityclosed} and monotonicity we may assume that \( C = \Cl  A \),
hence 
\begin{equation*}\label{eq:Phi(closed)}
A\in \ran \left ( \Phi  \restriction \bPi^{0}_{1} \right ) \IFF A = \Phi  ( \Cl A ) .
\end{equation*}
Similarly
\begin{equation*}\label{eq:Phi(open)}
A\in \ran \left ( \Phi  \restriction \bSigma^{0}_{1} \right ) \IFF A = \Phi  ( \Int  A ) .
\end{equation*}

If \( A \equiv B \) then \( \supt^- ( A ) = \supt^- ( B ) \) and \( \supt^+ ( A ) = \supt^+ ( B ) \),
in particular the inner/outer support of \( A \) is the same as the inner/outer 
support of \( \Phi  ( A ) \), but in general \( A \not\equiv \supt^\pm ( A )\).
In fact the sets \( \supt^\pm ( A ) \) are not complete invariants for the
relation of measure equivalence --- in other words, the map 
\( \MALG \to \bSigma^{0}_{1} \times  \bPi^{0}_{1} \), 
\( \eq{A} \mapsto ( \supt^- ( A ) , \supt^+ ( A ) ) \) is not injective
(Proposition~\ref{prop:supports}).

Using the preceding results, \eqref{eq:densityinteriorclosure0} can be refined to
\begin{equation} \label{eq:densityinteriorclosure}
\Int  A  \subseteq \supt^- ( A ) \subseteq \Phi  ( A ) \subseteq \supt^+ ( A ) \subseteq \Cl A ,
\end{equation}
hence
\begin{equation*}
 \mu  \left (\Fr  A \right ) = 0 
\IFF  \Int A \equiv \supt^- ( A ) \equiv
 A \equiv \supt^+ ( A ) \equiv \Cl A 
\end{equation*}
and thus 
\begin{equation*}
A = \Phi ( A ) \AND  \mu ( \Fr A ) = 0 \IFF  A = \Phi  ( \Cl A ) = \Phi ( \Int  A ) .
\end{equation*}
Therefore a \( \mathcal{T} \)-regular set is the image via \( \Phi \)  of an open 
and of a closed set if and only if its frontier is null,  i.e.
\begin{equation}\label{eq:D(closed)}
 A = \Phi ( A ) \AND  \mu ( \Fr A ) = 0 \IFF  A \in \ran (\Phi \restriction \bSigma^{0}_{1} ) 
\cap \ran ( \Phi \restriction \bPi^{0}_{1} ) .
\end{equation}

If \( A \equiv B \) then \( \Phi ( A ) = \Phi ( B ) \subseteq \Cl  B  \) hence 
\( \Cl \Phi ( A )  \subseteq \Cl  B  \). 
Therefore \( \Phi  ( A ) \) is a set \( X \equiv A \) such that \( \Cl X  \) is minimal.

\begin{lemma}\label{lem:closureofDissupt}
\( \Cl \Phi  ( A ) = \supt^+ ( A ) \) and 
\( \Int  \Phi  ( A ) = \supt^- ( A ) \).
\end{lemma}

\begin{proof}
As \( \supt^+ ( A ) \) is closed, it is enough to show that 
\( \supt^+ ( A ) \subseteq \Cl \Phi  ( A )  \).
Let \( x \in \supt^+ ( A ) \): then \( \mu ( U \cap A ) > 0 \) for every 
open neighborhood \( U \) containing \( x \), and since 
\( \mu \left (A \symdif \Phi  ( A ) \right ) = 0 \) then 
\( \mu \left ( U \cap \Phi  ( A ) \right ) > 0 \).
Therefore \( U \cap \Phi ( A ) \neq \emptyset \) 
hence \( x \in \Cl \Phi  ( A ) \).
This proves the first equality.

For the second, as \( \supt^- ( A ) \) is open, it is enough to show that 
\(\Int \Phi  ( A )  \subseteq \supt^- ( A ) \).
Let \( x \in V \subseteq \Phi  ( A ) \) with \( V \) open: as
\( V \cap A\equiv V \cap \Phi  ( A ) = V \)
then \( \mu ( V \cap A) = \mu ( V ) \), hence \( V \subseteq \supt^- ( A ) \).
\end{proof}

If \( C \) is closed and \( \mathcal{T} \)-regular, then \( \neg C \) is open
hence \( \Phi  ( \neg C  ) = \neg C \)
by~\eqref{eq:densityopen}, \eqref{eq:densityintersection} and~\eqref{eq:Phi(empty)}.
Therefore
\begin{equation}\label{eq:closedT-regular=>complementT-regular}
C \in \bPi^{0}_{1} \text{ and \( \mathcal{T} \)-regular } \IMPLIES \neg  C 
 \text{ is \( \mathcal{T} \)-regular.}
\end{equation}
Example~\ref{xmp:opennotdualistic} shows that~\eqref{eq:closedT-regular=>complementT-regular} fails if 
\( \bPi^{0}_{1} \) is replaced by \( \bSigma^{0}_{1} \).

\begin{definition}\label{def:D(A)}
For \( A\) a measurable set of positive measure, let 
\[ 
\densitytree ( A ) = \setofLR{s \in  \pre{ < \omega }{2} }{ \mu ( \LOC{A}{s} ) > 0 } .
\]
\end{definition}

Then \( \densitytree ( A ) \) is a pruned tree, and by the Lebesgue Density Theorem
it has no isolated branches.
Thus
\[
x \in \Phi ( A ) \IFF x \in \body{ \densitytree ( A ) } \AND \mathcal{D}_A ( x ) = 1 .
\]

\begin{proposition}
\( \body{\densitytree ( A ) } = \supt^+ ( A ) \).
\end{proposition}

\begin{proof}
Clearly \( A \setminus \body{\densitytree ( A )} \) is null.
If \( \mu ( A \setminus \body{ T } ) = 0 \) for some pruned tree \( T \), then 
\[
s \notin T \implies \mu ( A \cap \Nbhd_s ) = 0 \implies \mu ( \LOC{A}{s} ) = 0 \implies s \notin \densitytree ( A ) 
\]
that is: \( \densitytree ( A ) \subseteq T \).
\end{proof}

By~\eqref{eq:densityinteriorclosure}, 
\begin{equation}\label{eq:Phidensitytree}
 C \in \bPi^{0}_{1} \IMPLIES  \Phi ( C ) = \Phi ( \densitytree ( C ) ) .
\end{equation}

\begin{proposition}\label{prop:climb}
For every \( s \in \densitytree ( A ) \), if  \( \mu ( \LOC{A}{s} ) < r < 1 \) then
\[
\exists  t \supset s \left ( \mu ( \LOC{A}{t} )\geq r \AND 
\forall u \left (s \subseteq u \subset t \implies \mu ( \LOC{A}{u} ) \geq \mu ( \LOC{A}{s} ) \right ) \right ) .
\]
\end{proposition}

\begin{proof}
Replacing \( \LOC{A}{s} \) with \( A \) we may assume that \( s = \emptyset \).
Let 
\[
\mathcal{B} = \setofLR{t \in \densitytree ( A ) }{ \mu ( \LOC{A}{t}) \geq r \AND 
\forall t' \subset t \left (\mu (\LOC{A}{t'} < r ) \right )} .
\]
\begin{claim}
\( \mathcal{B} \) is a maximal antichain in \( \densitytree ( A )\).
\end{claim}

\begin{proof}
It is clear that \( \mathcal{B} \) is an antichain.
Suppose \( s \in \densitytree ( A ) \): by the Lebesgue Density Theorem, there is an \( x \in \Phi ( A ) \cap \Nbhd_s\).
Let \( n \) be least such that \( \mu ( \LOC{A}{x \restriction n} ) \geq r \).
Then \( x \restriction n \in \mathcal{B}\) and \( x \restriction n \) is compatible with \( s \).
\end{proof}

Towards a contradiction, suppose that for every \( t \in \mathcal{B} \) there is a minimal 
\( u_t  \subset t \) such that \( \mu ( \LOC{A}{u_t} ) < \mu ( A ) \), and let \( \mathcal{A} \)
be the set of all these \( u_t \).
It is easy to check that \( \mathcal{A} \) is also a maximal antichain in \( \densitytree ( A ) \).

\begin{claim}
\( \mu ( A ) = \sum_{ u\in \mathcal{A}} 2^{- \lh ( u ) } \mu ( \LOC{A}{u} ) \).
\end{claim}

\begin{proof}
It is enough to show that \( N = A \setminus \bigcup_{u \in \mathcal{A}} u \conc \LOC{A}{u} \) is null.
Otherwise, let \( x \in N \cap \Phi ( A ) \) and let \( n \) be least such that 
\( \mu ( \LOC{A}{ x \restriction n} ) \geq r \).
Therefore \( x \restriction n \in \mathcal{B} \) hence 
\( \exists  u \in \mathcal{A} \left ( u \subseteq x \restriction n \right ) \): a contradiction.
\end{proof}
Therefore by~\eqref{eq:measureantichain}
\[
\mu ( A ) < \sum_{ u \in \mathcal{A}} 2^{-\lh ( u )} \mu ( A ) \leq  \mu ( A ) ,
\]
a contradiction.
\end{proof}

\subsection{Dualistic sets}\label{subsec:dualisticsets}
By~\eqref{eq:densityclopen}, \( \Phi  ( \neg A ) = \neg \Phi  ( A ) \) 
for \( A \) clopen, or more generally if \( A \) has the following property.

\begin{definition}\label{def:manic}
A set is \textbf{dualistic} if it belongs to the family
\[
 \mathcal{M} =  \setofLR{A \in \MEAS}{ \FORALL{x \in \pre{\omega}{2} }
\left (  \mathcal{D}_A ( x ) \text{ exists and it is equal to either \( 0 \) or } 1 \right) } . 
\]
\end{definition}
Sets in \( \mathcal{M} \) have a very black-or-white vision of the points of the
space, so they should perhaps be called Manich\ae an
(hence the \( \mathcal{M} \)).
If  \( x \) witnesses that \( A \notin \mathcal{M} \),
then such \( x \) belongs to the complement of  \( \Phi  ( A ) \cup \Phi  ( \neg A ) \), 
so the inclusion in~\eqref{eq:densityunion} will be proper.

The family \( \mathcal{M} \) is an algebra containing \( \bDelta^{0}_{1}\), and it is
the largest algebra \( \mathcal{N} \subseteq \MEAS \) such that 
\( \Phi \restriction \mathcal{N} \colon \mathcal{N} \to \Bor \) is a lifting, 
i.e. a selector that is a homomorphism.
It does not contain every open or closed set and therefore it is not a \( \sigma \)-algebra 
(Example~\ref{xmp:opennotdualistic}), but it contains sets of arbitrarily high complexity.
In fact
\begin{equation}\label{eq:dualisticclosedunderequivalence}
 A \in \mathcal{M} \AND B \equiv A \IMPLIES B \in \mathcal{M}
\end{equation}
hence \( \mathcal{M} \supset \NULL\), and since the ideal \( \NULL \) 
contains sets of arbitrary complexity, the claim is proved.
On the other hand, if \( A \in \mathcal{M} \) then
\[
x \in \Phi  ( A ) \IFF \exists{n} \FORALL{m \geq n}\left ( \mu 
\left ( \LOC{A}{x \restriction m} \right ) > 1/2 \right )  
\]
hence using that \( \tilde{A}\) is clopen (see~\eqref{eq:closedfordensity}), \( \Phi  ( A ) \) is easily seen to 
be \( \bSigma^{0}_{2} \).
Since \( \neg \Phi  ( A ) = \Phi  ( \neg A ) \) and \( \neg A \in \mathcal{M} \), it follows that 
\[
A \in \mathcal{M} \IMPLIES \Phi  ( A ) \in \bDelta^{0}_{2} .
\]
By Example~\ref{xmp:densityopen} not every \( \bSigma^{0}_{1} \) (and therefore: 
not every \( \bDelta^{0}_{2} \)) set is \( \mathcal{T} \)-regular, hence 
\( \ran \left ( \Phi  \restriction \mathcal{M} \right ) \neq \bDelta^{0}_{2} \), and
by Theorem~\ref{th:T-regularDelta02}, \( \ran  \left ( \Phi  \restriction \mathcal{M} \right )  \)
intersects every Wadge degree in \( \bDelta^{0}_{2} \).
 
By~\eqref{eq:dualisticclosedunderequivalence}
\(A \in \mathcal{M} \IFF \Phi  ( A ) \in \mathcal{M}\),
and it is easy to check that
\begin{equation}\label{eq:T-regularcomplement=>dualistic}
A = \Phi  ( A ) \AND \neg A = \Phi  ( \neg A )  \IMPLIES A , \neg A\in \mathcal{M}.
\end{equation}
The notions of dualistic and \( \mathcal{T} \)-regular set are independent:
not every \( \mathcal{T} \)-regular set is in \( \mathcal{M} \), as there are sets \( X \) such that \( \Phi  ( X ) \) 
is \( \bPi^{0}_{3} \) complete (Theorem~\ref{th:Pi03complete})
 and not every set in \( \mathcal{M} \) is \( \mathcal{T} \)-regular
--- see~\eqref{eq:dualisticclosedunderequivalence} or Example~\ref{xmp:densityopen}.

A set is \( \mathcal{T} \)-clopen iff \( A = \Phi ( A ) \) and \( \neg A = \Phi ( \neg A ) \), 
so both \( A \) and \( \neg A \) are dualistic by~\eqref{eq:T-regularcomplement=>dualistic}.
Therefore 
\begin{equation*}\label{eq:T-regulardualisticPhiclopen}
A \text{ is \( \mathcal{T} \)-regular and dualistic} \IFF A \text{ is \( \mathcal{T} \)-clopen.} 
\end{equation*}
If \( A \) and \( B \) are clopen or, more generally, dualistic,
the inclusion in~\eqref{eq:densityunion} can be replaced with equality.
But if  \( A \) is such that \( \Phi  ( \neg A ) \subset \neg \Phi  ( A ) \) then 
\( \pre{\omega}{2} = \Phi  ( A \cup \neg A ) \supset \Phi  ( A ) \cup \Phi  ( \neg A ) \). 

\subsection{Examples}\label{sec:Someexamples}
Given a function \( f \colon \omega  \to \omega \setminus\setLR{0} \) consider the sets
\begin{align*}
U_f & = \bigcup \setofLR{\Nbhd_s}{ \exists {n \in \omega } \left ( s = 0^{( n )}  1^{( f ( n ) )} \right ) }
\\
V_f &  = \bigcup \setofLR{ \Nbhd_s}{ \FORALL{n \in \omega } \left ( s \perp 0^{( n )}  1^{( f ( n ) )}\right ) }
\\
& = \bigcup \setofLR{ \Nbhd_s}{ \exists{n} \exists{m} \left ( 0 < m < f ( n ) \AND
	s = 0^{(n)} 1^{(m)}0 \right ) }.
\end{align*}
By construction \( U_f \) and \( V_f \) are disjoint open sets, and together with \( \setLR{0^{(\infty)}} \) 
they partition \( \pre{\omega}{2} \).
Also
\begin{align*}
\Cl U_f  & = U_f \cup \setLR{ 0^{( \infty )}}
\\
\intertext{and}
U_f  & = \neg \left ( V_f \cup \setLR{0^{( \infty )}} \right ) 
\in \bSigma^{0}_{1} \setminus \bPi^{0}_{1} .
\end{align*}
Every point in \( U_f \) or in \( V_f \) has density \( 1\) in the respective set, so \( 0^{( \infty )} \)
is the only point where density must be established.

\begin{example}\label{xmp:densityopen}Dualistic sets which are open or closed but not \( \mathcal{T} \)-regular.

Suppose \( \FORALLS{\infty}{n}  f ( n ) = 1 \).
Then \( V_f \) is clopen.
Since \( 0^{( \infty )} \) has density \( 1 \) in \( U_f \),  then \( \Phi  ( U_f ) = U_f \cup \setLR{0^{( \infty )}}\).

Therefore \( U= U_f \) and its complement \( F = V_f \cup \setLR{ 0^{( \infty )}} \)
are examples of open (resp. closed) sets which are dualistic, but not \( \mathcal{T} \)-regular, since 
\( \Phi  ( U ) \supset U \) and \( \Phi  ( F ) \subset F \).
\end{example}

\begin{example}\label{xmp:opennotdualistic}
An open \( \mathcal{T} \)-regular set which is not dualistic.

Suppose that
\[ 
\EXISTSS{\infty}{n}  ( f ( n ) = 1  ) \AND 
\EXISTSS{\infty}{n}  ( f ( n ) \neq 1  )  .
\]
If \( f ( n ) = 1 \) then 
\begin{equation*}
\begin{split}
 \mu \bigl (\LOC{( U_f )}{0^{( n )}} \bigr ) & =  \frac{1}{2} \mu \bigl (\LOC{( U_f )}{0^{( n + 1)}} \bigr ) 
	+  \frac{1}{2} \mu \bigl (\LOC{( U_f )}{0^{( n )}1} \bigr )
\\
 & >  \frac{1}{2} \mu \bigl (\LOC{( U_f )}{0^{( n )}1} \bigr )
 \\
 & = \frac{1}{2},
\end{split} 
\end{equation*}
and if \( f ( n ) > 1 \) then 
\begin{equation*}
 \mu \bigl (\LOC{( U_f )}{0^{( n )}} \bigr ) 
 = \frac{1}{2}  \mu \bigl (\LOC{( U_f )}{0^{( n + 1)}} \bigr ) 
	+  \frac{1}{4} \mu \bigl (\LOC{( U_f )}{0^{( n )}11} \bigr )
  \leq \frac{1}{2} + \frac{1}{4} .
\end{equation*}
Therefore the density of \( 0^{(\infty )} \), if it exists, is neither \( 0 \) nor \( 1 \).
This implies that \( U_f \) is \( \mathcal{T} \)-regular, and that neither \( U_f \) nor its complement 
\( \pre{\omega}{2} \setminus U_f \) are dualistic.

Note that \( \pre{\omega}{2} \setminus U_f \) is not \( \mathcal{T} \)-regular
by~\eqref{eq:closedT-regular=>complementT-regular}.
\end{example}

\begin{example}\label{xmp:closedT-regular}
Dualistic \( \mathcal{T} \)-regular sets  which are open or closed.

Suppose \( \lim_{n \to \infty} f ( n  ) = + \infty \).
Fix \( k \) and choose \( N \) such that \( f ( n ) \geq k \) for all \( n \geq N \),
so that 
\[
\mu \left ( \LOC{(U_f)}{0^{(n)}} \right ) \leq 2^{- k} \sum_{i= 0}^\infty \frac{1}{2^i}  = 2^{-k + 1}
\]
which goes to \( 0 \) as \( k \to \infty \).
Therefore \( 0^{(\infty )} \) has density \( 0 \) in \( U_f \) and
\[ 
F = \pre{\omega}{2} \setminus U_f = V_f \cup \setLR{0^{( \infty )}}
\]
is an example of a \( \mathcal{T} \)-regular, nonempty closed set, 
hence by~\eqref{eq:closedT-regular=>complementT-regular}
and~\eqref{eq:T-regularcomplement=>dualistic} \( F \) and its complement \( U_f \) are in \( \mathcal{M} \).
\end{example}

\section{The Wadge hierarchy on the Cantor space}\label{sec:Wadgehierarchy}
If \( X \) and \( Y \) are topological spaces and \( A \subseteq X \) and \( B \subseteq Y \)
we write 
\[
( X , A ) \leqW ( Y , B ) 
\] 
just in case \( A = f^{-1} ( B ) \) for some continuous 
\( f \colon X \to Y \).
If \( X \) and \( Y \) are metric spaces and the function \( f \) is Lipschitz,
that is \( d_Y ( f ( x_1 ) , f ( x_2 ) ) \leq d_X ( x_1 , x_2  )\), 
we write 
\[ 
( X , A ) \leql ( Y , B ) .
\]
Whenever \( X = Y \) and the ambient space \( X \) is understood from the context,
we simply write \( A \leqW B \) or \( A \leql B \),
and the relations \( {\leqW} \) and \( \leql\) are pre-orders on \( \Pow ( X ) \) with remarkable properties, 
at least when \( X \) is Polish and zero-dimensional.
W.~Wadge was the first to conduct a systematic analysis in~\cite{Wadge:1983sp} 
of these preorders on the Baire space \( \pre{\omega}{\omega} \), whence
\( \leqW \) and \( \leql \) became known as \markdef{Wadge reducibility} and 
\markdef{Lipschitz reducibility}, respectively. 
Their induced equivalence relations are defined by
\begin{align*}
A \Wequiv B & \IFF A \leqW B \AND B \leqW A 
\\
A \lequiv B & \IFF A \leql B \AND B \leql A 
\end{align*}
and their equivalence classes are called, respectively, \markdef{Wadge degrees} and 
\markdef{Lipschitz degrees}.
The Wadge degree of \( A \) is denoted by \( \Wdeg{A} \).
The structure of the Wadge and Lipschitz degrees of the Borel subsets of \( \pre{\omega}{\omega} \)
has been completely analyzed in~\cite{Wadge:1983sp} and there are several accounts
of the basic facts about Wadge degrees in \( \pre{\omega}{\omega} \), see
e.g.~\cite{Andretta:2003th,Andretta:2006xv,Louveau:1983dk,Louveau:1988mi,Van-Wesep:1978fv}.
Most of the results and techniques apply to the Cantor space as well, but 
other parts of the theory require some reworking, so for the reader's benefit we will 
briefly summarize the main facts in this area.

Let us assume from now on that, unless otherwise stated, all sets in sight are Borel subsets of the Cantor space.
Thus \( A \leqW B \) and \( A \leql B \) mean that 
\( ( \pre{\omega}{2} , A) \leqW ( \pre{\omega}{2} , B ) \) and
\( ( \pre{\omega}{2} , A) \leql ( \pre{\omega}{2} , B ) \), respectively.
Since the subsets of the Cantor space are also subsets of the Baire space, 
we might want to study Wadge or Lipschitz reducibility within the ambient space \( \pre{\omega}{\omega} \),
and in this case we will write \( A \leqW^* B \) and \( A \leql^* B \) for
\( ( \pre{\omega}{\omega} , A) \leqW ( \pre{\omega}{\omega} , B ) \) and
\( ( \pre{\omega}{\omega}  , A) \leql ( \pre{\omega}{ \omega } , B ) \).

A set \( A \) is \markdef{self-dual} if \( A\leqW \neg A\) or, equivalently, if \( A \Wequiv \neg A\),
otherwise it is said to be \markdef{non-self-dual}. 
These notions are invariant under \( \Wequiv \) so we will speak of self-dual/non-self-dual degrees.
The \markdef{Lipschitz game} \( \GL ( A , B ) \) is the zero-sum, perfect information 
game of length \( \omega \) on \( \setLR{0 , 1} \) 
\[
\begin{tikzpicture}
\node (II) at (0,0) [anchor=base] { \( \II \)};
\node (I) at (0,1) [anchor=base] {\( \I \)};
\node (x_0) at (1,1) [anchor=base] {\( a_0 \)};
\node (x_1) at (2,0) [anchor=base] {\(b_0 \)};
\node (x_2) at (3,1) [anchor=base] {\( a_1\)} ;
\node (x_3) at (4,0) [anchor=base] {\( b_1 \)};
\node at (5.5,1) [anchor=base] {\( \cdots \)};
\node at (5.5,0) [anchor=base] {\( \cdots \)};
\node (dots) at (5.5,0) [anchor=base] {\( \cdots \)};
\node (A) at (x_3.south -| II.east) {}; 
\node (B) at (I.north -| II.south east) {}; 
\node (C) at ($0.5*(A) +0.5 *(B)$){}; 
\draw (II.south east) -- (I.north -| II.south east) 
(II.north west |- C.center) -- (dots.east |- C.center); 
\node at (C-|II.south west) [anchor=east]{\( \GL (A , B ) \)};
\begin{pgfonlayer}{background} 
\fill[gray!30,rounded corners] 
(II.south west) rectangle  (dots.east |-  I.north); 
\end{pgfonlayer} 
\end{tikzpicture} 
\]
where \( \II \) wins iff 
\[ 
(a_n )_n  \in A \IFF ( b_n )_n  \in B .
\]
Then \( \II \) has a winning strategy in \( \GL ( A , B ) \) iff \( A\leql  B \).
The \markdef{Wadge game} \( \GW ( A , B ) \) is similar to \( \GL ( A , B ) \) but \( \II \) has the option of
passing at any round, with the proviso that he must play infinitely many times.
Then \( \II \) has a winning strategy for \( \GW ( A , B ) \) iff \( A \leqW B \).

The moves of the games \( \GL \) and \( \GW \) are in \( \setLR{0,1}\) since
we are dealing with subsets of the Cantor space \( \pre{\omega}{2} \).
In most papers on the Wadge hierarchy the underlying space is the Baire space 
\( \pre{\omega}{\omega} \) so the moves are in \( \omega \), and
here we will denote this variant by \( \GL^* \) and \( \GW^* \):
the definition is as before and for \( A , B \subseteq \pre{\omega}{\omega} \)
\begin{align*}
( \pre{\omega}{\omega} , A ) \leql ( \pre{\omega}{\omega} , B ) & 
\iff \II \text{ has a winning strategy in } \GL^* ( A , B )
\\
( \pre{\omega}{\omega} , A ) \leqW ( \pre{\omega}{\omega} , B ) & 
\iff \II \text{ has a winning strategy in } \GW^* ( A , B )
\end{align*}

By results of Wadge and Martin, for all Borel sets \( A , B \subseteq \pre{\omega}{2} \)
Wadge's Lemma holds, that is
\begin{equation*}\label{eq:Wadge'sLemma}
A\leql B \OR \neg B \leql A,
\end{equation*}
and the relation \( \leql \) is well-founded on Borel sets.
Analogous results hold for \( \leqW \) as well.

The Wadge rank \( \Wrank{A} \) of a Borel set \( A \) is its height
in the pre-order \( \leqW \) --- for technical reasons we start counting from \( 1 \) rather than \( 0 \).
At the bottom of the hierarchy we have two non-self-dual degrees,
namely \( \Wdeg{\emptyset} = \setLR{\emptyset} \) and 
\( \Wdeg{ \pre{\omega}{2} } = \setLR{ \pre{\omega}{2} } \), and the self-dual degrees 
and non-self-dual pairs alternate, and since the Cantor space is compact,
there is a non-self-dual pair at all limit levels:
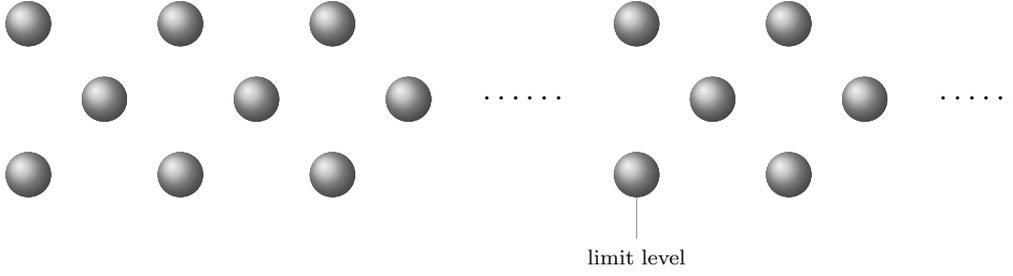
\[
\begin{tikzpicture}[degree/.style={circle,ball color=gray!60,inner sep=6pt},]
 \foreach \x in {0,2,4,10} \foreach \y in {0,2}
\node at (\x,\y) [degree] {};
 \foreach \z in {1,3,5,9,11}
\node at (\z,1) [degree] {};
\node at (6.5,1) {\( \cdots \cdots \)};
\node at (12.5,1) {\( \cdots \cdots \)};
\node at (8,2) [degree] {};
\node (A) at (8,0) [degree,pin=270:{\tiny limit level}]{};
\end{tikzpicture}
\]
This should be contrasted with the case of the Wadge hierarchy in the Baire space,
where self-dual degrees occur at limit levels of countable cofinality while 
non-self-dual pairs occur at all other limit levels.
Let us briefly justify the diagram above.
\begin{itemizenew}
\item 
If \( A\) is non-self-dual then
\begin{equation}\label{eq:oplus}
A\oplus \neg A = 0\conc A \cup 1 \conc \neg A 
\end{equation}
is a self-dual set immediately above \( A \).
\item
If \( A \) is self-dual, then 
\[
A^\triangledown = \bigcup_{n} 0^{(n)}\conc 1 \conc A 
\qquad\text{and}\qquad 
A^\circ = A^\triangledown \cup \setLR{0^{(\infty )}}
\]
are a non-self-dual pair immediately above \( A \).
\item
The tree \( \Wtree ( A ) = \setofLR{s \in  \pre{< \omega }{2} }{A \leqW \LOC{A}{s} } \)
detects the self-duality of \( A \) in the following sense. 
If \( s \) is a terminal node of \( \Wtree ( A ) \) then 
\[ 
\LOC{A}{s \conc 0 } , \LOC{A}{s \conc 1} \leW  \LOC{A}{s \conc 0 } 
\oplus  \LOC{A}{s \conc 1} = \LOC{A}{s} \Wequiv A 
\]
and by Wadge's Lemma either \( \LOC{A}{s \conc 0 } \leqW \LOC{A}{s \conc 1 } \)
or \( \neg \LOC{A}{s \conc 1 } \leqW \LOC{A}{s \conc 0 }\): the former would imply
\( \LOC{A}{s} \Wequiv\LOC{A}{s \conc 1} \) which is impossible, so 
\( \neg \LOC{A}{s \conc 1 } \leqW \LOC{A}{s \conc 0 }\) holds.
Similarly \( \neg \LOC{A}{s \conc 0 } \leqW \LOC{A}{s \conc 1 }\)
so \( \LOC{A}{s}\) (and hence \( A \)) is self-dual.
Therefore if \( A \) is non-self-dual then the tree \( \Wtree ( A ) \) is pruned.
Conversely, suppose \( A \) is self-dual.
By a result of Steel and Van Wesep \( A \leql \neg A \), and since
\( \LOC{A}{i} \lel A \) for \( i \in \setLR{0 , 1} \), any branch
of \( \Wtree ( A ) \) would yield an infinite \( \lel \)-descending chain: a contradiction.
Therefore if \( A \) is self-dual then the tree \( \Wtree ( A ) \) is well-founded, hence finite
by K\"onig's lemma.
This implies that at limit levels there is always a pair of non-self-dual degrees.
\item
If \( A_n \leW A_{n + 1} \) for all \( n \) then 
\begin{equation}\label{eq:trianglecircle}
( A_n )_n^\triangledown \equalsdef \bigcup_{n} 0^{(n)}\conc 1 \conc A_n 
\qquad\text{and}\qquad 
( A_n )_n^\circ \equalsdef \setLR{0^{(\infty )}} \cup ( A_n )_n^\triangledown
\end{equation}
give the least non-self-dual pair immediately above the \( A_n \)s.
\end{itemizenew}
\subsection{Proof of Theorem~\ref{th:dense}}\label{subsec:proofofThmdense}

We can now show how Theorem~\ref{th:global} implies that the sets
\( \mathscr{W}_{ \boldsymbol{d}} \) are topologically dense.

\begin{proof}
Let \( A \in \bPi^{0}_{3} \setminus \setLR{\emptyset , \pre{\omega}{2} }\) 
and let \( \boldsymbol{d} = \Wdeg{A} \).
By Theorem~\ref{th:global} there is a \( \mathcal{T} \)-regular \( B \in \Wdeg{A} \) 
such that \( B = \Phi ( U ) = \Phi ( C ) \) for some open set \( U \) and closed set \( C \).
Let \( D \neq \pre{\omega}{2} \) be  clopen and let \( \Nbhd_t \cap D = \emptyset \).
The function \( x \mapsto t \conc f ( x ) \) witnesses that \( A \leqW D \cup t \conc B \),
where \( f \) reduces \( A \) to \( B \).
Conversely \( \II \) wins \( \GW ( D \cup t \conc B , A ) \) as follows:
\begin{quote}
\( \II \) passes until \( \I \) reaches a position inside \( D \),
or else reaches a position oustide \( D \cup \Nbhd_t \), 
or else reaches \( t \).
In the first case \( \II \) plays an \( x \in B \), 
in the second case \( \II \) plays an \( x \notin B \)
in the third case \( \II \) applies the reduction witnessing \( B \leqW A\).
\end{quote}
Therefore \( D \cup t \conc B \in \Wdeg{A} \), and moreover \( D \cup t \conc B \)
is \( \mathcal{T} \)-regular.
Hence Lemma~\ref{lem:approximation} can be applied to the family 
\[ 
\mathcal{B} = 
\setofLR{X \in \Wdeg{A} }{ \EXISTS{U \in \bSigma^{0}_{1}} \EXISTS{C \in \bPi^{0}_{1}} 
( X = \Phi ( U ) = \Phi ( C ) ) } = \mathscr{W}_{\boldsymbol{d}} . \qedhere
\]
\end{proof}

\subsection{Wadge's constructions}\label{subsec:Wadgesconstructions}
Wadge defined the sum of two subsets of the Baire space as
\[
A \Wsum  B = \setofLR{s^+ \conc 0 \conc a}{ s \in \pre{ < \omega }{ \omega} \AND a\in A} \cup B^+
\]
where \( B^+ = \setofLR{b^+}{ b\in B} \) and for \( x \in  \pre{ \leq  \omega }{  \omega }\)
let \( x^+ = \seqofLR{x ( i ) + 1 }{ i \in \dom ( x )}\).
Since in the current set-up \( x \in  \pre{ \leq  \omega }{  2 } \), i.e., it is a 
sequence taking values values in \( 2 \) (rather than \( \omega \)),
then \( x^+ \) is replaced by 
\[
\overline{x} \colon  2 \cdot \dom ( x )\to 2 , \qquad \FORALL{i \in \dom ( x )}  
( \overline{x} ( 2i ) = \overline{x} ( 2i+ 1 ) = x ( i ) ) ,
\]
the sequence obtained from \( x \) by doubling each entry.
If \( T \) is a tree on \( 2 \) and \( A \subseteq \pre{\omega}{2} \) set 
\[
\overline{T} = \setofLR{ \overline{t}}{ t\in T} \quad \text{and} \quad
 \overline{A} = \setofLR{\overline{a}}{ a\in A} . \label{pag:double}
\]
Then for \( A , B \subseteq \pre{\omega}{2} \) let 
\[
A \Wsum  B = \setofLR{ \overline{s} \conc t \conc a}
{ s \in \pre{ < \omega }{ 2} \AND t \in \setLR{01 , 10} \AND a \in A} \cup \overline{B} .
\]
A straightforward adaptation of Wadge's arguments (see~\cite{Andretta:2007ce}
for proofs) yields that if  \( A \) is self-dual, then
\begin{gather*}
A \Wsum  \emptyset \Wequiv A^\triangledown \AND
 A \Wsum  \pre{\omega}{2} \Wequiv A^\circ ,
\\
 B \leqW C \iff A\Wsum  B \leqW A \Wsum  C ,
\\
A \leW B \implies \exists C \leqW B \left ( A \Wsum  C \Wequiv B \right ) ,
\\
 \Wrank{A \Wsum B } = \Wrank{A} + \Wrank{B},
\end{gather*}
and for any \( A \) (not necessarily self-dual)
\begin{equation}\label{eq:emptyinterior}
\Wrank{A} \geq  \omega \implies  A \setminus \Int A \Wequiv A . 
\end{equation}

Starting from \( \emptyset \) and \( \pre{\omega}{2} \) and using
the operations \((A , B ) \mapsto A \oplus B \), \( (A , B) \mapsto A \Wsum  B\)
and  the constructions in~\eqref{eq:oplus} and~\eqref{eq:trianglecircle}
it is easy to construct subsets of the Cantor space in any Wadge degree  of rank
\( < \omega _1 \).
To reach further heights we modify again two constructions from~\cite{Wadge:1983sp}.
Let
\begin{align*}
A^\natural & =\setofLR{\overline{s_1} \conc \eta_1 \conc 
		\overline {s_2} \conc \eta_2 \conc \dots 
		\conc \overline{s_n} \conc \eta_n \conc \overline{a}}
		{n \in \omega \AND s_i \in  \pre{< \omega }{ 2 } \AND \eta_i \in \setLR{01, 10}\AND a \in A }
\\
A^\flat & = A^\natural \cup \setofLR{x \in \pre{\omega}{2} }
		{ \EXISTSS{\infty}{n} \left (x ( 2n )\neq x ( 2n+1)  \right ) } .
\end{align*}
Both \( A^\natural \) and \( A^\flat \) have a self-similarity property, 
in the sense that \( \LOC{A^\natural }{\overline{s}} = A^\natural \) and 
\( \LOC{A^\flat }{\overline{s}} = A^\flat \) for any \( s \in  \pre{ < \omega }{2}\). 
The intuition behind the definition of \( A^\natural \) is that it is the union of \( \omega \)-many layers ---
at each layer there is a copy of \( \overline{A} \) and in order to 
leave the \( n \)-th layer and enter the \( n + 1 \)-st layer we must follow a string 
of the form
\[
\overline{s_1} \conc \eta_1 \conc \overline {s_2} \conc \eta_2 \conc \dots 
\conc \overline{s_n} \conc \eta_n \conc \overline{s_{n+1}} \conc \eta_{n+1}
\]
where the \( \eta_i \)'s are \( 01 \) or \(10 \).
Wadge's original definition was given for subsets of the Baire space
\( A \subseteq \pre{\omega}{\omega} \)  
\begin{align*}
A^\natural & =\setofLR{s_1 ^+ \conc 0 \conc s_2^+ \conc 0 \conc \dots 
\conc s_n^+ \conc 0 \conc x^+}{n \in \omega , s_i \in  \pre{< \omega }{ \omega } , x \in A }
\\
A^\flat & = A^\natural \cup \setofLR{x \in \pre{\omega}{\omega} }
{ \EXISTSS{\infty}{n} \left ( x ( n ) = 0  \right ) },
\end{align*}
and in~\cite{Wadge:1983sp} it is shown (see~\cite{Andretta:2006xv} for detailed proofs)
that whenever \( A \) is self-dual then:
\begin{subequations}
\begin{align}
& A^\natural \text{ and \( A^\flat \) are non-self-dual,} \label{eq:Wadgenatural1}
\\
& A^\natural \Wequiv \neg A^\flat , \label{eq:Wadgenatural2}
\\
& \Wrank{A^\natural} = \Wrank{A^\flat } = \Wrank{A}\cdot \omega _1 . \label{eq:Wadgenatural3}
\end{align}
\end{subequations}
The proofs of~\eqref{eq:Wadgenatural1} and~\eqref{eq:Wadgenatural2}
generalize to the Cantor space with minor adjustments.
For~\eqref{eq:Wadgenatural3} we must show that
\begin{enumerate-(A)}
\item\label{en:WadgenatA}
for every \( 1 \leq \alpha < \omega _1 \) there is a self-dual set \( A_ \alpha \) of 
Wadge rank \( \Wrank{A} \cdot \alpha \) if \( \alpha \) is a successor, or 
\( \Wrank{A} \cdot \alpha + 1 \) if \( \alpha \) is limit, 
and such that \( A_ \alpha \leqW A^\natural , A^\flat \), and
\item\label{en:WadgenatB}
if \( B \leW A^\natural , A^\flat \) then \( B \leqW A_ \alpha \), for some \( \alpha \).
\end{enumerate-(A)}
The sets \( A_ \alpha \) are constructed by induction on \( \alpha \)
by taking \( A_1 = A \), \( A_{ \alpha + 1} = A_ \alpha \Wsum  A\)
and, for \( \lambda \) limit, \( A_  \lambda = (A_{ \alpha _n})_n^\triangledown
\oplus (A_{ \alpha _n})_n^\circ \), where \( (\alpha _n)_n \) 
is increasing and converging to \( \lambda \).
To check that \( A_ \lambda \leqW A^\natural , A^\flat \) for \( \lambda \) limit
it is enough to check that 
\[
( A_{ \alpha _n})_n^\triangledown \leqW  A^\natural \qquad
\text{and} \qquad  (A_{ \alpha _n})_n^\circ \leqW  A^\natural .
\]
To prove the first inequality it is enough to show that \( \II \) wins 
\( \GL ( (A_{ \alpha _n})_n^\triangledown  , A^\natural ) \) as follows:
\begin{quote}
As long as \( \I \) plays \( 0\) let \( \II \) enumerate \( \overline{b} \), for some \( b \notin A \).
If \( \I \) plays \( 1 \)  at round \( n \), then \( \II \) plays \( 0 1 \) and then follows
a reduction witnessing \( A_{ \alpha _n} \leqW A^\natural \).
\end{quote}
If the real \( b \) is taken to be in \( A \), the strategy above shows that 
\( \II \) wins \( \GL ( (A_{ \alpha _n})_n^\circ  , A^\natural ) \).
Therefore~\ref{en:WadgenatA} is proved.
To prove~\ref{en:WadgenatB} fix a set \( B \leW A^\natural \).
By (a simple adaptation of) \cite[Claim 3.9, p.\ 49]{Andretta:2007ce} we may assume that 
\( \II \) wins \( \GW ( B ,  A^\natural )\) via some strategy \( \tau \) that always yields 
reals in \( ( \pre{\omega}{2} )^\natural \).
Let \( \mathcal{T} \) be the tree of attempts to construct a play for \( \I \) 
such that \( \tau \)'s reply is an element of 
\( \setofLR{x \in \pre{\omega}{2} }{\EXISTSS{\infty}{n} \left (x ( 2n )\neq x ( 2n+1)  \right ) }\).
To be more precise: call \( s \in  \pre{ < \omega }{2}\) a position for \( \I \)
\markdef{special} if:
\begin{enumerate-(i)}
\item
\( \tau \) does not pass when pitted against \( s \), that is \( \tau ( s ) \in \setLR{0,1}\), and
\item
\( \II \)'s position after this inning is of even length and of the form \( u \conc ( 1-i)\conc i \).
\end{enumerate-(i)}
Then
\begin{multline*}
\mathcal{T} = \setLR{ \emptyset} \cup\bigl \{ \seqLR{s_0 , \dots , s_n } \Mid  \FORALL{i \leq n } 
 ( s_0 \conc \dots \conc s_i \text{ is special} )
\\
{} \AND \forall t \subseteq s_0 \conc \dots \conc s_n \left ( t \text{ special} \implies \EXISTS{i\leq n} 
 \left (t = s_0 \conc \dots \conc s_i \right )  \right ) \bigr \} .
 \end{multline*}
By assumption on \( \tau \), the tree \( \mathcal{T}\) is well-founded, hence of rank \( \alpha < \omega _1\).
We will show that \( B \leqW A_{ \alpha +1} \).
If \( \alpha =0\) then \( \tau \)  induces a continuous function
\( f \colon \pre{\omega}{2}  \to \overline{ \pre{\omega}{2}} \) witnessing that 
\( B = f^{-1} ( \overline{A} ) \).
Thus \( B \leqW A = A_1\).
Suppose now \( \alpha > 0 \): as long as \( \I \) never reaches a special position, then 
\( \tau \) reduces \( B \) to \( \overline{A} \) as before; if at some stage \( \I \) reaches a
special position \( s \) for the first time, then the rank of the node \( \seqLR{s }\)
in \( \mathcal{T} \) will be \( \beta < \alpha \), hence by inductive assumption there is a 
continuous reduction of \( \LOC{B}{s} \) to \( A_{ \beta +1} \leqW A_ \alpha \).
Therefore \( B \leqW A_{ \alpha } \Wsum  A = A_{ \alpha +1}\), as required.
This proves~\ref{en:WadgenatB}, hence~\eqref{eq:Wadgenatural3} is established. 

\subsection{The hierarchy of \texorpdfstring{\( \bDelta^{0}_{3}\)}{Delta03} sets in the Cantor space}%
\label{subsec:Delta03hierarchy}

Since every winning strategy for \( \II \) in \( \GL ( A , B ) \) or in \( \GW ( A , B ) \) 
can easily be extended to a winning strategy for \( \II \) in \( \GL^* ( A , B ) \) or in \( \GW^* ( A , B ) \),
then 
\[ 
A \leqW B \implies A \leqW^* B \quad\text{and}\quad  A \leql B \implies A \leql^* B .
\]
The converse is not necessarily true: for example \( 0 \conc {}\pre{\omega}{2} \leql^* \pre{\omega}{2} \) 
but \( 0 \conc {}\pre{\omega}{2} \nleqW \pre{\omega}{2} \).
\begin{lemma}
Suppose \( A , B \subseteq \pre{\omega}{2} \) and \( B\) has empty interior in \( \pre{\omega}{2} \).
Then 
\[ 
A \leql^* B \implies A \leql B \quad\text{and}\quad A \leqW^* B \implies A \leqW B.
\]
\end{lemma}

\begin{proof}
Let \( \tau \) be a winning strategy for \( \II \) in the game \( \GL ^* ( A , B ) \).
We will transform \( \tau \) into  \( \tilde{\tau }\), still a winning strategy for \( \II \) in the same game 
so that its restriction to \(  \pre{<  \omega }{2}\) is a winning strategy in \( \GL ( A , B ) \).
(The result for Wadge reductions is proved similarly.)
Suppose that at some round of \( \GL ( A , B ) \) \( \I \) has reached a position \( p \)
and that \( \II \), following \(  \tau \), has reached a position \( q \).
Call such a \(  p \) \markdef{critical} iff its length is \( n + 1 \) and 
\begin{itemize}
\item
\( p \in  \pre{ n+1}{2}\), 
\item
\( \forall k < n \left ( q ( k ) \in \setLR{0,1} \right ) \), and
\item
\( q ( n ) \in  \omega \setminus \setLR{0,1}\).
\end{itemize}
(Note that \( \LOC{A}{p} \leql^* \LOC{B}{q} = \emptyset \) hence \( \LOC{A}{p} = \emptyset \).)
As \( \LOC{B}{q \restriction n} \neq \pre{\omega}{2}\) by our assumption on \( B \),
fix \( b_p \in \pre{\omega}{2} \setminus B \) such that \( b_p \supseteq q \restriction n \).
We are now ready to define \( \tilde{ \tau} \):
\begin{quote}
As long as \( \I \) does not reach a critical position, the \( \tilde{  \tau }\) is just \( \tau \).
As soon as \( \I \) reaches a critical position \( p \), then from this point on \( \tilde{  \tau } \) follows \( b_p \).
\end{quote}
We leave it to the reader to check that \( \tilde{  \tau } \) is a winning strategy 
for \( \II \) in the game \( \GL^* ( A , B ) \) such that its restriction to 
\(  \pre{<  \omega }{2}\) is a winning strategy for \( \II \) in \( \GL ( A , B ) \).
\end{proof}
By~\eqref{eq:emptyinterior} we obtain
\begin{corollary}\label{cor:equivalenceofWadgereducibility}
If \( A , B \subseteq \pre{\omega}{2} \) and \( \Wrank{B} \geq  \omega \), then
\[ 
 A \leqW^* B \iff A \leqW B.
\]
In particular the map \( \Wdeg{A}\mapsto \Wdeg{A}^* \) is well defined and injective, 
as long as \( \Wrank{A}\geq  \omega \).
\end{corollary}

Wadge showed~\cite{Wadge:1983sp} that the length of the Wadge hierarchy 
of \( \bDelta^{0}_{2} \) and \( \bDelta^{0}_{3} \) subsets of the Baire space is, respectively, 
\( \omega _1 \) and \( \omega _1^{ \omega _1} \), hence
\begin{align*}
\sup \setofLR{\Wrank{A}}{A \in \bDelta^{0}_{2} } & \leq  \omega _1
\\
\sup \setofLR{\Wrank{B}}{B \in \bDelta^{0}_{3} } & \leq  \omega _1^{\omega_1} .
\end{align*}

\begin{theorem}\label{th:WadgehierarchyinCantorandBaire}
Let \( \mathcal{A} \subseteq \Pow ( \pre{\omega}{2} ) \) be the smallest family containing 
\( \emptyset \) and \(\pre{\omega}{2} \) and closed under the operations 
\begin{enumerate}[label={\upshape (O-\arabic*)}, leftmargin=3pc]
\item\label{th:WadgehierarchyinCantorandBaire-1}
\( A \mapsto \neg A\),
\item
\( ( A , B ) \mapsto A\Wsum B \),
\item
\( ( A , B ) \mapsto A \oplus B \),
\item
\( (A_n )_n \mapsto (A_n )_n^\triangledown\), and
\item\label{th:WadgehierarchyinCantorandBaire-5}
\( (A_n )_n \mapsto (A_n )_n^\circ\),
\end{enumerate}
and let \( \mathcal{B} \subseteq \Pow ( \pre{\omega}{2} )\) be the smallest 
family containing \( \emptyset \) and \(\pre{\omega}{2} \) and closed under the 
operations~\ref{th:WadgehierarchyinCantorandBaire-1}--\ref{th:WadgehierarchyinCantorandBaire-5} 
above and also closed under 
\begin{enumerate}[label={\upshape (O-\arabic*)}, leftmargin=3pc,resume]
\item\label{th:WadgehierarchyinCantorandBaire-6}
\( A \mapsto A^\natural \),
\item\label{th:WadgehierarchyinCantorandBaire-7}
\( A \mapsto A^\flat\).
\end{enumerate}

Then \( \mathcal{A} \subseteq \bDelta^{0}_{2}\) and \( \mathcal{B} \subseteq \bDelta^{0}_{3}\) 
and \( \mathcal{A} \) intersects every Wadge degree in \( \bDelta^{0}_{2}\) 
and \( \mathcal{B} \) intersects every Wadge degree in \( \bDelta^{0}_{3}\), that is
\begin{align*}
& \forall X \subseteq \pre{\omega}{2} \left ( X \in \bDelta^{0}_{2} 
\implies \exists  A\in \mathcal{A} \left ( A \Wequiv X \right ) \right )
\\
&\forall X \subseteq \pre{\omega}{2} \left ( X \in \bDelta^{0}_{3} 
\implies \exists  B\in \mathcal{B} \left ( B \Wequiv X \right ) \right ) .
\end{align*}
\end{theorem}

\begin{proof}
It is immediate to check that both \( \bDelta^{0}_{2}\) and \( \bDelta^{0}_{3}\) are closed 
under~\ref{th:WadgehierarchyinCantorandBaire-1}--\ref{th:WadgehierarchyinCantorandBaire-5}
and that \( \bDelta^{0}_{3}\) is closed under~\ref{th:WadgehierarchyinCantorandBaire-6}
and~\ref{th:WadgehierarchyinCantorandBaire-7}.
It is enough to prove by induction on \( \alpha \) that
\begin{subequations}
\begin{align}
\Wrank{X} =  \alpha <  \omega _1 & \implies \exists A \in \mathcal{A} \left (A \Wequiv X \right ) \label{eq:WadgehierarchyinCantorandBaire1}
\\
\Wrank{X} =  \alpha <  \omega _1^{\omega_1} & \implies \exists B \in \mathcal{B} \left ( B \Wequiv X \right ) .
\label{eq:WadgehierarchyinCantorandBaire2}
\end{align}
\end{subequations}
So fix \( \Wrank{X} =  \alpha <  \omega _1^{  \omega _1}\).
\begin{itemizenew}
\item
If \(  \alpha = 1\) then \( X = \emptyset\) or \( X = \pre{\omega}{2} \), 
so \( X \) belongs to both \( \mathcal{A} \) and \( \mathcal{B} \).
\item
If \(  \alpha =  \beta + 1\), then there is a set \( Y \) of rank \( \beta \)
which is in \( \mathcal{A} \) if \(  \alpha <  \omega _1\) or in \(  \mathcal{B}\) otherwise.
If \( X \) is self-dual, then \( X \Wequiv Y \oplus \neg Y \), and if \( X \) is non-self-dual, then 
either \( X \Wequiv Y^\triangledown \) or else  \( X \Wequiv Y^\circ \), so the theorem 
is proved when \( \alpha \) is a successor.
\item
Suppose now \( \alpha \) is limit.
\begin{itemizenew}
\item
If \( \cof (  \alpha ) =  \omega \) choose an increasing sequence \(  \alpha _n \to  \alpha \) 
and sets \( Y_n \) such that \( \Wrank{Y_n}=  \alpha _n \): then either 
\( X \Wequiv (Y_n)_n^\triangledown\) or else \( X \Wequiv (Y_n)_n^\circ \).
Since \( \cof (  \alpha ) =  \omega \) when \(  \alpha <  \omega _1\), 
then~\eqref{eq:WadgehierarchyinCantorandBaire1} is proved.
\item
Suppose now \( \cof (  \alpha ) =  \omega _1\).

If \( \alpha = \beta _1 + \beta _2 \) with \( \beta _1 , \beta _2 < \alpha \),
then --- by replacing \( \beta _1 \) with its successor if needed --- we may
assume that any set of Wadge rank \( \beta _1\) is self-dual.
By inductive assumption there are \( B_1 , B_2 \in \mathcal{B}\) such that 
\( \Wrank{B_i} = \beta _i\) and \( B_1 \Wsum  B_2 \Wequiv X \).
Since \( \mathcal{B}\) is closed under addition of sets we are done.

Therefore we may assume that \( \alpha \) is additively indecomposable, hence
\( \alpha = \omega _1^ \xi \cdot \nu \) with \( 1 \leq \nu < \omega _1\).
As \( \cof ( \alpha ) > \omega \), then \( \nu \) cannot be a limit or a successor ordinal \( > 1 \), hence \( \nu =1 \),
so \( \alpha = \omega _1 ^ \xi \).
Again by \( \cof ( \alpha ) > \omega \) it follows that \( \xi \) cannot be limit, so
\( \alpha =  \omega _1^ { \gamma}\cdot \omega _1 \).
If \( \gamma = 0 \) then \( \alpha = \omega _1\) and therefore either 
\( X \Wequiv D^ \natural\) or else \( X \Wequiv D^\flat \) with
\( D \) self-dual, hence~\eqref{eq:WadgehierarchyinCantorandBaire2} holds.
Thus we may assume that \( \gamma  > 0 \).
By inductive hypothesis there is a set \( B \in \mathcal{B}\) of Wadge rank 
\( \omega _1^ { \gamma} + 1\), and since \( \omega _1^ { \gamma} \)
is limit then \( B \) is self-dual.
As \( \alpha = ( \omega_1^ \gamma + 1 ) \cdot \omega _1 \)
hence by~\eqref{eq:Wadgenatural3} it follows that either \( X \Wequiv B^\natural \) 
or else \( X \Wequiv B^\flat \), and since \( \mathcal{B}\) is closed under 
operations~\ref{th:WadgehierarchyinCantorandBaire-6}
and~\ref{th:WadgehierarchyinCantorandBaire-7}, 
then~\eqref{eq:WadgehierarchyinCantorandBaire2} holds for \( \alpha \). 
\end{itemizenew}
\end{itemizenew}
This completes the induction and the theorem is proved.
\end{proof}

\begin{corollary}\label{cor:lengthofWadge}
The length of the Wadge hierarchy on \( \pre{\omega}{2} \) restricted to 
\( \bDelta^{0}_{2}\) is \( \omega _1\), and restricted to 
\( \bDelta^{0}_{3}\) is \( \omega _1^{  \omega _1} \).
\end{corollary}

\subsection{A well-quasi-order on {\rm\( \MALG \)}}\label{subsec:wqoonMALG}
The Wadge hierarchy induces a well-quasi-order  \( \preceq\) on \( \MALG \) 
\[
\eq{A} \preceq \eq{B} \IFF \EXISTS{ f }  \FORALL{x \in \pre{\omega}{2} } 
\bigl ( x\in \hat{ \Phi } ( \eq{A} ) \iff f ( x ) \in \hat{ \Phi } ( \eq{B} ) \bigr ) 
\]
which is \( \bSigma^{1}_{2}\).
Similarly, for any Borel set \( C \) the sets 
\begin{align*}
P_{\leq C} = &  \setofLR{\eq{A} \in \MALG}{\EXISTS{ f }  \FORALL{x \in \pre{\omega}{2} } 
\bigl ( x\in \hat{ \Phi } ( \eq{A} ) \iff f ( x ) \in C \bigr )}
\\
P_{C \leq } = &  \setofLR{\eq{A} \in \MALG}{\EXISTS{ f }  \FORALL{x \in \pre{\omega}{2} } 
\bigl ( x\in C \iff f ( x ) \in \hat{ \Phi } ( \eq{A} ) \bigr )}
\end{align*}
are \( \bSigma^{1}_{2}\).
Therefore if \( \boldsymbol{d} \subseteq \bPi^{0}_{3}\) is a Wadge degree then  
\[ 
\begin{split}
\mathscr{W}_{ \boldsymbol{d}} & = \setofLR{\eq{A} \in \MALG }{ \Phi ( A ) \in \boldsymbol{ d} }
\\
 & = P_{\leq A_0} \cap P_{A_0 \leq}
\end{split} 
\]
is \( \bSigma^{1}_{2}\), where \( A_0 \) is a Borel set such that \( \Wdeg{A_0} \in \boldsymbol{d}\).

We will now observe that the \( \mathscr{W}_{ \boldsymbol{d}}\) are \( \bDelta^{1}_{2}\).
Since \( \setofLR{\eq{A} \in \MALG }{ \Phi ( A ) \in  \bDelta^{0}_{3} } = P_{ \leq C} \) with \( C \)
a complete \( \bSigma^{0}_{3}\) set, if \(  \boldsymbol{d} = \bPi^{0}_{3} \setminus \bDelta^{0}_{3} \) 
then \( \mathscr{W}_{ \boldsymbol{d}} \) is \( \bPi^{1}_{2}\) hence \( \bDelta^{1}_{2} \).
If \( \boldsymbol{d} \subseteq \bDelta^{0}_{3} \) we take cases, accordingly to whether 
it is self-dual or non-self-dual.
If  \( \boldsymbol{d} \) is self-dual, choose \( B_1 , B_2 , B_3 , B_4\) such that 
\( \Wdeg{B_1} , \Wdeg{B_2} \) are the immediate predecessors of \( \boldsymbol{d} \)
and \( \Wdeg{B_3} , \Wdeg{B_4} \) are the immediate successors of \( \boldsymbol{d} \):
then \( \MALG \setminus \mathscr{W}_{\boldsymbol{d}} = P_{\leq B_1} \cup P_{\leq B_2}
\cup P_{\geq B_3} \cup P_{\geq B_4} \) is \( \bSigma^{1}_{2}\).
If \( \boldsymbol{d} \) is self-dual, choose \( B \) such that \( \Wdeg{B} = \breve{\boldsymbol{d}} \):
then  \( \MALG \setminus \mathscr{W}_{\boldsymbol{d}} = P_{\leq B} \cup  P_{\geq B} \) 
is \( \bSigma^{1}_{2}\).

Theorems~\ref{th:Pi03complete} and~\ref{th:belowDelta03} will guarantee that every  
\( \mathscr{W}_{ \boldsymbol{d}} \) is nonempty, hence
the well-quasi-order \( \preceq\) has length \( \omega_1^{\omega_1} + 1 \).

\section{Climbing the \texorpdfstring{\( {\boldsymbol\Delta} ^{0}_{2} \)}{Delta02}-hierarchy}\label{sec:climbing}%
In this and the next  section, the  constructions of Section~\ref{subsec:Wadgesconstructions}  will be modified so that
they take \( \mathcal{T} \)-regular sets into \( \mathcal{T} \)-regular sets. 

If \( A \) and \( B \) are \( \mathcal{T} \)-regular, so is \( A \oplus B \).
But even if every \( A_n \) is \( \mathcal{T} \)-regular, there is no guarantee 
that \( (A_n )_n^\triangledown \) and \( ( A_n )_n^\circ \) will be \( \mathcal{T} \)-regular.
Our first goal is to fix this problem.

The sets \( U_f \) of the examples of Section~\ref{sec:Someexamples} are 
obtained by appending \( \pre{\omega}{2} \) to the terminal nodes of the tree
\[
\setofLR{s \in  \pre{ < \omega }{2} }{ \exists{n} \left ( s \subseteq 0^{( n)} 1^{( f ( n ) )} \right ) } ,
\]
which is shaped like a rake (Figure~\ref{fig:rake}).
\begin{figure}
   \centering
   \begin{tikzpicture}
\draw (0,0)--(-5, -5); \draw[densely dotted] (-5, -5) -- (-6,-6) node [below] {\( 0^{( \infty )} \)};
\filldraw (0,0)--++(2,-2) circle (2pt) node [right]{\( 1^{( f ( 0 ) )}\)}
(-1,-1)-- ++(1,-1) circle (2pt) node [right]{\(0 1^{( f ( 1 ) )}\)}
(-2,-2)-- ++(3,-3) circle (2pt) node [right]{\(0^{( 2 )} 1^{( f ( 2 ) )}\)}
 (-3,-3)-- ++(1,-1) circle (2pt) node [right]{\(0^{(3)} 1^{( f ( 3 ) )}\)}
  (-4,-4)-- ++(2,-2) circle (2pt) node [right]{\(0^{(4)} 1^{( f ( 4 ) )}\)};
\end{tikzpicture}
   \caption{A rake.}
   \label{fig:rake}
\end{figure}
This construction can be generalized by appending different sets at the terminal nodes:
for any \( f \colon \omega  \to \omega \setminus \setLR{0} \) and any sequence of 
sets \( A_n \) (\(n \in \omega \)), let
\[
\Rake ( f ; ( A_n )_n ) = \bigcup_{n} 0^{(n)} 1^{(f ( n ))} \conc A_n .
\]
When \( A_n = A \) for all \( n \), we write \( \Rake ( f ; A ) \).
Note that the sets  \( U_f \) of Section~\ref{sec:Someexamples} are 
exactly the sets \(\Rake ( f ; \pre{\omega}{2} ) \).

There are times when we need rakes with a pole and densely packed tines.
In our case we need the tree whose terminal nodes are the sequences 
\( 0^{(n)} 1 \conc s \) of length \(n + f ( n ) \) --- in Figure~\ref{fig:rake+} 
the nodes different from the ones of the form \( 0^{( n)} 1^{( f ( n ) )} \) are 
drawn in a paler shade of gray.
\begin{figure}
   \centering
\begin{tikzpicture}
\draw (0,0)--(-5, -5); \draw[densely dotted] (-5, -5) -- (-6,-6) ;
\filldraw (0,0)--++(2,-2) circle (2pt) 
(-1,-1)-- ++(0.7,-1) circle (2pt) 
(-2,-2)-- ++(3,-3) circle (2pt) 
 (-3,-3)-- ++(0.7,-1) circle (2pt) 
  (-4,-4)-- ++(2,-2) circle (2pt) ;
\filldraw[very thick, gray!50] (1,-1)--+(-0.7,-1) circle (2pt)
(-1,-3)--+(-0.7,-1)
 (0,-4)--+(-0.7,-1) circle (2pt)
(-3,-5)--+(-0.7,-1) circle (2pt);
\filldraw[very thick, gray!50] (-1.7, -4)--+(-0.5 ,-1) circle (2pt);
\filldraw[very thick, gray!50] (-1.7, -4)--+(0.5 ,-1) circle (2pt);
\end{tikzpicture}
   \caption{A rake with a pole and densely packed tines.}
   \label{fig:rake+}
\end{figure}
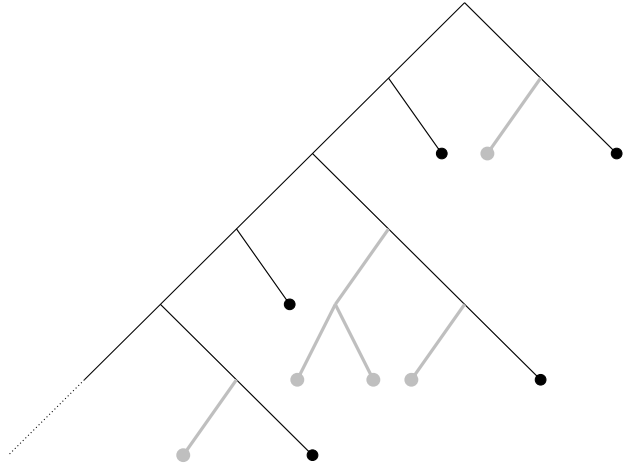
Let \( \Rakep ( f ; ( A_n )_n )\) be the set obtained by appending a copy of \( A_n \) to 
the \( n \)th terminal node
\tikz \fill circle (2pt);, 
and by taking the basic open sets in all other terminal nodes
\tikz \fill[gray!50] circle (2pt);, 
together with the zero-sequence, that is
\begin{multline*}
\Rakep ( f ; ( A_n )_n ) = \setLR{0^{( \infty )}} \cup \Rake ( f ; ( A_n )_n ) \cup {}
\\
\bigcup \setofLR{\Nbhd_t}{ \exists n
\left (\lh ( t ) = n + f ( n )  \AND t \neq 0^{(n)}1^{(f ( n ))} \AND t \supseteq 0^{(n)} 1 \right )}.
\end{multline*}
Note that the \( \Rake \) and \( \Rakep\) constructions commute with the \( \Phi \) operation, 
in the sense that if \( \lim_n f ( n ) = + \infty \), then 
\begin{align*}
\Phi \left ( \Rake ( f  , (A_n)_n ) \right )  & =  \Rake ( f  , ( \Phi (A_n) )_n ) 
\\
\Phi \left ( \Rakep ( f  , (A_n)_n ) \right )  & =  \Rakep ( f  , ( \Phi (A_n) )_n ) .
\end{align*}

\begin{proposition}\label{prop:Rake}
Let \( f \colon \omega \to \omega \setminus \setLR{0} \) and \( A_n \subseteq \pre{\omega}{2} \).
Then
\[
( A_n )_n^\triangledown \Wequiv \Rake ( f ; ( A_n )_n ) .
\]
Suppose moreover that \( \lim_n f ( n ) = \infty \).
Then:
\begin{itemize}
\item
if \( A_n \in \mathcal{M} \) for every \( n \), then 
\( \Rake ( f ; (A_n)_n ) \in \mathcal{M} \),
\item
if \( A_n \in \ran ( \Phi\restriction \bPi^{0}_{1} ) \) for every \( n \), then 
\( \Rake ( f ; (A_n)_n ) \in \ran ( \Phi\restriction \bPi^{0}_{1} ) \),
\item
if \( A_n \in \ran ( \Phi\restriction \bSigma^{0}_{1} ) \) for every \( n \), then 
\( \Rake ( f ; (A_n)_n ) \in \ran ( \Phi\restriction \bSigma^{0}_{1} ) \).
\end{itemize}
\end{proposition}

\begin{proof}
\( \II \) wins \( \GW (( A_n)_n^\triangledown , \Rake ( f ; (A_n)_n )  ) \) as follows:
\begin{quote}
As long as \( \I \) plays \( 0 \)'s then \( \II \) copies \( \I \)'s moves.
If \( \I \) reaches a position \( 0^{(n )} 1 \) then \( \II \) plays \( 1 \) from now on
until position \( 0^{( n)} 1^{(f ( n ) )} \) is reached: at this point \( \II \) will copy 
the moves \( \I \) played after position \( 0^{(n)} 1\).
\end{quote}

Conversely \( \II \) wins \( \GW ( \Rake ( f ; (A_n)_n ) , ( A_n)_n^\triangledown  ) \) as follows:
\begin{quote}
As long as \( \I \) plays \( 0 \)'s then \( \II \) copies \( \I \)'s moves.
If  after \( 0^{(n )} \) \( \I \) starts playing \( 1 \)s, then \( \II \) passes until \( \I \) 
has reached position \( 0^{(n)} 1^{( f ( n ) )} \): at that point \( \II \) plays \( 1 \) 
and  from now on copies \( \I \)'s moves.

If instead \(\I \) does not reach \( 0^{(n)} 1^{( f ( n ) )} \), i.e., \( \II \) plays \( 0 \) 
after \(  0^{(n)} 1^{(m)} \) with \( m < f ( n ) \) so that his play will not be in 
\( \Rake ( f ; (A_n)_n ) \), then \( \I \) plays \( 0 \)'s from now on so that the resulting 
play will be \( 0^{( \omega )} \notin ( A_n)_n^\triangledown\).
\end{quote}

Suppose now that \( f ( n ) \to \infty \).
If \( A_n = \Phi ( C_n ) \) with \( C_n \) closed for all \( n \), then
\[ 
\Rake ( f ; (A_n)_n )  = \Phi \bigl ( \setLR{0^{(\infty )}} 
\cup \Rake ( f ; ( C_n)_n ) \bigr ) .
\]
Similarly if \( A_n = \Phi (U_n ) \) with \( U_n \) open, then 
\( \Rake ( f ; (A_n)_n )  = \Phi ( \Rake ( f ; ( U_n)_n ) ) \in \ran ( \Phi \restriction \bSigma^{0}_{1})\).
Finally,  if \( A_n \in \mathcal{M} \) for all \( n \), then \( \Rake ( f ; (A_n)_n )  \in \mathcal{M} \).
\end{proof}

Arguing as in Proposition~\ref{prop:Rake} we obtain:
\begin{proposition}\label{prop:Rakep}
Let \( f \colon \omega \to \omega \setminus \setLR{0} \) and \( A_n \subseteq \pre{\omega}{2} \).
Then
\[
( A_n )_n^\circ \Wequiv \Rakep ( f ; ( A_n )_n ) .
\]
Suppose moreover that \( \lim_n f ( n ) = \infty \).
Then:
\begin{itemize}
\item
if \( A_n \in \mathcal{M} \) for every \( n \), then 
\( \Rakep ( f ; (A_n)_n ) \in \mathcal{M} \),
\item
if \( A_n \in \ran ( \Phi\restriction \bPi^{0}_{1} ) \) for every \( n \), then 
\( \Rakep ( f ; (A_n)_n ) \in \ran ( \Phi\restriction \bPi^{0}_{1} ) \),
\item
if \( A_n \in \ran ( \Phi\restriction \bSigma^{0}_{1} ) \) for every \( n \), then 
\( \Rakep ( f ; (A_n)_n ) \in \ran ( \Phi\restriction \bSigma^{0}_{1} ) \).
\end{itemize}
\end{proposition}

We are now ready to prove Theorem~\ref{th:belowDelta03} for Wadge 
degrees contained in \( \bDelta^{0}_{2}\).

\begin{theorem}\label{th:T-regularDelta02}
The class 
\[
\mathcal{N} \equalsdef  \mathcal{M} \cap \ran ( \Phi \restriction \bPi^{0}_{1} ) \cap
 \ran ( \Phi \restriction \bSigma^{0}_{1} )
\] 
intersects every Wadge degree in \( \bDelta^{0}_{2}\), that is 
\( \FORALL{A \in \bDelta^{0}_{2}} \EXISTS{B \in \mathcal{N} } \bigl ( A \Wequiv B \bigr ) \).
\end{theorem}

\begin{proof}
The result is proved by induction on \( \Wrank{A} < \omega _1 \),
using the fact that \( A \in \bDelta^{0}_{2} \iff  \Wrank{A} < \omega _1 \)
(Corollary~\ref{cor:lengthofWadge}).
The case \( \Wrank{A} = 1 \) is trivial, since it implies that \( A = \pre{\omega}{2} \)
or \( A = \emptyset \) hence \( A \) is \( \mathcal{T} \)-regular and \( A \in\mathcal{N}\),
so we may assume that \( \Wrank{A} > 1 \).

If either \( \Wrank{A} \) is limit, or  \( A \) is non-self-dual and \( \Wrank{A} \) 
is a successor ordinal, then apply the inductive assumption to
Propositions~\ref{prop:Rake} and~\ref{prop:Rakep} so that 
\( A \Wequiv B \) where either \( B = \Rake ( f , ( A_n )_n ) \in \mathcal{N}\) 
or \( B = \Rakep ( f , ( A_n )_n ) \in \mathcal{N}\).

If \( A \) is self-dual then \( A \Wequiv C \oplus \neg C \), 
hence by inductive assumption there are \( B_1 , B_2  \in \mathcal{N} \)
such that \( C \Wequiv B_1 \) and \( \neg C \Wequiv B_2 \).
Then \( B_1 \oplus B_2 \in \mathcal{N}\) 
and \( A \Wequiv  B_1 \oplus B_2 \).

As every Wadge degree in \( \bDelta^{0}_{2} \) is obtained via these operations,
the result is proved.
\end{proof}

%

Using the results proved so far, together with Example~\ref{xmp:opennotdualistic}
for the last inclusion we obtain
\begin{corollary}\label{cor:ManicheanDelta02}
If  \( \mathcal{T} \) is the density topology, then
\[
 \setofLR{A}{ A\text{ is \( \mathcal{T} \)-clopen}} = \ran ( \Phi ) \cap \mathcal{M} 
 = \ran ( \Phi  \restriction \mathcal{M} ) \subset  \bDelta^{0}_{2}\cap \ran ( \Phi) .
\]
\end{corollary}

\section{Wadge-style constructions}\label{sec:Wadge-styleconstructions}
The next goal is to define operations on subsets of the Cantor space that are
the analogues of \( \overline{A}\), \( A \Wsum  B \), \( A^\flat \), and \( A^\natural\), and that 
preserve \( \mathcal{T}\)-regularity.
In order to avoid repetitions, let's agree that in this section, unless otherwise stated, 
\( A \) and \( B \) vary over measurable subsets of \( \pre{\omega}{2} \) and
\( 0 < \mu ( A ) , \mu (B) < 1 \).

Since \( \overline{A} \) is always null, we must 
add some extra open sets on the side.
To this end we define canonical clopen sets.

\begin{definition}\label{def:O(r)}
For \( r \in [ 0 ; 1 ) \) let 
\[
 k = k( r ) = \text{the least \( h > 0 \) such that } r \leq 1 - 2^{- h} , 
\]
and let 
\begin{equation*}
 \uu ( r )  =   0^{( k - 1 )} 1 ,
\end{equation*}
and let
\[
O ( r ) = \pre{\omega}{2}  \setminus \Nbhd_{ \uu ( r )} = 
\Nbhd_{0^{(k)}}\cup \bigcup_{m + 1< k} \Nbhd_{0^{(m)} 1} .
\]
\end{definition}
Figure~\ref{fig:u(r)} may help the reader to visualize the node \( \uu ( r ) \) and the set \( O ( r ) \) 
as the union of \( k\) basic open sets.

\begin{remark}\label{rmk:O(r)}
The definition of \( \uu(r) \) (and hence of \( O ( r ) \)) seems unduly strange, but it has the merit that 
given any \( x \in \pre{\omega}{2} \setminus \setLR{0^{( \infty )}} \) there is a unique \( u \subset x \)
that is of the form \( \uu ( r ) \), a crucial fact for proving~\eqref{eq:incompatiblenodesofE}.
\end{remark}

It is easy to check that
\begin{equation*}
 r \leq \mu \left (O ( r ) \right ) = 1 - 2^{ k ( r ) } < 1  ,
\end{equation*}
and that for any measurable set \( B \)
\begin{equation}\label{eq:O(r)increasing}
 r < r' \IMPLIES  \mu \left ( O ( r ) \cup \uu ( r ) \conc B  \right ) \leq 
\mu \left ( O ( r' ) \cup \uu ( r' ) \conc B  \right ) .
\end{equation}

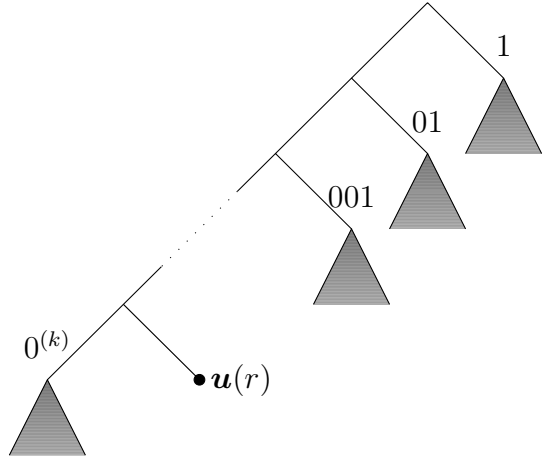
\begin{figure}[t]
   \centering
\begin{tikzpicture}
\node (1) at (1 , -1) [label=90:\( 1 \)]{};
\node (01) at (0 , -2) [label=90:\( 01 \)]{};
\node (001) at (-1 , -3) [label=90:\( 001 \)]{};
\node (00000) at (-5 , -5) [label=90:\( 0^{(k)}\)]{};
\fill [top color=gray, bottom color=gray!60] (1 , -1)--(0.5 , -2)--(1.5 , -2)--cycle;
\fill [top color=gray, bottom color=gray!60] (0 , -2) --(-0.5 , -3)--(0.5 , -3)--cycle;
\fill [top color=gray, bottom color=gray!60] (-1 , -3) --(-1.5 , -4)--(-0.5 , -4)--cycle;
\fill [top color=gray, bottom color=gray!60] (-5 , -5) --(-5.5 , -6)--(-4.5 , -6)--cycle;
\filldraw (-4, -4) -- (-3 , -5)  circle (2pt) node [right] {\(  \uu ( r ) \)};
\draw (0,0)--(-2.5,-2.5)  (-3.5,-3.5) -- (-5 , -5)
(0,0)--(1, -1)--(0.5 , -2)
(1, -1)--(1.5 , -2)
(-1,-1) --(0 , -2) --(-0.5 , -3)
(0 , -2) --(0.5 , -3)
(-2,-2)-- (-1 , -3)--(-1.5 , -4)
(-1 , -3)--(-0.5 , -4)
(-5 , -5) --(-5.5 , -6)
(-5 , -5)--(-4.5 , -6);
\draw[loosely dotted] (-2.5, -2.5) -- (-3.5,-3.5);
\end{tikzpicture}
   \caption{The open set \( O ( r ) \) and the node \(\uu ( r ) \).}
   \label {fig:u(r)}
\end{figure}

We are now ready to define the analogue of \( \overline{A} \).

\subsection{The analogue of \texorpdfstring{\( \overline{A} \)}{A-}}
Fix once and for all 
\[
( r_n )_n \text{ a strictly increasing sequence of reals in \( ( 0 ; 1 ) \) such that } \lim_n r_n =  1 .
\]
\begin{definition}\label{def:Plus}
For \( r \in [ 0 ; 1) \) 
\begin{multline*}
\Plus (A , ( r_n )_n , r ) 
\\
{} = \overline{A} \cup 
\bigcup \setofLR{\overline{s} \conc \eta \conc O ( \max \setLR{ r , r_{\lh ( s ) } \cdot \mu (\LOC{A}{s}) } )}
{ s \in  \pre{ < \omega }{2} \AND \eta \in \setLR{01,10} } .
\end{multline*}
When there is no danger of confusion we will simply write \( \Plus ( A , r ) \)
and if  \( r = 0 \) we write \( \Plus (A )  \).
\end{definition}

\begin{remark}
The naive approach would suggest to define \( \Plus ( A ) \) as the union of \( \overline{A}  \) and the sets
of the form \( \overline{s} \conc i \conc ( 1 - i ) \conc O (  \mu (\LOC{A}{s} ) ) \).
The problem is that if \( A \) has full measure when localized at \( s \), then  
\( O ( \mu (\LOC{A}{s} ) ) \) should be an open set of measure \( 1 \),  
and there would be no room left to move out of \( \Plus ( A ) \).
Thus the values \( \mu ( \LOC{A}{s} ) \) are reduced by the factor \( r_{\lh ( s ) } \).
The parameter \( r \) is needed for Definition~\ref{def:FLAT(A)}, but for 
the time being the reader can safely ignore it and always think of \( r = 0 \).
\end{remark}

Note that 
\begin{equation}\label{eq:FrPlus}
  \Fr ( \Plus ( A , r ) ) \subseteq \overline{ \pre{\omega}{2} } ,
\end{equation}
hence \( \Fr \Plus ( A , r ) \) is  null.
The set of exit nodes for \( \Plus ( A , r ) \) is the set 
\[ 
\mathcal{E}^{\Plus} ( A , ( r_n )_n , r ) = \mathcal{E}^{\Plus} ( A , r ) 
\]
of all nodes of the form 
\[
\overline{s} \conc i \conc ( 1 - i ) \conc \uu ( \max \setLR{r , r_{\lh ( s )} \cdot \mu ( \LOC{A}{s} ) } ) 
\]
and let 
\[
\mm ( s , r ) = \mm ( s ) = \lh ( \uu ( \max \setLR{r , r_{\lh ( s )} \cdot \mu ( \LOC{A}{s} ) } ) ) 
\]
so that by construction 
\[ 
\mu ( O ( \max \setLR{r , r_{\lh ( s )} \cdot \mu ( \LOC{A}{s} ) } ) )
= 1 - 2^{- \mm ( s )} .
\]
Note that \( z \in \pre{\omega}{2} \) is not in \( \Plus (A , r )  \) if and only if either
\begin{itemize}
\item
\( z = \overline{x} \) and \( x \notin A\), or else
\item
\( z \supset e \) for some unique \( e \in \mathcal{E}^{\Plus} ( A , r ) \),
that is: \( z \) \emph{exits} from  \( \overline{A} \) through \( e \), 
hence the reason for the name \emph{exit nodes}.
\end{itemize}

\subsection{The analogue of \texorpdfstring{\( A\Wsum  B\)}{A+B}}\label{subsec:A+B}
\begin{definition}\label{def:Sum}
For \( r \in [ 0 ; 1 ) \) let 
\[ 
\Sum ( B , A , (r_n )_n , r ) = \Sum ( B , A , r ) \equalsdef
\Plus ( A , r ) \cup \bigcup_{e \in \mathcal{E}^{\Plus} ( A , r )} e \conc B .
\]
\end{definition}
Note that for all \( s \in  \pre{ <  \omega }{2} \)
\begin{equation*}
\LOC{ \Plus (A , ( r_n)_n , r )}{\bar{s}}   = \Plus (\LOC{A}{s}, (r_n)_{n \geq \lh ( s )} , r ) 
\end{equation*}
and
 \begin{equation*}
\begin{split}
	 \LOC{\Sum ( B , A , (r_n)_n , r )}{\overline{s}} & = \LOC{\Plus ( A , (r_n )_n , r )}{\overline{s}}
	\cup \bigcup \setofLR{e \conc B}{ \overline{s} \conc e \in \mathcal{E}^{\Plus} ( A , r )} 
\\
& = \Sum ( B , \LOC{A}{s} , (r_n)_{n \geq \lh ( s )} , r ).
\end{split}
\end{equation*}
Therefore for any \( s \in  \pre{ < \omega }{2} \) and any \( i \in 2 \), 
\begin{equation}\label{eq:measurelocalizationSum}
 \mu ( \LOC{ \Sum ( B , A , r ) }{\overline{s} \conc i \conc ( 1 - i ) } )
=  \mu ( O ( \max \setLR{ r ,  r_{\lh ( s )}
	 \cdot \mu (\LOC{A}{s} ) } ) ) 
	 +   \frac{\mu ( B )}{ 2^{\mm ( s )}} 
\leq 1 .
\end{equation}
As \( \overline{ \pre{\omega}{2} } \) is null, 
then~\eqref{eq:series1} and~\eqref{eq:measurelocalizationSum} imply that 
\begin{equation}\label{eq:mu(Plus)}
\begin{split}
\hskip -0.7cm \mu \left ( \Sum ( B , A , r ) \right ) & 
= \smash{\sum_{s \in  \pre{< \omega }{2} } }
	2^{-2 \lh ( s ) - 2} \Bigl [ \mu (\LOC{\Sum ( B , A , r )}{\overline{s}\conc 01} )
\\
&\hphantom{{} = \smash{\sum_{s \in  \pre{< \omega }{2} } }
	2^{-2 \lh ( s ) - 2} \Bigl [} {} + \mu (\LOC{ \Sum ( B , A , r ) }{\overline{s}\conc 10}  ) \Bigr ] 
\\
& \leq \frac{1}{2} \Bigl [ \mu ( O ( \max \setLR{ r ,  r_0
	 \cdot \mu ( A) } ) ) + \frac{ \mu (B)}{2^{\mm (\emptyset )  }} \Bigr ]
\\
& \qquad\qquad \qquad{}+  \sum_{s \in  \pre{< \omega }{2}\setminus \setLR{ \emptyset} } 
	 2^{- 2 \lh ( s) -1} 
\\
& = \frac{1}{2} \mu ( O ( \max \setLR{ r ,  r_0
	 \cdot \mu ( A) } ) ) + \frac{ \mu (B)}{2^{\mm (\emptyset ) + 1 }} +  \frac{1}{2}.
\end{split}
\end{equation}
Note that if \( \mu ( B ) < 1 \) then the inequality in~\eqref{eq:measurelocalizationSum}
and hence the one in~\eqref{eq:mu(Plus)} are strict.
Since \( \Plus ( A , r ) = \Sum ( \emptyset , A , r ) \) we obtain an upper 
bound for the measure of \( \Plus ( A , r ) \):
if \( m \) is least such that \( r , r_0 \cdot \mu ( A ) \leq 1-2^{- m} \) so that 
\( \mu ( O ( \max \setLR{ r , r_0 \cdot \mu ( A ) } ) ) = 1 - 2^{-m} \), then 
\begin{equation}\label{eq:upperbound}
  \mu \left ( \Plus (A , r ) \right )  < 1 - 2^{ - m  - 1} .
\end{equation}
Since \( \max\setLR{r , r_{\lh ( s )} \cdot \mu (\LOC{A}{s} )} \geq r , r_0  \cdot \mu (\LOC{A}{s} ) \),
we obtain two lower bounds for the measure of \( \Plus (A , r ) \).
The first one, which is only of interest when \( r > 0 \), is
\begin{equation*}
 \mu \left ( \Plus (A , r ) \right )  \geq r \cdot\sum_{s \in  \pre{< \omega  }{2} } 2^{- 2 \lh ( s ) - 1}
= r  ,
\end{equation*}
and therefore 
\begin{equation}\label{eq:estimate1}
\mu \bigl (\LOC{ \Plus (A , r )}{\overline{s}} \bigr )  \geq r 
\end{equation}
for any \( s \in  \pre{ <  \omega }{2} \).
For the second one, by~\eqref{eq:series2} we have
\begin{equation}\label{eq:estimate2}
\mu \left ( \Plus (A , r ) \right )  \geq  \sum_{s \in  \pre{< \omega }{2} } 
 \frac{ r_0 \cdot \mu (\LOC{A}{s} ) }{2^{ 2 \lh ( s ) + 1}} = r_0 \cdot \mu \left ( A \right )  .
\end{equation}

Then~\eqref{eq:upperbound} and~\eqref{eq:estimate2} imply that
\begin{equation}\label{eq:SpAT-regular}
 r_{\lh ( s )}\cdot \mu ( \LOC{A}{s} ) \leq
  \mu \bigl ( \LOC{ \Plus (A , r )}{\bar{s}} \bigr ) 
\leq 1 - 2^{ - m - 1} ,
\end{equation}
where  \( m \) is least such that \( r , r_{\lh ( s )} \cdot \mu ( \LOC{A}{s} ) \leq 1-2^{-m} \).
Therefore for \( i \in 2 \)
\begin{equation}
 \begin{split}\label{eq:SpAT-regular2}
\mu \bigl ( \LOC{ \Plus (A , r )}{\bar{s}\conc i} \bigr )  & 
= \frac{1}{2} \cdot \mu \bigl ( \LOC{ \Plus (A , r )}{\bar{s}\conc ii} \bigr )
 + \frac{1}{2} \cdot \mu \bigl ( \LOC{ \Plus (A , r )}{\bar{s}\conc i \conc ( 1 - i)} \bigr )
\\
 & \geq \frac{ r_{\lh ( s ) + 1}\cdot \mu \bigl ( \LOC{A}{s\conc i} \bigr ) +
\max \setLR{r , r_{\lh ( s )} \cdot \mu ( \LOC{A}{s} )} }{2} .
\end{split}
\end{equation}

\begin{proposition}\label{prop:A^+equivalentPlus(A)}
If \( r \in [ 0 ; 1) \) then 
\[
\Sum ( B , A , r ) \Wequiv  B \Wsum  A .
\]
\end{proposition}

\begin{proof}
Player \( \II \) wins \( \GL ( B \Wsum  A , \Sum ( B , A , r ) ) \) via the following strategy:
\begin{quote}
As long as \( \I \)'s positions are of the form \( \overline{s}\) or \( \overline{s} \conc i \conc (1 - i) \)
with \( i \in 2 \), then \( \II \) copies \( \I \)'s moves.
If ever \( \I \) reaches a position of the form \( \overline{s} \conc i \conc (1 - i ) \), then 
\( \II \) plays \( \uu ( \max \setLR{r , r_{\lh ( s )} \cdot \mu ( \LOC{A}{s} ) } ) \)
reaching the exit node extending his current position, and then copies \( \I \)'s moves.
\end{quote}
Player \( \II \) has a winning strategy in the game 
\( \GW (  \Sum ( B , A , r ) , B \Wsum  A )\):
\begin{quote}
As long as \( \I \)'s positions are of the form \( \overline{s}\) or \( \overline{s} \conc i \conc (1 - i) \)
with \( i \in 2 \), then \( \II \) copies \( \I \)'s moves.
If ever \( \I \) reaches a position of the form \( \overline{s} \conc i \conc (1 - i) \), then 
\( \II\) passes until \( \I \) commits himself by either reaching the exit node extending
his current position, or else reaches a position incompatible with such exit node:
then in the first case \( \II \) copies \( \I\)'s moves, and in the second case \( \II \) plays 
a sequence  in \( B \).
\end{quote}
Therefore \( \Sum ( B , A , r ) \Wequiv B \Wsum  A  \).
\end{proof}

By a similar argument one could show that \( \overline{A} \Wequiv  \Plus (A , r )  \)
if the set \( A \) is dense,  but we have no use for this fact.

If \(  x \in \Phi  ( A ) \) then 
\( r_n \cdot \mu \bigl ( \LOC{A}{x \restriction n} \bigr ) \to 1\),
so \( \mu \bigl ( \LOC{ \Plus (A , r )}{\bar{x}\restriction 2n} \bigr ) \to 1 \) 
by~\eqref{eq:SpAT-regular}, and since by~\eqref{eq:SpAT-regular2}
\[
\mu \bigl ( \LOC{ \Plus (A , r )}{\bar{x}\restriction 2n + 1} \bigr ) \geq 
\frac{ r_{n + 1}\cdot \mu \bigl ( \LOC{A}{x \restriction n + 1} \bigr ) +
\max \setLR{r , r_n \cdot \mu ( \LOC{A}{ x \restriction n } )} }{2} \to 1 ,
\]
then \( \bar{x} \in \Phi  \bigl (  \Plus (A , r ) \bigr ) \).
Conversely, if  \( x \notin \Phi  ( A ) \) pick an increasing sequence \( n_k \) 
such that \( \sup_k \mu \left ( \LOC{A}{x \restriction n_k} \right )  < 1\),
hence there is an \( m \) such that for all \( k \)
\[
 r , r_{n_k} \cdot \mu ( \LOC{A}{x \restriction n_k } ) < 1-2^{-m}
\]
thus by~\eqref{eq:SpAT-regular}
\[
 \mu \bigl ( \LOC{\Plus (A , r )}{ \overline{x \restriction n_k} } \bigr ) \leq 1 - 2^{-m - 1} 
\]
and therefore \( \bar{x} \notin \Phi ( \Plus (A , r ) ) \).
Therefore we have shown that 
\[
 x \in \Phi  ( A ) \IFF 
\overline{x} \in \Phi  ( \Plus ( A , r  ) )  .
\]
If \( x \in  \pre{\omega}{2}\setminus  \overline{ \pre{\omega}{2} } \)
it is easy to check that 
\( x \in  \Plus ( A , r  ) \iff x \in \Phi \left (  \Plus ( A ,  r  )  \right ) \), so that
\begin{equation}\label{eq:T-regular=>T-regular}
A \text{ is \( \mathcal{T} \)-regular} \IMPLIES \Plus (A , r ) \text{ is \( \mathcal{T} \)-regular.} 
\end{equation}

\begin{proposition}\label{prop:Sumisregular}
If \( A \) and \( B \) are \( \mathcal{T} \)-regular, then so is \( \Sum ( B , A , r ) \).

Moreover if \( A, B \in \ran ( \Phi \restriction \bSigma^{0}_{1} ) \cap \ran ( \Phi \restriction \bPi^{0}_{1} ) \),
then \( \Sum ( B , A , r )\in \ran ( \Phi \restriction \bSigma^{0}_{1} ) \cap \ran ( \Phi \restriction \bPi^{0}_{1} ) \).
\end{proposition}

\begin{proof}
Let \( x \in \Sum ( B , A , r ) \).
If \( x \in \Plus ( A , r ) \), then \( x \in \Phi ( \Plus ( A , r ) )\) by~\eqref{eq:T-regular=>T-regular}
hence \( x \in \Phi ( \Sum ( B , A , r ) )\) by monotonicity of \( \Phi \).
If instead \( x = e \conc b \) with \( e \in \mathcal{E}^{\Plus} ( A , r ) \) and \( b \in B \),
then \( x\in \Phi ( \Sum ( B , A , r ) ) \) as \( b \in \Phi ( B ) \).
Therefore \( \Sum ( B , A , r ) \subseteq \Phi ( \Sum ( B , A , r ) )\).

Conversely, suppose \( x \notin \Sum ( B , A , r ) \), which means that either
\begin{enumerate-(A)}
\item\label{prop:Sumisregular-a}
\( x = \overline{y}\) with \( y\notin A \), or else
\item\label{prop:Sumisregular-b}
\( x = e\conc y\) with \( e \in \mathcal{E}^{\Plus} ( A , r )\)
and \( y \notin B \).
\end{enumerate-(A)}

If~\ref{prop:Sumisregular-a} holds pick an increasing sequence \( ( n_k )_k \) such that 
\( \sup_k \mu (\LOC{A}{y \restriction n_k } ) < 1 \), and let 
\[ 
\tilde{r} = \max \setLR{ r ,  \textstyle\sup_k  r_{n_k} \cdot\mu (\LOC{A}{y \restriction n_k } )  } 
\qquad\text{and}\qquad  \tilde{u} = \uu ( \tilde{r} ) .
\]
We must show that there is a fixed \( m > 0 \) such that for all \( k \)
\[ 
\begin{split}
\mu ( \LOC{\Sum ( B , A , ( r_n )_n , r )}{x  \restriction 2 n_k } )  & =
\mu (\Sum ( B , \LOC{A}{y \restriction n_k} , ( r_n )_{n \geq n_k} , r ) ) 
\\
& < 1 - 2^{-m - 1}
\end{split}
\]
hence \( x \notin \Phi  ( \Sum ( B , A , r ) ) \).
Choose \( m \) such that 
\[ 
\mu ( O ( \tilde{r} ) ) + 2^{-\lh ( \tilde{u} )}\cdot \mu ( B )  < 1 - 2^{-m} .
\] 
To simplify the notation let 
\[ 
S_k = \Sum  ( B , \LOC{A}{y\restriction n_k} , (r_n)_{n \geq n_k} , r )
\qquad\text{and}\qquad 
 \rho_k = r_{n_k} \cdot\mu (\LOC{A}{y \restriction n_k } ) .
\]
Arguing as in~\eqref{eq:mu(Plus)} and~\eqref{eq:upperbound} and using~\eqref{eq:O(r)increasing}
\begin{align*}
\MoveEqLeft
 \mu \bigl(\Sum (B , \LOC{A}{y\restriction n_k} , (r_n)_{n \geq n_k} , r  ) \bigr )  
\\
& =  \frac{1}{2} \Bigl [ \mu ( O ( \max \setLR{ r , \rho _k } ) ) 
+ \frac{ \mu ( B ) }{2^{ \lh \uu ( \max \setLR{ r , \rho_k } )  }} \Bigr ]
+\sum_{s \in  \pre{< \omega }{2} \setminus\setLR{\emptyset} } 
	 2 \cdot  \mu \bigl (\overline{s} \conc 01 \conc 
	 \LOC{(S_k)}{\overline{s} \conc 01 } \bigr ) 
 \\
& \leq \frac{1}{2} \Bigl [ \mu ( O ( \tilde{r} ) ) + \frac{ \mu ( B ) }{2^{ \lh  \uu ( \tilde{r} ) }} \Bigr ]
+ {\sum_{s \in  \pre{< \omega }{2} \setminus\setLR{\emptyset} } }
	2 \cdot  \mu \bigl (\overline{s} \conc 01 \conc 
	 \LOC{(S_k)}{\overline{s} \conc 01 } \bigr ) 
\\
& < \frac{1}{2} ( 1 - 2^{-m} ) + \smash{\sum_{s \in  \pre{< \omega }{2} \setminus\setLR{\emptyset} } 
 	2^{-2 \lh ( s )- 1}}
\\
& = 1 - 2^{-m - 1 }
\end{align*}
which is what we had to prove.

If instead~\ref{prop:Sumisregular-b} holds then 
\( \LOC{\Sum ( B , A , r ) }{x \restriction\lh ( e ) + n } = \LOC{B}{y \restriction n}\) 
for all \( n \), hence \( y \notin B = \Phi (B ) \) and therefore \( x \notin \Phi ( \Sum ( B , A , r ) ) \).

Thus either way \( x \notin \Phi ( \Sum ( B , A , r ) ) \), and this completes the proof that 
\( \Sum ( B , A , r ) \) is \( \mathcal{T}\)-regular.

Suppose now \( A , B \in \ran ( \Phi \restriction \bSigma^{0}_{1} ) \cap 
\ran ( \Phi \restriction \bPi^{0}_{1} ) \) towards proving that 
\[ 
\Sum ( B , A , r ) \in \ran ( \Phi \restriction \bSigma^{0}_{1} ) \cap 
\ran ( \Phi \restriction \bPi^{0}_{1} ) .
\]
By~\eqref{eq:D(closed)} and regularity of \( \Sum ( B , A , r ) \),
it is enough  to show that \( \mu (\Fr B ) = 0 \)
implies that \( \mu ( \Fr \Sum ( B , A , r ) ) = 0 \).
Since 
\[
 \Fr ( \Sum ( B , A , r ) ) \setminus \overline{ \pre{\omega}{2} } 
= \textstyle\bigcup_{e \in \mathcal{E}^{\Plus} (A , r)} e \conc \Fr B
\]
is a countable union of null sets and \( \overline{ \pre{\omega}{2}}  \) is null, 
the result follows. 
\end{proof}

Since \( \Plus ( A , r ) = \Sum ( \emptyset , A , r ) \), we obtain at once
\begin{corollary}\label{cor:Plusregular}
If \( A \) is \( \mathcal{T}\)-regular, then so is \( \Plus ( A , r )\).
Moreover if \( A \in \ran ( \Phi \restriction \bSigma^{0}_{1} ) \cap \ran ( \Phi \restriction \bPi^{0}_{1} ) \)
then \( \Plus ( A , r ) \in \ran ( \Phi \restriction \bSigma ^{0}_{1} ) \cap \ran ( \Phi \restriction \bPi^{0}_{1} ) \).
\end{corollary}

\subsection{The analogues of \texorpdfstring{\( A^\natural\)}{A-natural} and \texorpdfstring{\( A ^\flat\)}{A-flat}.}
All the constructions seen so far, as well as the ones in this section, are based 
on the idea of attaching a set to a node of a tree ---
but sometimes the set needs to be padded before attaching it.

\begin{definition}
For \( n > 0 \), the \markdef{\( n \)-th padding} of a set \( A \subseteq \pre{\omega}{2} \) is
\[
\begin{split}
 \mathrm{P}_n ( A ) &  = \left (1^{ ( n ) } \conc A \right ) \cup \bigcup \setofLR{\Nbhd_s}
{ s\in  \pre{ n}{2}\wedge s \neq 0^{ ( n ) } , 1^{ ( n ) }}
\\
 &  = \left (1^{ ( n ) } \conc A \right ) \cup \left ( \pre{\omega}{2} \setminus ( \Nbhd_{0^{ ( n ) } } 
 \cup \Nbhd_{1^{ ( n ) }} ) \right ) .
\end{split}
\]
\end{definition}
Thus \( \mathrm{P}_1 ( A ) = 1 \conc A \) and \( \LOC{ ( \mathrm{P}_n ( A ) ) }{1^{ ( n ) }} = A \).
Moreover
\begin{equation}\label{eq:padding}
\mu ( \mathrm{P}_n ( A ) ) = 1 - 2^{ - n } \bigl ( 2 - \mu ( A )  \bigr ) . 
\end{equation}

We start with defining \( \NATURAL ( A ) \), the analogue of \( A^\natural\).
First define 
\[ 
\mathcal{E}^{\NATURAL} ( A ) = \bigcup_{n> 0} \mathcal{E}_n ^{\NATURAL} ( A ) ,
\]
the set of all exit nodes for \( \NATURAL ( A ) \), where 
\(  \mathcal{E}^{\NATURAL}_n (A) \) is the set of all sequences of the form
\[
v_1 \conc 1\conc  v_2 \conc 1\conc \dots \conc v_n 
\]
where 
\[
v_i = \overline{s_i} \conc \eta_i  \conc \uu ( r_{\lh ( s_i )} \cdot \mu ( \LOC{A}{ s_i } ) ) 
\]
and \( s_1 , \dots , s_n \in  \pre{< \omega }{2} \) and \( \eta _1 , \dots , \eta _n \in \setLR{01 , 10} \). 
If \( e \in \mathcal{E}^{\NATURAL}_n (A) \) and 
\( e' \in \mathcal{E}^{\NATURAL}_{n'}  (A)\) then exactly one of the disjuncts below holds:
\begin{equation}\label{eq:incompatiblenodesofE}
\left (e \subset e' \AND n < n' \right ) \OR \left ( e' \subset e \AND n' < n\right ) \OR 
\left ( e = e' \AND n = n' \right ) \OR \left ( e \perp e' \right ) . 
\end{equation}
In particular, the elements in \( \mathcal{E}_n^{\NATURAL} ( A ) \) are pairwise incompatible, and 
\[
 \FORALL{e \in \mathcal{E}^{\NATURAL}_n ( A )} \FORALL{j < n} 
 \EXISTSONE{e' \in \mathcal{E}_j^{\NATURAL} ( A )}  ( e' \subset e  )
\]
so that if \( x \in \pre{\omega}{2} \) passes through infinitely many points 
of  \( \mathcal{E}^{\NATURAL} ( A ) \) then 
\begin{equation}\label{eq:xcrossesinfinitelymanylayers}
  x = \bigcup_{n} e_n
\end{equation}
with \( e_n \in \mathcal{E}_n^{\NATURAL} (A) \) and \( e_1 \subset e_2 \subset e_3 \subset \cdots \)\,.

\begin{definition}\label{def:NAT(A)}
\(\NATURAL ( A ) =   \bigcup_{ e \in \mathcal{E} ^{\NATURAL}  (A) }
e  \conc 1 \conc \Plus ( A ) \).
\end{definition}

Two remarks on \(  \NATURAL  (A) \)'s definition are in order.
\begin{remarks}
\begin{enumerate-(a)}
\item
\( \NATURAL  (A) \) is obtained by attaching the \( 1 \)-padding of \( \Plus (A) \)
to each \( e \in \mathcal{E}^{\NATURAL} ( A ) \), hence it
can be seen as a tree of sets: 
to move from a set at level \( n \)
to a set at level \( n + 1 \) we exit level \( n \) by following a node of the form 
\( \overline{s}\conc i \conc (1 - i ) \conc \uu ( r_{\lh ( s )} \cdot \mu (\LOC{A}{s} ) ) \conc 1 \)
--- choosing different strings \( s \) will take us to different nodes at level \( n + 1 \).
The digit `\( 1 \)' that separates different levels will ensure that every \( x \) as in~\eqref{eq:xcrossesinfinitelymanylayers}
will not have density \( 1 \) 
in \( \NATURAL (A) \), implying \( \mathcal{T} \)-regularity.
\item
Given any \( x \in \pre{\omega}{2} \) we have five mutually exclusive possibilities:
\begin{enumerate-(i)}
\item
\( x \) does not extend any node of \( \mathcal{E}^{\NATURAL} (A) \), hence \( x \notin \NATURAL  (A) \),
\item
\( x \) extends infinitely many nodes of \( \mathcal{E}^{\NATURAL} (A) \), hence 
it is of the form~\eqref{eq:xcrossesinfinitelymanylayers} and it 
is a branch of the tree of sets.
Also in this case  \( x \notin \NATURAL  (A) \).
\item
\( x \) extends  \( e \conc 0 \) with \( e \in \mathcal{E}^{\NATURAL} ( A ) \).
Then \(  x \notin \NATURAL  (A) \) by 
part~\ref{lem:opendisjointfromNAT-b} of Lemma~\ref{lem:opendisjointfromNAT} below.
\item
\( x \) is of the form \( e \conc 1 \conc \overline{y} \), and
\( e \) is the largest exit node contained in \( x \).
Then \( x \in \NATURAL  (A) \iff y \in A \).
\item
\( x \) extends \( e \conc 1 \conc \overline{s} \conc \eta \) for some \( \eta \in \setLR{ 01 , 10 }\), and
\( e \) is the largest exit node contained in \( x \).
By maximality  \( x \supset e \conc 1 \conc \overline{s} \conc \eta \conc v \) for some 
\( v \perp \uu ( r_{\lh ( s )} \cdot \mu ( \LOC{A}{s} )) \) hence \( x \in \NATURAL  (A) \).
\end{enumerate-(i)}
\end{enumerate-(a)}
\end{remarks}

\begin{lemma}\label{lem:opendisjointfromNAT}
Let \( e , e' \in \mathcal{E}^{\NATURAL} (A) \):
\begin{enumerate-(a)}
\item\label{lem:opendisjointfromNAT-a}
If \( e \subset e' \) then \( \left (e \conc 1 \conc \Plus ( A ) \right ) \cap  \Nbhd_{e'} = \emptyset \), 
hence 
\[ 
e \neq e' \IMPLIES \left ( e \conc 1 \conc \Plus ( A ) \right ) \cap 
\left ( e' \conc 1 \conc \Plus ( A ) \right ) = \emptyset .
\]
\item\label{lem:opendisjointfromNAT-b}
\(  \FORALL{e \in \mathcal{E}^{\NATURAL}(A) } 
\left ( \Nbhd_{e \conc 0} \cap \NATURAL  (A) = \emptyset \right ) \).
\end{enumerate-(a)}
\end{lemma}

\begin{proof}
\ref{lem:opendisjointfromNAT-a}
Let \( e = v_1 \conc 1 \conc v_2 \conc 1 \conc \dots \conc 1 \conc v_n \) and 
\( e' = e \conc 1\conc v_{n+1} \conc 1\conc \dots \conc 1 \conc v_{n+k}\).
Towards a contradiction suppose that there is an element of the Cantor space of the form \( e \conc 1 \conc x \) 
with \( x\in \Plus ( A ) \) that belongs to \( \Nbhd_{e' }\), that is 
\[
x = \overline{s_{n+1}} \conc \eta_{n+1} \conc \uu ( r_{\lh ( s_{n+1} ) } 
\cdot \mu ( \LOC{A}{s_{n + 1} }) )  \conc  y ,
\]
for some \( y \).
As \( x \notin \overline{ A} \) then \( x \) belongs to some 
\( \overline{s} \conc i \conc ( 1 - i ) \conc O ( r_{\lh ( s )} \cdot \mu ( \LOC{A}{s} )) \)
for some \( s \in  \pre{ < \omega }{2}\) and \( i \in 2 \).
This implies that \( s = s_{n+1} \) and \( \eta_{n+1} = i \conc ( 1 - i ) \) and
\[ 
\uu ( r_{\lh ( s_{n + 1}) }\cdot \mu ( \LOC{A}{s_{n + 1} }) )  \conc y 
\in O (  r_{\lh ( s_{n + 1} ) } \cdot \mu ( \LOC{A}{s} ) ) ,
\]
which contradicts Definition~\ref{def:O(r)}.

\ref{lem:opendisjointfromNAT-b}
It is enough to show that 
\[
\FORALL{e , e' \in \mathcal{E}^{\NATURAL} (A) } 
\left ( \Nbhd_{e' \conc 0} \cap \left ( e \conc 1 \conc \Plus ( A , r )  \right ) = \emptyset  \right ) .
\]
If \( e' \subseteq e \) then \(e' \conc 0\perp e \conc 1 \) hence the result holds at once.
If instead \( e \subset e' \) we apply part~\ref{lem:opendisjointfromNAT-a}.
\end{proof}

We now construct \( \FLAT ( A ) \), the analogue of \( A^\flat \).
First define 
\[ 
\mathcal{E}^{\FLAT} ( A )  = \bigcup_{n> 0} \mathcal{E}^{\FLAT} _n  ( A ) ,
\]
the set of all exit nodes of \( \FLAT ( A ) \), where
\(  \mathcal{E}^{\FLAT} _n (A) \) is the set of all sequences of the form
\[
w_1 \conc 1^{( \hh ( 1 ) )} \conc w_2 \conc 1^{( \hh ( 2 ) )} \conc \dots \conc 1 ^{( \hh ( n - 1 ) )} \conc w_n
\]
with
\[
w_i = \overline{s_i} \conc \eta_i \conc \uu (\max \setLR{r_i, r_{\lh ( s_i )} \cdot \mu ( \LOC{A}{ s_i} )})
\]
and \( s_1 , \dots , s_n \in  \pre{ < \omega }{2} \), \( \eta_1 , \dots , \eta_n \in \setLR{01 , 10}\) and
\[
\hh ( i )  = \min k \left ( 1 - 2^{ - k + 1 } \geq r_i \right ) .
\]
Notice that the elements of  \( \mathcal{E}^{\FLAT} ( A ) \) differ from
the ones of \( \mathcal{E}^{\NATURAL} ( A ) \) in that the \( \uu\) 
part is different and we use \( 1^{ ( \hh ( r_i ) )}\) to separate the blocks.
We leave it to the reader to check that the elements of \( \mathcal{E}^{\FLAT}(A) \)
 have properties similar to the ones in \( \mathcal{E}^{\NATURAL} (A)\)
--- in particular~\eqref{eq:incompatiblenodesofE} holds if 
\( e \in \mathcal{E}^{\FLAT}_n (A) \) and 
\( e' \in \mathcal{E}^{\FLAT}_{n'}  (A)\).

\begin{definition}\label{def:FLAT(A)}
\[ 
\FLAT ( A )  =  \Bigl ( \bigcup_{n>0} \bigcup_{ e \in \mathcal{E}_n^{\FLAT}(A) }  
e  \conc \mathrm{P}_{\hh ( n )} ( \Plus ( A ,  r_n ) ) \Bigr )
\cup \setofLR{x}{\exists^\infty e \in \mathcal{E}^{\FLAT} (A) \left ( e \subseteq x \right )} .
\]
\end{definition}

\begin{remarks}
\begin{enumerate-(a)}
\item
\( \FLAT ( A ) \) is the disjoint union of two sets.
The first one, like the case of \( \NATURAL ( A ) \), can be seen as a tree of sets hence it is stratified in layers,
the second one is the set of all branches through this tree.
\item
Given any \( x \in \pre{\omega}{2} \) we have six mutually exclusive possibilities:
\begin{enumerate-(i)}
\item
\( x \) does not extend any node of \( \mathcal{E}^{\FLAT} (A) \), hence \( x \notin \FLAT  (A) \),
\item
\( x \) extends infinitely many nodes of \( \mathcal{E}^{\FLAT} (A) \), hence 
it is in \( \FLAT ( A ) \).
In this case we will see that \( x \in \Phi ( \FLAT ( A ) ) \).
\item
\( x \) extends  \( e \conc 0^{( \hh ( n ) )} \) with \( e \in \mathcal{E}^{\FLAT}_n ( A ) \).
Then \(  x \notin \FLAT  (A) \) by 
part~\ref{lem:opendisjointfromFLAT-b} of Lemma~\ref{lem:opendisjointfromFLAT} below.
\item
\( x \) is of the form \( e \conc 1^{( \hh ( n ) ) } \conc \overline{y} \), and
\( e \in \mathcal{E}_n^{\FLAT} ( A ) \) is the largest exit node contained in \( x \).
Then \( x \in \FLAT ( A ) \iff y \in A \).
\item
\( x \) extends \( e \conc 1^{(\hh ( n ) )} \conc \overline{s} \conc \eta \) for some \( \eta \in \setLR{ 01 , 10 }\), and
\( e \) is the largest exit node contained in \( x \).
By maximality  \( x \supset e \conc 1^{(\hh ( n ) )} \conc \overline{s} \conc \eta \conc v \) for some 
\( v \perp \uu (\max \setLR{r_{ n + 1 } , r_{\lh ( s )} \cdot \mu ( \LOC{A}{s} ) } ) \) hence \( x \in \FLAT  (A) \).
\item
\( x \) extends \( e \conc t \) with \( e \in \mathcal{E}^{\FLAT}_n ( A ) \) and 
\( t \in  {}^{ \hh ( n ) } 2 \setminus \setLR{ 0^{( \hh ( n ))} , 1^{ (\hh ( n ) )}} \).
Then \( x \in \FLAT ( A ) \cap \Phi ( \FLAT ( A ) ) \).
\end{enumerate-(i)}

\end{enumerate-(a)}
\end{remarks}

The following is proved as Lemma~\ref{lem:opendisjointfromNAT}.
\begin{lemma}\label{lem:opendisjointfromFLAT}
Let \( e \in \mathcal{E}^{\FLAT}_n (A) \) and \( e' \in \mathcal{E}^{\FLAT}_{n'} (A) \):
\begin{enumerate-(a)}
\item\label{lem:opendisjointfromFLAT-a}
If \( e \subset e' \) then \( \left (e \conc 1^{ ( \hh ( n ) ) } \conc \Plus ( A , r_n ) \right ) \cap  \Nbhd_{e'} = \emptyset \), 
hence 
\[ 
e \neq e' \IMPLIES \left ( e \conc 1^{ ( \hh ( n ) ) } \conc \Plus ( A , r_n ) \right ) \cap 
\left ( e' \conc 1^{ ( \hh ( n' ) ) } \conc \Plus ( A , r_{n'} ) \right ) = \emptyset .
\]
\item\label{lem:opendisjointfromFLAT-b}
\(  \FORALL{e \in \mathcal{E}^{\FLAT} ( A ) } 
\left ( \Nbhd_{e \conc 0^{ ( \hh ( n ) )  }} \cap \FLAT ( A ) = \emptyset \right ) \).
\end{enumerate-(a)}
\end{lemma}

Fix an \( s \in  \pre{ < \omega }{2}\).
Since 
\( \LOC{\FLAT ( A )}{e \conc 1^{( \hh ( n ) )} \conc \overline{s}} \supseteq \LOC{\Plus (A , r_n )}{\overline{s}} \) 
when \( e \in \mathcal{E}_n^{\FLAT} (A)\), then~\eqref{eq:estimate1} implies
\begin{equation*}\label{eq:localizedFlat0}
\FORALL{e \in \mathcal{E}_n^{\FLAT} (A) } 
\bigl ( \mu ( \LOC{\FLAT ( A )}{e \conc 1^{( \hh ( n ) )} \conc \overline{s}} ) \geq r_n \bigr ) .
\end{equation*}

\begin{lemma}\label{lem:intermediatevalues}
\( \FORALL{e \in \mathcal{E}_n^{\FLAT} (A) } \FORALL{k \leq \hh ( n ) } 
\bigl (  \mu ( \LOC{\FLAT ( A )}{e \conc 1^{( k)} } ) \geq r_n\bigr ) \).
\end{lemma}

\begin{proof}
The case \( k = \hh ( n ) \) is the preceding inequality with \( s = \emptyset \), and for \( 0 < k < k' \) use that 
\(  \mu (  \LOC{\FLAT ( A )}{e \conc 1^{( k)} } ) \geq \mu (  \LOC{\FLAT ( A )}{e \conc 1^{( k' )} }) \).
If \( k = 0 \) then use~\eqref{eq:padding}.
\end{proof}

By Definition~\ref{def:Plus} 
\[ 
\mu \bigl ( \LOC{\Plus ( A, r_n ) }{ \overline{s} \conc i \conc ( 1 - i)} \bigr ) =
 \mu ( O (\max \setLR{r_n , r_{\lh ( s )}\cdot \mu ( \LOC{A}{s} )}) ) \geq r_n , 
\]
hence, arguing as in~\eqref{eq:SpAT-regular2}, 
\( \mu ( \LOC{ \Plus (A , r_n )}{\overline{s}\conc i} ) \geq r _n \) too. 
Therefore for all \( e \in \mathcal{E}_n^{\FLAT} (A)  \), 
all \( s \in  \pre{< \omega }{2}  \), and \( i \in 2 \)
\begin{equation*}\label{eq:localizedFlat}
  \mu ( \LOC{\FLAT ( A )}{e \conc 1^{( \hh ( n ) )} \conc \overline{s} \conc i } ) ,
\mu ( \LOC{\FLAT ( A )}{e \conc 1^{( \hh ( n ) )} \conc \overline{s} \conc i \conc ( 1 - i )} ) \geq r_n .
\end{equation*}
To simplify the notation, let \( \eta = i \conc ( 1 - i )\) and 
\( e' = e \conc 1^{( \hh ( n ) )} \conc \overline{s} \conc \eta \conc u \in \mathcal{E}_{n + 1} ^{\FLAT } ( A ) \) 
where 
\[ 
u = \uu ( \max \setLR{ r_{n+1} , r_{\lh ( s)}\cdot \mu ( \LOC{A}{s} ) }) .
\]
If  \( v \subset u \) then \( u = 0^{(m)} 1\) and  \(v = 0^{(k)} \) for some \(k \leq m \).
\begin{figure}
   \centering
\begin{tikzpicture}
\fill [top color=gray, bottom color=gray!60] (1 , -1)--(0.5 , -2)--(1.5 , -2)--cycle;
\fill [top color=gray, bottom color=gray!60] (0 , -2) --(-0.5 , -3)--(0.5 , -3)--cycle;
\fill [top color=gray, bottom color=gray!60] (-4 , -4) --(-4.5 , -5)--(-3.5 , -5)--cycle;
\draw (-2 , -4) -- (-2.5 , -5)--(-1.5 , -5)--cycle;
\node  at (-1,-1)[inner sep=1pt, label=0:\( 0\)]{};
\node  at (1,-1)[inner sep=1pt, label=180:\( 1\)]{};
\node  at (0,-2)[inner sep=1pt, label= 180:\( 01\)]{};
\node at (-3,-3)[inner sep=1pt, label= 0:\( 0^{(m-k)}\)]{};
\node at (-2,-4)[inner sep=1pt, label= 0:\( 0^{(m-k)}1\)]{};
\node at (-4,-4)[inner sep=1pt, label= 0:\( 0^{(m-k+1)}\)]{};
\node at (-2,-5)[inner sep=0pt, label=270:\( \LOC{\FLAT(A)}{e'}\)]{};
\draw (0,0)--(-1.5,-1.5)  (-2.5,-2.5) -- (-4 , -4)
(0,0)--(1, -1)--(0.5 , -2)
(1, -1)--(1.5 , -2)
(-1,-1) --(0 , -2) --(-0.5 , -3)
(0 , -2) --(0.5 , -3)
(-3,-3)--(-2,-4)
(-4, -4) -- (-3.5 , -5)
(-4, -4) -- (-4.5 , -5) ;
\draw[loosely dotted] (-1.5, -1.5) -- (-2.5,-2.5);
\draw [thick,decorate,decoration=brace] (2,0) -- (2, -6) ;
\node at (2,-3) [anchor=west] {\( {}= \LOC{\FLAT(A)}{e \conc  1^{( \hh ( n ) )}\conc \overline{s} \conc \eta \conc 0^{(k)}}\)};
\end{tikzpicture}
   \caption{}
   \label{fig:natural}
\end{figure}
Thus (see Figure~\ref{fig:natural})
\[
 \LOC{\FLAT(A)}{e  \conc 1^{( \hh ( n ) )}\conc \overline{s} \conc \eta \conc 0^{(k)}} = 
 \Nbhd_1 \cup \Nbhd_{01} \cup \dots \cup 0^{( m - k )}1\conc \LOC{\FLAT(A)}{e'} \cup \Nbhd_{0^{( m - k +1 )}}
\]
and
\begin{equation*}
\begin{split}
\mu ( \LOC{\FLAT(A)}{e \conc  1^{( \hh ( n ) )} \conc \overline{s} \conc \eta \conc 0^{(k)}} ) 
	& = 1 - \frac{1}{2^{m-k+1}} + \frac{\mu ( \LOC{\FLAT(A)}{e'})}{2^{m-k+1}}
\\
 & \geq 1 - \frac{1}{2^{m-k+1}} + \frac{r_{n+1}}{2^{m-k+1}}
\\
& \geq r_{n+1} .
\end{split} 
\end{equation*}
Therefore
\begin{equation}\label{eq:e_n}
 \FORALL{e_n \in \mathcal{E}_n^{\FLAT} (A)}
 \FORALL{ e_{ n + 1} \in \mathcal{E}_{n + 1}^{\FLAT} (A)}
 \FORALL{t} 
 \bigl ( e_n \subset t \subseteq e_{n + 1}
\IMPLIES  \mu \bigl ( \LOC{\FLAT ( A )}{t} \bigr ) \geq r_{n} \bigr ) . 
\end{equation}

\begin{proposition}\label{prop:NAT<=>nat}
If \( A \neq \emptyset , \pre{\omega}{2} \) then
\[ 
\NATURAL (A)  \Wequiv A^\natural \qquad\text{ and }\qquad \FLAT (A)  \Wequiv A^\flat .
\]
\end{proposition}

\begin{proof}
We first look at \(  \NATURAL (A)\) and \( A^\natural \).
Fix \( e \in \mathcal{E}^{\NATURAL} ( A ) \): we will show that 
\( A^\natural \leql \LOC{\NATURAL ( A )}{ e \conc 1} \), hence 
\( A^\natural \leql \NATURAL ( A ) \).
Player \( \II \) wins \( \GL ( A^\natural , \LOC{\NATURAL (A) }{ e \conc 1 } ) \) by copying
\( \I \)'s moves and playing an appropriate \( \uu\)-node followed by \( 1 \)
whenever \( \I \) breaks a sequence \( \overline{s} \) by playing \( 01\) or \( 1 0 \).
Conversely \( \II \) wins \( \GW ( \NATURAL (A) , A^\natural ) \) as follows:
\begin{quote}
Player \( \II \) enumerates a sequence \( \overline{a} \) with \( a \notin A\), 
until \( \I \) reaches, if ever, a position 
\( \overline{s_1} \conc \eta_1 \conc \uu ( r_{\lh ( s _1 ) } \cdot \mu ( \LOC{A}{s_1} ) ) \conc 1 \) 
with \( \eta_1 \in \setLR{01 ,10}\).
Suppose \( \I \) has reached such position: then \( \II \) plays \( 01\) and  from now on copies
\( \I \)'s moves, removing the sequences of the form \( \uu \conc 1 \).
This works as long as \( \I \) plays inside the tree generated by the nodes
in \( \mathcal{E}^{\NATURAL} ( A ) \).
Suppose at some stage \( \I \) goes astray and leaves this tree:
\begin{itemizenew}
\item
if \( \I \) enters an open set of the form \( O ( r ) \) then \( \II \) plays from now 
on \( \overline{a} \) with \( a \in A \), 
\item
if \( \I \) followed the relevant \( \uu\) node but after that played \( 0 \) 
instead of \( 1 \), then \( \II \) from now enumerates a sequence
\( \overline{a} \) with \( a \notin A \).
\end{itemizenew}
\end{quote}
This proves the first equivalence.
The second equivalence is similar and it is left to the reader.
\end{proof}

\begin{lemma}\label{lem:nullfrontierFLAT}
The set
\( \setofLR{ x \in \pre{\omega}{2} }{\exists^\infty e \in \mathcal{E}^{\NATURAL} ( A )\left ( e \subset x  \right )} \) is null.
Similarly for 
\( \setofLR{ x \in \pre{\omega}{2} }{\exists^\infty e \in \mathcal{E}^{\FLAT} ( A )\left ( e \subset x  \right )}\).
\end{lemma}

\begin{proof}
We shall prove only the first statement, leaving the second to the reader.
Let \( U_n = \bigcup_{e \in \mathcal{E}_n^{\NATURAL} ( A ) } \Nbhd_{e} \) and
\( U_0 = \pre{\omega}{2} \).
Then  \( U_{n + 1} \subseteq U_n \) and
\[ 
\bigcap_{n} U_n =
\setofLR{ x \in \pre{\omega}{2} }{\exists^\infty e 
\in \mathcal{E}^{\NATURAL} ( A ) \left ( e \subset x  \right )} .
\]
The result  will be proved by establishing that \( \mu  ( U_{n + 1} ) \leq  \mu ( U_n ) / 2 \).
As \( U_{n+1} \) is the disjoint union \( \bigcup_{e\in \mathcal{E}^{\NATURAL}_n ( A )} V_e\)
where \( V_e = \bigcup \setofLR{ \Nbhd_{e'}}{ e \subset e'
\in \mathcal{E}^{\NATURAL}_{n + 1} ( A )}\), it is enough to show that
\( \mu ( V_ e ) \leq  \mu ( \Nbhd_e ) / 2 \) for all \( e \in \mathcal{E}^{\NATURAL}_n ( A ) \).
Fix \( e \in \mathcal{E}^{\NATURAL}_n ( A ) \) and let 
\( E = \setofLR{ e' \in \mathcal{E}_{n + 1}^{\NATURAL} ( A ) }{ e \subset e'} \).
If \( e_0 , e_1 \in E \) are distinct,
then by Definition~\ref{def:O(r)} of the nodes \( \uu \), there are
 \( s_0 , s_1 \in  \pre{ < \omega }{ 2}\),
\( \eta_0 , \eta_1\in \setLR{01, 10}\) and \( k_i \in \omega \) such that  
\( e_i = e \conc 1 \conc \overline{s_i} \conc \eta_i \conc 0^{( k _ i )} \conc 1 \),
hence \( \Nbhd_{e \conc 1 \conc \overline{s_0} \conc \eta_0 \conc 0^{( k_0 )}} \cap
\Nbhd_{e \conc 1 \conc \overline{s_1} \conc \eta_1 \conc 0^{( k_1 )}} = \emptyset \).
Therefore 
\begin{align*}
\mu ( \Nbhd_e ) & \geq \sum_{ e' \in E} 
\mu ( \Nbhd_{ e' \restriction \lh ( e' ) - 1})
\\
& =  \sum_{ e' \in E  } 
2 \mu ( \Nbhd_{ e' })
\\
& = 2 \mu ( V_e ) 
\end{align*}
as required.
\end{proof}

\begin{proposition}\label{prop:NATisT-regular}
If \( A \) is \( \mathcal{T} \)-regular then \( \NATURAL  (A) \) and \( \FLAT ( A )\) are \( \mathcal{T} \)-regular.

Moreover, if \( A \) is in \( \ran( \Phi \restriction \bPi^{0}_{1} ) \cap  \ran( \Phi \restriction \bSigma^{0}_{1} ) \), 
then so are \( \NATURAL  (A) \) and \( \FLAT ( A )\).
\end{proposition}

\begin{proof}
First we deal with \( \NATURAL (A) \).
Suppose \( x \in \NATURAL (A) \): then there is \( n \in \omega \), \( e \in \mathcal{E}_n ^{\NATURAL} ( A ) \) 
and \( y \in \Plus ( A ) = \Phi  ( \Plus( A ) ) \) such that \( x = e \conc 1 \conc y \).
As 
\[
\lim_{m \to \infty} \mu \bigl ( \LOC{ \NATURAL (A)}{x \restriction m} \bigr )
\geq \lim_{m \to \infty} \mu \bigl ( \LOC{ \Plus (A )}{y \restriction m} \bigr ) = 1 
\]
then \( x \in \Phi  ( \NATURAL (A) ) \).

Suppose now \( x \notin \NATURAL (A) \) towards proving that \( x \notin \Phi  ( \NATURAL (A) ) \).
We distinguish four cases.
\begin{description}
\item[Case A]\label{caseA}
\( x \) extends no \( e \in \mathcal{E}^{\NATURAL} ( A ) \).
The either 
\begin{itemizenew}
\item
\( x \supset \overline{s}\conc \eta \conc v \) with \( \eta \in \setLR{01 , 10}\)
and \( v \perp \uu ( r_{\lh ( s ) } \cdot \mu ( \LOC{A}{s} ) ) \),
or else
\item
\( x = \overline{y} \) for some \( y \in \pre{\omega}{2} \).
\end{itemizenew}

In the first case \( x \in \Nbhd_{\overline{s}\conc \eta \conc v }\) and this basic open 
set is disjoint from \( \NATURAL ( A ) \), hence \( \mathcal{D}_{ \NATURAL ( A )} ( x  ) = 0 \).

In the second case: given \( m \), for any \( \eta \in \setLR{01,10}\) there is an \( i \in 2\) such that 
\( \LOC{\NATURAL ( A ) }{ x \restriction 2 m  \conc \eta \conc i} = \emptyset \),
hence \( \mu (  \LOC{\NATURAL ( A ) }{ x \restriction 2 m } ) \leq 3/4\).
In particular, \( x \notin \Phi ( \NATURAL ( A ) )\).

\item[Case B]\label{caseB}
\( x = \bigcup_{n} e_n \) with \( e_n \in \mathcal{E}_n^\mathrm{Nat} ( A ) \), 
hence \( \mu \bigl ( \LOC{\NATURAL (A)}{e_n } \bigr )  \leq \frac{1}{2}\)
by part~\ref{lem:opendisjointfromNAT-b} of Lemma~\ref{lem:opendisjointfromNAT}, 
and therefore \( x \notin \Phi  ( \NATURAL (A) ) \).
\item[Case C]
\( x \) extends \( e \conc 0 \) for some \( e \in \mathcal{E}^{\NATURAL} ( A ) \).
Then \( \mathcal{D}_{ \NATURAL ( A )} ( x ) = 0 \) 
by part~\ref{lem:opendisjointfromNAT-b} of Lemma~\ref{lem:opendisjointfromNAT}. 
\item[Case D]
 \( x = e \conc 1 \conc \overline{y} \) with \( y \notin A \), and 
\( e \) is the largest exit node contained in \( x \).
Since \( A = \Phi  ( A ) \), fix an increasing sequence \( ( m_k )_k \) and an \( \varepsilon > 0 \)
such that \( \mu ( \LOC{A}{ y \restriction m_k } ) < 1 - \varepsilon \), for all \( k \in \omega \).
Then there is an \( L \in \omega \) such that 
\[
\lh  \uu \bigl ( r_{m_k} \cdot \mu ( \LOC{A}{y \restriction m_k} ) \bigr ) \leq L 
\]
for all \( k \in \omega \).
By part~\ref{lem:opendisjointfromNAT-b} of Lemma~\ref{lem:opendisjointfromNAT}, 
\( \LOC{\NATURAL ( A )}{ e \conc 1 \conc \overline{y} \restriction 2 m_k }\) is disjoint 
from the two basic open neighborhoods given by 
\( i \conc (1 - i ) \conc \uu ( r_{m_k} \cdot \mu ( \LOC{A}{y \restriction m_k} ) ) \conc 0 \)
with \( i \in \setLR{0,1 }\), hence 
\[
\forall k \Bigl ( \mu \bigl (\LOC{\NATURAL ( A )}{ e \conc 1 \conc \overline{y} \restriction 2 m_k } \bigr ) 
\leq 1 - 2^{-L - 2} \Bigr ) ,
\]
proving that \( x \notin \Phi  ( \NATURAL (A) ) \).
\end{description}
Therefore \( \NATURAL (A) =  \Phi  ( \NATURAL (A) ) \).

Now we turn to \( \FLAT ( A ) \). 
If \( x \) extends infinitely many \( e \in \mathcal{E}^{ \FLAT} ( A ) \)
then \( \mu ( \LOC{\FLAT ( A )}{ x \restriction n} ) \to 1\) 
by~\eqref{eq:e_n}, hence \( x \in \Phi ( \FLAT ( A ) ) \).
Suppose now \( x \in \FLAT ( A ) \)  and \( x \supset e \in \mathcal{E}_n^{\FLAT} ( A ) \)
for some largest \( n \).
Then either
\begin{itemizenew}
\item
\( x = e \conc s \conc y \) with \( s \in  {}^{ \hh ( n ) }2  \) and \( s \neq 1^{ ( \hh ( n ) )} , 0^{ ( \hh ( n ) )} \).
Then \( x \) is in the interior of \( \FLAT ( A ) \) hence \( x \in \Phi ( \FLAT ( A ) ) \).
\item
\( x \supset e \conc 1^{( \hh( n ) )} \conc \overline{s} \conc \eta \conc v \) with 
\( v \perp \uu ( \max \setLR{ r_{n + 1} , r_{\lh ( s )} \cdot \mu ( \LOC{A}{s} )}) \)
and \( \eta \in \setLR{01,10}\).
Again \( x \) is in the interior of \( \FLAT ( A ) \).
\item
\( x = e \conc 1^{( \hh ( n ) ) } \conc \overline{y} \) with \( y\in A = \Phi ( A ) \).
By \( \mathcal{T} \)-regularity  \( \overline{y} \in \Plus ( A , r_n ) \), 
and since \( \LOC{ \FLAT ( A ) }{ e \conc 1^{( \hh ( n ) )}\conc \overline{y} \restriction m} 
\supseteq \LOC{\Plus ( A , r_n )}{\overline{y} \restriction m} \)
it follows that \( x \in \Phi  ( \FLAT ( A ) ) \).
\end{itemizenew}
Therefore 
\[
x \in \FLAT ( A ) \implies  x\in \Phi ( \FLAT ( A ) ) .
\]
Suppose now \( x \notin \FLAT ( A ) \) towards proving that \( x \notin \Phi ( \FLAT ( A ) ) \).
We distinguish three cases.

\begin{description}
\item[Case E]
\( x \) extends no \( e \in \mathcal{E}^{\FLAT} ( A ) \).
Then proceed as in Case A.
\item[Case F]
\( x \) extends \( e \conc 0^{ ( \hh ( n ) ) } \) for some \( e \in \mathcal{E}_n^{\FLAT} ( A ) \).
Then \( \mathcal{D}_{  \FLAT ( A )} ( x ) = 0 \)
by part~\ref{lem:opendisjointfromFLAT-b} of Lemma~\ref{lem:opendisjointfromFLAT}. 
\item[Case G]
 \( x = e \conc 1^{( \hh ( n ) )} \conc \overline{y} \) with \( y \notin A \), and 
\( e \in \mathcal{E}_n^{\FLAT} ( A ) \) is the largest exit node contained in \( x \).
As in Case D, fix an increasing sequence \( ( m_k )_k \) and an \( \varepsilon > 0 \)
such that \( \mu ( \LOC{A}{ y \restriction m_k } ) < 1 - \varepsilon \), for all \( k \in \omega \).
Then there is an \( L \in \omega \) such that \( \forall k \left ( \lh ( \uu_k ) \leq L \right ) \), where 
\[
\uu_k =  \uu \bigl ( \max \setLR{ r_{n+1} ,  r_{m_k} \cdot \mu ( \LOC{A}{y \restriction m_k} ) } \bigr ) .
\]
By part~\ref{lem:opendisjointfromFLAT-b} of Lemma~\ref{lem:opendisjointfromFLAT}, 
\( \LOC{\FLAT ( A )}{ e \conc 1^{( \hh ( n  ) )} \conc \overline{y} \restriction 2 m_k }\) is disjoint 
from the two basic open neighborhoods given by 
\( i \conc (1 - i ) \conc \uu_k \conc 0^{( \hh ( n + 1 ) )} \)
with \( i \in \setLR{0,1 }\), hence 
\[
\forall k \Bigl ( \mu \bigl (\LOC{ \FLAT ( A )}{ e \conc 1^{( \hh ( n ) )} \conc \overline{y} \restriction 2 m_k } \bigr ) 
\leq 1 - 2^{-1 -L - \hh ( n + 1 ) } \Bigr ) .
\]
\end{description}

Suppose now \( A = \Phi ( A ) \) towards proving that 
\( \mu ( \Fr  \NATURAL ( A ) ) = \mu ( \Fr \FLAT ( A ) ) = 0 \), and hence 
that \( \NATURAL ( A ) , \FLAT ( A )  \in \ran ( \Phi \restriction \bPi^{0}_{1} ) \cap 
\ran ( \Phi \restriction \bSigma^{0}_{1} ) \), 
by~\eqref{eq:D(closed)}.
Using~\eqref{eq:FrPlus} it is easy to check that
\begin{align*}
\Fr  \NATURAL ( A ) & \subseteq \setofLR{ x \in \pre{\omega}{2} }
{\exists^\infty e \in \mathcal{E}^{\NATURAL} ( A )\left ( e \subset x  \right )} 
\cup \bigcup_{ s \in  \pre{ < \omega }{2} } s \conc \overline{ \pre{\omega}{2} }
\\
\Fr  \FLAT ( A )  & \subseteq \setofLR{ x \in \pre{\omega}{2} }
{\exists^\infty e \in \mathcal{E}^{\FLAT} ( A )\left ( e \subset x  \right )} 
\cup \smash{\bigcup_{ s \in  \pre{ < \omega }{2} }}s \conc \overline{ \pre{\omega}{2} }
\end{align*}
and since \( \overline{ \pre{\omega}{2} } \) is null, we are done.

\end{proof}
\subsection{Proof of Theorem~\ref{th:belowDelta03}}

By Corollary~\ref{cor:lengthofWadge} it is enough to show by induction on 
\( \alpha < \omega _1^{ \omega _1} \)  that for each Borel set \( A \) of Wadge rank \( \alpha \),
there is an open set \( U \) and a closed set \( D \) such that 
\( \Phi ( U ) =  \Phi ( D ) \Wequiv A\).
Theorem~\ref{th:T-regularDelta02} takes care of the case when \( \alpha <  \omega _1\)
so we may assume that \(  \alpha \geq  \omega _1 \).
Let \( A \subseteq \pre{\omega}{2} \) be a set of Wadge rank \( \alpha \).
If \( \alpha \) is either a successor ordinal or a limit ordinal of countable cofinality
proceed as in the proof of Theorem~\ref{th:T-regularDelta02}, so we may assume 
that 
\begin{equation}\label{eq:cof(alpha)=omega1}
 \cof ( \alpha ) =  \omega _1 .
\end{equation}
Suppose that \( \alpha  =  \beta  + \gamma \) with \( \beta , \gamma <  \alpha \):
by replacing \( \beta \) with \( \beta + 1\) if needed, we may assume that any
\( B \) of Wadge rank \( \beta \) is self-dual.
Then \( A \Wequiv B \Wsum  C \) for some \( C \) of Wadge rank \( \gamma \).
By inductive assumption and Proposition~\ref{prop:Sumisregular} then 
\( A \Wequiv \Phi ( U ) = \Phi ( D ) \) for some \( U \in \bSigma^{0}_{1}\) and \( D \in \bPi^{0}_{1} \).
Therefore we may assume that 
\begin{equation*}\label{eq:alphaadditivelyindecomposable}
  \alpha \text{ is additively indecomposable.} 
\end{equation*}
Write \( \alpha =  \omega_1^\xi \cdot \delta + \mu \) with  \( \mu , \xi < \omega _1\):
by indecomposability \( \mu = 0 \) and 
therefore \( \delta \) is not a successor ordinal \( > 1 \), while
by~\eqref{eq:cof(alpha)=omega1} \( \delta \) cannot be limit.
Thus \(  \alpha = \omega_1^\xi \): by~\eqref{eq:cof(alpha)=omega1}
\( \xi \) cannot be be limit hence we may assume that 
\begin{equation*}
  \alpha =  \omega _1 ^{\nu  + 1}.
\end{equation*}
Let \( B \) be a set of Wadge rank \(  \omega _1^  \nu + 1 \), so that \( B \) is self-dual.
Then \( A \) is Wadge equivalent to either \( B^\natural \) or else to \( B^\flat \).
By inductive assumption \( B \Wequiv C \) for some 
\( C \in \ran ( \Phi \restriction \bSigma^{0}_{1}) \cap \ran ( \bPi^{0}_{1} ) \) hence 
\( \NATURAL ( C ) , \FLAT ( C ) \in \ran ( \Phi \restriction \bSigma^{0}_{1}) 
\cap \ran ( \Phi \restriction \bPi^{0}_{1} ) \)
by Proposition~\ref{prop:NATisT-regular}.
By Proposition~\ref{prop:NAT<=>nat} \( A \) is Wadge equivalent to either \( \NATURAL ( C ) \) 
or \( \FLAT ( C ) \), and this completes the proof of Theorem~\ref{th:belowDelta03}.

\section{Attaining the maximal complexity}\label{sec:ClosedsetswithcomplicatedD}
In this section we shall prove Theorems~\ref{th:Pi03complete},
\ref{th:emptyinteriorPi03}, \ref{th:Pi03comeagre} and~\ref{th:forcingdensePi03},
and a result on supports (see Section~\ref{sec:easyfacts}). 

\subsection{Proof of Theorem~\ref{th:emptyinteriorPi03}}
Let \( A\neq \emptyset \) be \( \mathcal{T} \)-regular, with empty interior.
We will show that \(  \mathbf{P}_3 \leqW A \), where 
\[
\mathbf{P}_3 = \setofLR{z \in  \pre{ \omega \times \omega }{2} }{ \FORALL{n} 
\FORALLS{\infty}{m} z ( n , m ) = 0 }.
\]
Since \( \mathbf{P}_3 \) is a complete \( \bPi^{^0}_{3} \) set~\cite[p.~179]{Kechris:1995kc}
the result follows.
Recall that \( \densitytree ( A ) \) is the tree of Definition~\ref{def:D(A)}.
Given a \( 0 \)-\( 1 \) matrix \( a = \seqofLR{a ( i , j ) }{ i , j < n} \) of order \( n \), a sequence 
\( \varphi ( a ) \in \densitytree ( A ) \) will be constructed so that 
\[
a \subset b \IMPLIES \varphi  ( a ) \subset \varphi  ( b )
\]
and therefore 
\[
f \colon  \pre{ \omega \times \omega }{2}  \to \body{ \densitytree ( A )}, \qquad 
f ( z ) = \bigcup_{n} \varphi ( z \restriction n \times n )
\]
is continuous.
The function \( f \) will witness that \( \mathbf{P}_3 \leqW \Phi  ( A ) \).

Let \( I_n = [ 1 - 2^{ - n } ; 1 - 2^{ - n - 1} ) \) and let 
\( \rho \colon \densitytree ( A )  \to \omega \) be 
\[
 \rho ( s ) = n \IFF \mu ( \LOC{A}{s} ) \in I_n .
\]
The map \( \rho \) is well defined since \( \mu (\LOC{A}{s} ) \neq 1 \) for all \( s \),
by the assumption on \( A \).
If \( s \in \densitytree ( A ) \) then 
\begin{equation*}
 \mu ( \LOC{A}{s} ) \geq 1 - 2^{- n} \IMPLIES 
 \mu ( \LOC{A}{s \conc i } ) = 2 \mu ( \LOC{A}{s} ) -  \mu ( \LOC{A}{s \conc ( 1 - i ) } ) 
 \geq 1 - 2^{- n + 1} 
\end{equation*}
hence 
\begin{equation}\label{eq:descending}
 \FORALL{s \in \densitytree ( A ) } \FORALL{n > 0} \bigl ( \rho ( s) \geq n \implies  
\rho ( s \conc 0 ) , \rho ( s \conc 1 ) \geq n - 1 \bigr )  .
\end{equation}
It follows at once that 
\[
x \in \Phi ( A ) \IFF x \in \body{ \densitytree ( A )} \wedge  \lim_{n \to \infty} \rho ( x \restriction n ) = \infty .
\]

\begin{claim}\label{cl:descending}
Suppose \( s \in \densitytree ( A ) \).
For any \( j < \rho ( s ) \) there is a \(s \subset t \in \densitytree ( A ) \)
such that \( \rho ( t ) = j \) and 
\( \forall u \left ( s \subseteq u \subseteq t \implies \rho ( u ) \geq j \right ) \).
\end{claim}

\begin{proof}
Let \( x \in \Nbhd_s \) have density \( 0 \) in \( A \).
By~\eqref{eq:descending} let \( t \subset x \) be the shortest node extending \( s \) 
such that \( \mu ( \LOC{A}{t} ) < 1 - 2^{- j - 1} \).
\end{proof}

If \( a \) is the empty matrix, then \( \varphi ( a ) = \emptyset \), and if 
\( a = \seqofLR{a( i , j ) }{ i , j \leq n} \) is a matrix of order \( n + 1 \), 
we set 
\[ 
 \varphi ( a ) = t
\] 
where \( t \) is defined as follows:
\begin{description}
\item[Case 1] 
\( \FORALL{j \leq n} a ( j , n ) = 0 \).
By Proposition~\ref{prop:climb} let \( t \in \densitytree ( A ) \) be an extension of  
\( \varphi ( a \restriction n \times n ) \) such that 
\( \rho ( t ) = n + 1 \) and 
\[ 
\forall u \left ( \varphi ( a \restriction n \times n ) \subseteq u \subseteq t
\implies \rho ( u ) \geq \rho \circ \varphi ( a \restriction n \times n ) \right ) .
\]
\item[Case 2] 
\( \EXISTS{j \leq n} a ( j , n ) = 1 \).
Let \( j_0 \) be the least such \( j \) and by Proposition~\ref{prop:climb} 
and Claim~\ref{cl:descending} let \( t \in \densitytree ( A ) \) be such that 
\( t \supset \varphi ( a \restriction n \times n ) \), \( \rho ( t ) = j _0  \), and
\[
 \forall v \left ( \varphi ( a \restriction n \times n ) \subseteq v \subseteq t
\implies \rho ( v ) \geq \min \setLR{ \rho ( \varphi ( a \restriction n \times n ) ) , \rho  ( t ) } \right ).
\]
\end{description}

Suppose \( z \in \mathbf{P}_3 \).
For every \( k \in \omega \) choose \(m_k \) such that \( \FORALL{m \geq m_k } a( k , m ) = 0 \)
and let 
\[
M_k = \max \setLR{m_0 , \dots , m_k} .
\] 
Therefore for every \( n \geq \max \setLR{ k , M_k }\) the least \( j \leq n \) such that \( z ( j , n ) = 1 \) 
--- if such \( j \) exists --- is larger than \( k \) and therefore 
\( \rho ( \varphi ( z \restriction n \times n ) ) \geq k \).
This shows that \( \lim_{i \to \infty } \rho ( f ( z ) \restriction i ) = \infty \)
hence \( f ( z ) \in \Phi  ( A ) \).

Conversely suppose \( z \notin \mathbf{P}_3 \).
Let \( n_0 \) be the least \( n \) such that the \( n \)th row contains infinitely many \( 1 \)s,
i.e. \( \EXISTSS{\infty}{m} z ( n_0 , m ) = 1 \) and 
\( \FORALL{i < n_0} \FORALLS{\infty}{m} z ( i , m ) = 0 \).
Then for arbitrarily large \( n \), \( \varphi ( z \restriction n \times n ) \) is computed 
as in Case 2, hence \( \rho ( f ( z) \restriction i ) = n_0 \)
for infinitely many \( i \).
In particular \( \lim_{i \to \infty} \rho \left (f ( z ) \restriction i \right ) \neq \infty \), 
hence \( f ( z ) \in \body{ \densitytree ( A ) } \setminus \Phi  ( A ) \).

This finishes the proof of Theorem~\ref{th:emptyinteriorPi03}.

\subsection{Closed sets with empty interior and the proof of Theorem~\ref{th:Pi03complete}}
\begin{theorem}\label{th:densePi03complete}
Let \( S \) be a perfect pruned tree such that \( \mu \body{S} > 0 \) and let \( \varepsilon > 0 \) be given.
Then there is a pruned tree \( T \subseteq S \) such that \( \body{T} \) 
has empty interior in \( \body{S} \), \( \mu ( \body{T} ) + \varepsilon  > \mu ( \body{S} )\).
\end{theorem}

\begin{proof}
The tree \( T \) will be defined as
\[
T = \setofLR{u \in  S }{ \FORALL{n \left (  t_n \nsubseteq u \right )} }
\]
for an appropriate sequence \(( t_n )_n \subseteq S\).
Density amounts to say that
\begin{equation} \label{eq:conditionsont1}
\FORALL{s \in S } \exists {n} 
\left ( s \subseteq t_n \OR t_n \subseteq s \right ) 
\end{equation}
and since the sets \( \Nbhd_{t_n } \cap \body{S} \) are disjoint,
\begin{equation}\label{eq:conditionsont2}
 n \neq m \IMPLIES t_n \perp t_m .
\end{equation}
Let \( ( \ell_n )_n \) be a strictly increasing sequence of natural numbers 
such that \( \ell_0 > 0 \).
A sequence \( (t_n )_ n \)  that satisfies~\eqref{eq:conditionsont1},~\eqref{eq:conditionsont2} and
\begin{gather}
  \forall n , m < \omega  \left (n < m \IMPLIES \lh ( t_n ) < \lh ( t_m )  \right )  \label{eq:conditionsont3}
 \\
 \exists ^{\infty} n \left (\lh ( t_{n + 1} ) > \lh ( t_n ) + 1 \right )  \label{eq:conditionsont4}
 \\
\forall n \left ( \ell_n \leq \lh ( t_n )  \right )  \label{eq:conditionsont5}
\end{gather}
is called a  \markdef{sparse sequence of order \( ( \ell_n )_n \) for \( S \)}. 
By~\eqref{eq:conditionsont5}  
\[ 
\sum_{n=0}^\infty 2^{-\lh ( t_n )} \leq \sum_{n=0}^\infty 2^{- \ell_n} 
\]
so if \( ( \ell_n )_n \) grows fast enough, then \( \mu ( \body{S} \setminus \body{T} ) < \varepsilon \)
as required. 

To show the existence of such sequence, let \( \lhd \) be the well-order of \(  \pre{< \omega }{2}\) obtained by ordering the 
nodes according to their length and comparing nodes of equal length given by the lexicographic order:
\begin{equation}\label{eq:triangleorder}
s \mathrel{\lhd} t \IFF \lh ( s ) < \lh ( t ) \OR \left ( \lh ( s ) = \lh ( t ) \AND s \lelex t \right ) .
\end{equation}
We shall define inductively \( t_n , u_n \in S \) such that 
\begin{subequations}
\begin{align}
& u_n \text{ is the \( \lhd \)-least \( u \in S \) such that } \FORALL{i < n} ( u \perp t_i ) , \label{eq:conditionsonu_n1}
\\
&   t_n \supset u_n \AND \lh ( t_n ) \geq \max \setLR{\ell _n , t_{n-1} } \AND \exists u \in S
\left ( u \perp t_n \AND \FORALL{i < n} ( u \perp t_i )  \right ) .  \label{eq:conditionsonu_n2}
\end{align}
\end{subequations}
Suppose \( u_i , t_i \) have been defined for all \( i < n \) and satisfy~\eqref{eq:conditionsonu_n1}
and~\eqref{eq:conditionsonu_n2}.
By~\eqref{eq:conditionsonu_n2} there is a  \( \lhd \)-least \( u_n \in S \) which is incompatible with 
\( t_0 , \dots , t_{n-1} \).
As \( S \) is perfect, there exist \( t_n  , u \in S \) incompatible extensions of \( u_n \),
such that \( \lh ( t_n ) \geq \ell_n , \lh ( t_{n - 1} ) \). 
Since \( t_n \supseteq u_n \) and  \( u_n \perp t_i \) for \( i < n \),
it follows that the \( t_n \)'s are pairwise incompatible, i.e.~\eqref{eq:conditionsont2} holds.
Given \( s \in S \) such that \( t_k \nsubseteq s \) for all \( k \), pick \( n \) least such that \( s \lhd u_{n + 1} \):
since \( s = u_n \subseteq t_n \) is impossible, then  \( s \) must be compatible 
with some \( t_i \) with \( i \leq n \), hence \( s \subset t_i \).
Therefore~\eqref{eq:conditionsont1} holds.
Moreover it is trivial to arrange the construction so that \( \lh ( t_n ) + 1 < \lh ( t_{n+1} ) \) 
for infinitely many (or even for every) \( n \), hence~\eqref{eq:conditionsont4} holds as well.
\end{proof}

In particular, taking \( S =  \pre{ < \omega }{2}\), a closed set of positive measure
and empty interior \( C = \body{T} \) is obtained.
By~\eqref{eq:densityclosed} \( \Phi ( C ) \subseteq C \) hence also \( \Phi ( C ) \) has empty 
interior and therefore \( \Phi ( C ) \) is complete \( \bPi^{0}_{3}\) by Theorem~\ref{th:emptyinteriorPi03}.
In other words we have shown that 
\begin{equation*}
\exists T \text{ perfect pruned tree such that \(  \Phi ( \body{T} ) \)  is complete } \bPi^{0}_{3}   
\end{equation*}
which is half of Theorem~\ref{th:Pi03complete}. 
To prove the other half, for any \( r \in ( 0 ; 1 ] \) 
pick an increasing sequence \( ( n_k )_k \) such that 
\( r = \sum_{k = 0}^\infty 2^{-n_k - 1}  \) and let 
\[
 O^* ( r ) = \bigcup_{k \in \omega } \Nbhd_{ 0^{ ( n_k - 1 )} 1 } .
\]
The set \( O^* ( r ) \) is open, \( \Fr ( O^* ( r ) ) = \setLR{ 0^{ ( \infty ) }}\),
and \( \mu ( O^* ( r )  ) = r \).
For \( T \) as above, consider the open set
\[
W = \bigcup_{s \in T} \bigcup_{i \in 2} \overline{s} \conc i \conc ( 1 - i ) \conc O^* ( \mu \body{\LOC{T}{s} } ) .
\]
Then \( \Fr W \subseteq \overline{\body{T}} \cup \setofLR{x \in \pre{\omega}{2} }{\FORALLS{\infty}{ n} x ( n ) = 0 }\) 
is null hence \( \Phi ( W ) = \Phi ( \Cl ( W ) ) \).
It is enough to prove

\begin{proposition}
The map \( x \mapsto \overline{x} \) witnesses \( \Phi ( \body{T} ) \leqW \Phi ( W )\).
\end{proposition}

\begin{proof}
By~\eqref{eq:series2}
\begin{equation*}
\begin{split}
\mu ( \LOC{W}{\overline{t}}) & = \sum_{ s \in \LOC{T}{t}} 2^{- 2 \lh ( s ) - 1} \mu \body{ \LOC{T}{t \conc s} } 
\\
 &  = \sum_{ s \in  \pre{ < \omega }{2} } 2^{- 2 \lh ( s ) - 1} \mu \body{ \LOC{T}{t \conc s} }
 \\
 & = \mu \body{ \LOC{T}{t} } .
\end{split} 
\end{equation*}
 Suppose \( x \in \Phi ( \body{T} ) \).
 Then \( \mu ( \LOC{W}{ \overline{x} \restriction 2 n } ) = \mu ( \body{ \LOC{T}{ x \restriction n} } ) \to 1\).
 Since \( \mu ( \LOC{W}{ \overline{x} \restriction 2 n + 1 } ) = \frac{1}{2}  \mu ( \body{ \LOC{T}{ x \restriction n} } ) 
 + \frac{1}{2}  \mu ( \body{ \LOC{T}{ x \restriction n + 1 } } ) \to 1 \), then \( \overline{x} \in \Phi ( W )\).
Conversely, if \( x\notin \Phi ( \body{T} ) \), take \( ( n_k )_k \) such that 
\( \mu (\body{\LOC{T}{x \restriction n_k} } ) < 1 - \varepsilon \)
for some \( \varepsilon \), hence 
\( \mu ( \LOC{W}{ \overline{x} \restriction 2 n_k } ) = \mu ( \body{\LOC{T}{ x \restriction n_k } } ) < 1 - \varepsilon \),
 hence \( \overline{x} \notin \Phi ( W ) \).
\end{proof}

Using sparse sequences it is possible to show that the assumption in Theorem~\ref{th:emptyinteriorPi03}
cannot be weakened by requiring that \( A \) be \( \mathcal{T} \)-regular and with a frontier of positive measure.

\begin{corollary}\label{cor:regularwithlargefrontier}
There is a \( \mathcal{T} \)-regular open set \( U \) such that \( \mu ( \Fr U ) > 0 \).
\end{corollary}

\begin{proof}
Let \( \ell_n = 2n \) and let \( T \) be the closed set with empty interior 
constructed from a sparse sequence of order \( ( \ell_n )_n \).
Let \( U = \neg \body{T} \).
Then \( \mu ( U ) = \sum_{n = 0}^\infty  2^{-2n - 2} = 2 / 3\) and for \( t \in T \),
\begin{align*}
\mu ( \LOC{U}{t} ) & = 2^{\lh ( t ) } \sum_{ t_n \supset t} 2^{- \lh ( t_n )} 
\\
 & \leq 2^{\ell _k - 1 } \sum_{n = k}^\infty 2^{- \ell_n} &&(\text{for some } k = k ( t ))
\\
& \leq 2^{2 k + 1} \sum_{n = k}^\infty 2^{- 2 n - 2 } 
\\
& = \textstyle\frac{2}{3} .
\end{align*}
Therefore \( \Phi ( U ) = U \) but \( \Fr ( U ) = \body{T} \) has positive measure.
\end{proof}

\subsection{Supports are not complete invariants}
Using a sparse sequence it is possible to show that the inner and outer supports are not 
complete invariants for measure equivalence.

\begin{proposition}\label{prop:supports}
There are measurable sets \( A \not \equiv B \) such that
\[
 \supt^{- } ( A ) = \supt^{- } ( B ) \quad\text{and}\quad \supt^{+ } ( A ) = \supt^{+ } ( B ) .
\]
\end{proposition}

\begin{proof}
Let \( U \) and \( T\) be as in Corollary~\ref{cor:regularwithlargefrontier}.
For \( \ell_n ' = 3n+2 \) let \( ( t'_n )_n \) be a sparse sequence of order 
\( ( \ell_n' )_n \) in \( T \), and let \( T' = \setofLR{u \in T}{ \forall n \left (t'_n \nsubseteq u \right )}\).
Finally let
\[
A = U \qquad\text{and}\qquad B = U \cup \body{T'} .
\] 
As \( \mu \body{T} = \frac{1}{3} > \sum_{n=0}^{\infty} 2^{ - 3 n - 2} \) it follows 
that \( \mu \body{T'} > 0 \), hence \( \mu ( A ) < \mu ( B ) < 1 \).
As \( U \) is open and dense, then \( \supt^+ ( A ) = \supt^+ ( B ) = \pre{\omega}{2} \),

By Corollary~\ref{cor:regularwithlargefrontier} and~\eqref{eq:densityinteriorclosure},
\( U = \Phi ( A ) = \supt^- ( A ) \) hence it is enough to show that \( U = \supt^- (  B )\).
Again by~\eqref{eq:densityinteriorclosure} it is enough to show that \( U = \Int \Phi ( B ) \).
By monotonicity \( U = \Phi ( A ) \subseteq \Phi ( B ) \), so it is 
enough to check that \( \Int \Phi ( B )  \subseteq U \).
Given \( x \in \body{T} \) and \( n \in \omega \)  it is enough to show that there are elements 
in \( \Nbhd_{x \restriction n}\)  whose density in \( B \) is not \( 1 \).
But this is immediate since \( \body{T} \setminus \body{T'} \) is open and dense in \( \body{T} \).
\end{proof}

By Lemma~\ref{lem:closureofDissupt} we obtain

\begin{corollary}
There are measurable sets \( A \not \equiv B \) such that 
\[
\Cl \Phi ( A ) = \Cl \Phi ( B ) \quad\text{and}\quad \Int \Phi ( A ) = \Int \Phi ( B ) .
\]
\end{corollary}

%
%

\subsection{Density in the sense of forcing and the proof of Theorem~\ref{th:forcingdensePi03}}
The Boolean algebra \( \MALG \) is endowed with a partial order 
\[
\begin{split}
 \eq{A} \leq \eq{B} & \iff \mu ( A \setminus B ) =  0
\\
 & \iff A \cap B \equiv A .
\end{split}
\]
The minimum of \( \MALG \) is \( \eq{ \emptyset}\) the collections of null sets,
and is denoted by \( 0 \).
If \( A , B \) are \( \mathcal{T} \)-regular, then~\eqref{eq:densityintersection} implies that
\[
\eq{A} \leq \eq{B} \iff A \subseteq B .
\]
We will say that \( \eq{B} \in \MALG \) has empty interior just in case \( \Int ( \Phi ( B ) ) = \emptyset \),
hence if \( \eq{B } \) has empty interior, then every \( \eq{A} \leq \eq{B} \) has also empty interior.
From Theorem~\ref{th:emptyinteriorPi03} we obtain

\begin{corollary}
If \( 0 <  \eq{B} \in \MALG \) has empty interior then \( \Phi ( A ) \)
is \( \bPi^{0}_{3}\) complete, for every \( 0 < \eq{A} \leq \eq{B} \).
In particular, \( \mathscr{W}_{ \boldsymbol{d}}\) is not dense in the sense of forcing
in \( \MALG \) for any Wadge degree \( \boldsymbol{d} \subseteq \bDelta^{0}_{3}\).
\end{corollary}

On the other hand, 

\begin{proposition}\label{prop:densebelownonemptyinterior}
For every Wadge degree \( \boldsymbol{d} \subseteq \bDelta^{0}_{3}\) and every 
\( \eq{A} \in \MALG \) with nonempty interior, there is a 
\( \eq{B} \in  \mathscr{W}_{ \boldsymbol{d}}\) with \( \eq{B} \leq \eq{A} \).
\end{proposition}

\begin{proof}
Suppose \( A \) is \( \mathcal{T} \)-regular and suppose \( \Nbhd_s \subseteq A \).
By Theorem~\ref{th:belowDelta03}
let \( D = \Phi ( D ) \in \boldsymbol{d}\): since 
\( s \conc D \in \boldsymbol{d}\) is also \( \mathcal{T} \)-regular 
we are done.
\end{proof}

Given \( A \) of positive measure and an \( \varepsilon > 0 \), choose 
a perfect pruned tree \( S \) such that \( \body{S} \subseteq A\) and \( \mu ( A \setminus \body{S} ) < \varepsilon / 2\).
Let \( T \subseteq S \) be a perfect pruned tree such that \( \mu ( \body{S} \setminus \body{T} ) < \varepsilon / 2 \) 
and such that \( C = \body{T}\) has empty interior in \( \body{S} \).
Then \( \eq{C} \leq \eq{A} \) and since \( \Phi ( C ) \subseteq C \) then \( \eq{C} \) has empty interior.
Therefore we have shown that

\begin{proposition}
Let \( \mathscr{W} \) be the collection of all \( \eq{A} \in \MALG \) with empty interior.
Then 
\[
 \FORALL{ \varepsilon > 0} \FORALL{ \eq{A} \in \MALG } \EXISTS{ \eq{B} \in \mathscr{W} }
 \bigl ( \eq{B} \leq \eq{A} \AND \mu ( A \symdif B ) < \varepsilon \bigr ).   
\]
\end{proposition}
This and Theorem~\ref{th:emptyinteriorPi03} yield 
(a slight strengthening of) Theorem~\ref{th:forcingdensePi03}.

\subsection{Proofs of Theorems~\ref{th:new} and~\ref{th:Pi03comeagre}}
Fix \( ( V_n )_n \)  an enumeration without repetitions of  
\( \setofLR{ \Nbhd_s}{ s \in  \pre{ < \omega }{2} } \).

\begin{proposition}
The set
\[
\mathscr{A} = \setofLR{ \eq{A} \in \MALG }
	{ \forall n \left ( \mu ( A \cap V_n ) > 0 \AND \mu ( V_n \setminus A ) > 0  \right )}
\]
is comeager in \( \MALG \).
\end{proposition}

\begin{proof}
We will show that the sets
\begin{align*}
 \mathscr{A} ^- &= \setofLR{ \eq{A}\in \MALG }{ \forall n \left ( \mu ( V_n \setminus A ) > 0 \right )}
 \\
 \mathscr{A} ^+ &= \setofLR{ \eq{A}\in \MALG }{ \forall n \left ( \mu ( A \cap V_n ) > 0 \right )}
\end{align*}
are comeager, and this will suffice since \( \mathscr{A} = \mathscr{A}^- \cap \mathscr{A} ^+ \).
Let us start with \( \mathscr{A}^- \).

By a result of Banach and Mazur (see~\cite[Theorem 8.33]{Kechris:1995kc}),
it is enough to show that Player \( \II \) has a winning strategy in the game 
\( G^{**} (  \mathscr{A}^-) \) in which the two
players choose alternatively nonempty open subsets of \( \MALG \)
\[
\begin{tikzpicture}
\node (II) at (0,0) [anchor=base] { \( \II \)};
\node (I) at (0,1) [anchor=base] {\( \I \)};
\node (x_0) at (1,1) [anchor=base] {\( U_0 \)};
\node (x_1) at (2,0) [anchor=base] {\(U_1 \)};
\node (x_2) at (3,1) [anchor=base] {\( U_2 \)} ;
\node (x_3) at (4,0) [anchor=base] {\( U_3 \)};
\node at (5.5,1) [anchor=base] {\( \cdots \)};
\node at (5.5,0) [anchor=base] {\( \cdots \)};
\node (dots) at (5.5,0) [anchor=base] {\( \cdots \)};
\node (A) at (x_3.south -| II.east) {}; 
\node (B) at (I.north -| II.south east) {}; 
\node (C) at ($0.5*(A) +0.5 *(B)$){}; 
\draw (II.south east) -- (I.north -| II.south east) 
(II.north west |- C.center) -- (dots.east |- C.center); 
\node at (C-|II.south west) [anchor=east]{\( G^{**} (  \mathscr{A}^- ) \)};
\begin{pgfonlayer}{background} 
\fill[gray!30,rounded corners] 
(II.south west) rectangle  (dots.east |-  I.north); 
\end{pgfonlayer} 
\end{tikzpicture} 
\]
such that
\( U_0 \supseteq U_1 \supseteq U_2 \supseteq U_3 \supseteq \dots \),
and \( \II \) wins iff \( \bigcap_{n} U_n \subseteq  \mathscr{A}^- \).
Since \( \MALG \) is a metric space with distance \( \delta ( \eq{A} , \eq{B} ) =  \mu ( A \symdif B ) \),
without loss of generality we may assume that each \( U_n \) is an open ball
\[ 
U_n = \Ball_ \delta ( \eq{A_n} ; \varepsilon _ n ) 
	= \setofLR{\eq{B} \in \MALG }{ \delta ( \eq{B} , \eq{A_n} ) < \varepsilon _n } .
\] 
The strategy for \( \II \) requires that 
\begin{enumerate}
\item\label{th:Pi03comeagre-ii}
\( \Cl U_{ 2 n + 1 } \subseteq U_{2n}\),
\item\label{th:Pi03comeagre-iii}
\( \varepsilon _{ 2 n + 1 } < 2^{-n} \),
\item\label{th:Pi03comeagre-iv}
\( \varepsilon _{ 2 n + 1 } \leq \mu ( V_n \setminus A_{2 n + 1 } ) \).
\end{enumerate}
Conditions~\eqref{th:Pi03comeagre-ii} and~\eqref{th:Pi03comeagre-iii} are easily satisfied.
For~\eqref{th:Pi03comeagre-iv} pick \( A_{ 2 n + 1 } '  \)  such that 
\[ 
r \equalsdef  \mu ( A_{ 2 n } \symdif A_{ 2 n + 1 }' ) = \delta  ( \eq{ A_{ 2 n } } , \eq{ A_{ 2 n + 1 } ' } ) < \varepsilon _{ 2 n } ,
\]
and let \( \varepsilon _{ 2 n + 1 } ' <  \min \setLR{ \varepsilon _{ 2 n } - r, 2^{-n} }\).
We have two cases.
 \begin{description}
\item[Case 1] 
\( \mu ( V_n \cap A'_{ 2 n + 1 } ) = 0 \).
Then let \( A_{2 n + 1} = A_{2 n + 1} ' \) and \( \varepsilon _{ 2 n + 1 } \leq 
 \varepsilon _{ 2 n + 1 } ' , \mu ( V_n )  \).
\item[Case 2] 
\( \mu ( V_n \cap A_{ 2 n + 1 }' ) > 0 \).
Then take \( V'_n \subseteq A_{2 n + 1}'  \cap V_n \) such that 
\[
 0 < \mu ( V_n' ) < \varepsilon _{ 2 n } - r .
\]
Let \( A_{2 n + 1 } = A_{ 2 n + 1 }' \setminus V_n' \).
Then
\[
\begin{split}
\delta ( \eq{A_{ 2 n}} , \eq{ A_{2 n + 1} } ) & \leq
\delta ( \eq{A_{ 2 n}} , \eq{ A_{2 n + 1} '}  ) + \delta ( \eq{A_{ 2 n + 1 }'} , \eq{ A_{2 n + 1} } )
\\
& < r +  \varepsilon _{ 2 n} - r  
\\
&= \varepsilon _{ 2 n} ,
\end{split}
\]
hence \( \eq{A_{ 2 n + 1 }} \in \Ball_ \delta  ( \eq{A_{2 n}} ; \varepsilon _{ 2 n } ) \).
Choose \( \varepsilon _{ 2 n + 1} \leq \varepsilon _{ 2 n + 1} ' , \mu ( V_n ' ) / 2 \)
such that \( \Cl \Ball_ \delta  ( \eq{A_{2 n + 1}} ; \varepsilon _{ 2 n + 1 } ) 
\subseteq \Ball_ \delta ( \eq{A_{2 n}} ; \varepsilon _{ 2 n } ) \). 
\end{description}
We leave it to the reader to verify that conditions~\eqref{th:Pi03comeagre-ii}--\eqref{th:Pi03comeagre-iv} are verified.

Let us check that this is a winning strategy for \( \II \).
Conditions~\eqref{th:Pi03comeagre-ii} and~\eqref{th:Pi03comeagre-iii} imply that \( ( \eq{A_n} )_n \) 
is a Cauchy sequence converging to some \( \eq{A_ \infty} \), 
and that \( \bigcap_{n}  U_n = \setLR{ \eq{ A_\infty }}\).
Since \( \delta ( \eq{A_\infty} , \eq{ A_{ 2 n + 1} } ) = \mu ( A_\infty \symdif A_{ 2 n + 1 } ) < \varepsilon _{ 2 n + 1 }\),
condition~\eqref{th:Pi03comeagre-iv} implies that 
\( \mu ( V_n \setminus A_\infty ) > 0 \).

We are now going to show that Player \( \II \) has a winning strategy in 
\( G^{**} ( \mathscr{A}^+ ) \), proving thus that \( \mathscr{A}^+ \) is comeager.
Conditions~\eqref{th:Pi03comeagre-ii} and \eqref{th:Pi03comeagre-iii}
are as before, while~\eqref{th:Pi03comeagre-iv} is replaced by
\begin{enumerate}[label={(\arabic*${}'$)}]\addtocounter{enumi}{2}
\item \label{th:Pi03comeagre-v}
\( \varepsilon  _{n + 1} \leq \mu ( V_n \cap A_{ 2 n + 1} ) \).
\end{enumerate} 
To satisfy~\ref{th:Pi03comeagre-v} pick \( A'_{2 n + 1}\) such that 
\( r = \mu ( A_{ 2 n} \symdif A_{ 2 n + 1} ' ) < \varepsilon _{ 2 n} \)
and let \( \varepsilon '_{ 2 n + 1} <  \min \setLR{ \varepsilon _{ 2 n } - r, 2^{-n} }\) 
as before.
We have two cases.
 \begin{description}
\item[Case 1\({}'\)] 
\( \mu ( V_n \cap A_{ 2 n + 1} ' ) > 0 \).
Then take \( A _{ 2 n + 1} = A _{ 2 n + 1}' \) and
\( \varepsilon _{ 2 n + 1}\leq  \varepsilon _{ 2 n + 1 } ' , \mu ( V_n \cap A _{ 2 n + 1}' )  \).
\item[Case 2\({}'\)] 
\( \mu ( V_n \cap A_{ 2 n + 1} ' ) =  0 \).
Then take \( V'_n \subseteq V_n \) such that \( \mu ( V_n' ) < \varepsilon _{ 2 n } - r \)
and let \( A _{ 2 n + 1} = A _{ 2 n + 1}' \cup V'_n \) so that
\( \eq{A_{ 2 n + 1}} \in \Ball_ \delta ( \eq{ A_{ 2 n}} ; \varepsilon _{ 2 n} ) \).
Choose \( \varepsilon _{ 2 n + 1} \leq \varepsilon _{ 2 n + 1}' , \varepsilon _{ 2 n} - r \).
\end{description}
As before, playing according to this strategy guarantees that 
\( \bigcap_{n} U_n = \setLR{ \eq{ A_\infty } } \) with \( \mu ( A_\infty ) > 0 \)
and since \( \mu ( A_\infty \symdif A_{ 2 n + 1} ) < \varepsilon _{ 2 n + 1} \), condition~\ref{th:Pi03comeagre-v} 
implies that \( \mu ( A_\infty \cap V_n ) > 0 \).
\end{proof}

Theorems~\ref{th:Pi03comeagre} and~\ref{th:new} now follow easily.

\begin{proof}[Proof of Theorem~\ref{th:Pi03comeagre}]
Let \( A = \Phi ( A ) \) and \( \eq{A} \in \mathscr{A} \).
Then \( V_n \nsubseteq A \) for all \( n \), hence \( \Int ( A ) = \emptyset \).
Therefore by Theorem~\ref{th:emptyinteriorPi03}
\[
\mathscr{A} \setminus\setLR{ \eq{ \emptyset}} \subseteq 
\setofLR{\eq{A}\in \MALG}{ \Phi ( A ) \text{ is complete } \bPi^{0}_{3} } \qedhere
\]
\end{proof}

\begin{proof}[Proof of Theorem~\ref{th:new}]
Let  \( \eq{A} \in \mathscr{A} \) and, towards a contradiction, suppose 
\( \mu ( A \symdif D ) = 0 \) with \( D \in \bDelta^{0}_{2}\).
By construction \( \mu ( D \cap V_n ) , \mu ( V_n \setminus D ) > 0 \) for all \( n \),
hence \( D \) would be dense and co-dense, contradicting Baire's category theorem.
Therefore
\[
\mathscr{A}  \setminus\setLR{ \eq{ \emptyset}}  \subseteq 
\setofLR{\eq{A}\in \MALG}{ \eq{A} \cap \bDelta^{0}_{2} = \emptyset } \qedhere
\]
\end{proof}

\section{Proof of the Lebesgue density theorem in the Cantor space}\label{sec:proofofdensitythm}
It is enough to show that \( A \setminus \Phi  ( A ) \) is null for every measurable set \( A \).
As \( A \setminus \Phi  ( A ) \subseteq \bigcup_{ \varepsilon \in \Q^+ } B_ \varepsilon  \) where 
\[
B_ \varepsilon  = \setof{x \in A}{\liminf_{n \to \infty} \mu ( \LOC{A}{x \restriction n} )  <1  - \varepsilon  },
\]
it is enough to show that each \( B_ \varepsilon \) is null.
Arguing as on page~\pageref{eq:folklore}, it is easy to check that 
each \( B_ \varepsilon \) is measurable.
Towards a contradiction, suppose that \( B = B_ \varepsilon \) is not null for some fixed \( \varepsilon < 1\).
Choose \( U \supseteq B \) open and such that \( \mu ( U ) < \mu ( B ) / (1 - \varepsilon ) \).
Let 
\[
\mathcal{B} = \setofLR{s \in  \pre{< \omega }{2} }{ \Nbhd_s \subseteq U \AND \mu ( \LOC{A}{ s} ) 
\leq 1 - \varepsilon   }.
\]
By definition of \( B \), any one of its points has arbitrarily small neighborhoods \( \Nbhd_s \) 
such that \( \mu ( \LOC{A}{s} ) \leq 1 - \varepsilon \), that is
\begin{equation}\label{eq:th:Lebesguedensity1}
 \FORALL{x \in B} \EXISTSS{\infty}{m} \left ( x \restriction m \in \mathcal{B} \right ) .  
\end{equation}
If \( \mathcal{A} \subseteq \mathcal{B} \) is an antichain (i.e., \( \Nbhd_s \cap \Nbhd_t = \emptyset \) 
for distinct \( s , t \in \mathcal{A} \)) then 
\begin{equation*}
\begin{split}
 \mu  ( B \cap \textstyle\bigcup_{s \in \mathcal{A}} \Nbhd_s ) & 
\leq \textstyle \sum_{s \in \mathcal{A}} \mu ( A \cap \Nbhd_s )
\\
 & \leq ( 1 - \varepsilon ) \cdot  \textstyle\sum_{s \in \mathcal{A}} \mu ( \Nbhd_s )
 \\
 & \leq ( 1 - \varepsilon ) \cdot \mu ( U ) 
 \\
 & < \mu  ( B ) ,
\end{split} 
\end{equation*}
hence 
\begin{equation}\label{eq:th:Lebesguedensity2}
 \FORALL{\mathcal{A} \subseteq \mathcal{B}} \left ( \mathcal{A} \text{ antichain} \IMPLIES 
\mu  ( B \setminus \textstyle\bigcup_{s \in \mathcal{A}} \Nbhd_s ) > 0 \right ) .
\end{equation}
Construct pairwise incompatible \( s_n \in \mathcal{B} \) as follows.
Let \( s_0 \in \mathcal{B} \) be arbitrary, and suppose \( s_0 , \dots , s_n \) have been chosen:
by~\eqref{eq:th:Lebesguedensity2} the set 
\( B \setminus \left ( \Nbhd_{s_0} \cup \dots \cup \Nbhd_{s_n} \right ) \) 
is not null,
and for any \( x \) in this set there are arbitrarily large \( m \) such that 
\( x \restriction m \in \mathcal{B} \) by~\eqref{eq:th:Lebesguedensity1}.
In particular, the collection 
\( \mathcal{B}_n = \setofLR{ s \in \mathcal{B}}{ \FORALL{i \leq n} \left ( s_i \perp s \right ) } \)
is nonempty, so let \( s_{n + 1 }\) be an element of \( \mathcal{B} \) of minimal length.
Since \( \mathcal{B}_n \supset \mathcal{B}_{n + 1} \) it follows that 
\( \lh ( s_n ) \leq \lh ( s_{n + 1} ) \) for all \( n \), hence \( \lh ( s_n ) \to \infty \).
As \( \setofLR{s_n}{n \in \omega } \subseteq \mathcal{B}\) is an antichain, there is an
\( \bar{x} \in B \setminus \bigcup_{n} \Nbhd_{s_n} \) and
by~\eqref{eq:th:Lebesguedensity1} there is an \( \bar{m} \) such that 
\( \bar{s} = \bar{x} \restriction \bar{m} \in \mathcal{B} = \mathcal{B}_{-1} \).
We will show by induction on \( n \) that \( \bar{s} \in \mathcal{B}_n \) ---
as \( \lh ( \bar{s} ) < \lh ( s_n ) \) for large enough \( n \), and \( \bar{s} \in \mathcal{B}_{n + 1 }\),
this would contradict the choice of \( s_{n + 1} \). 
Assume \( \bar{s} \in \mathcal{B}_n \): towards proving that \( \bar{s} \in \mathcal{B}_{n + 1} \)
it is enough to show that \( \bar{s} \perp s_{n + 1} \).
Assume otherwise, that is either \( s_{n + 1} \subseteq \bar{s} \) or \( \bar{s} \subset s_{n + 1} \).
If \( s_{n + 1} \subseteq \bar{s} \), then \( \bar{x} \in \Nbhd_{\bar{s}} \subseteq \Nbhd_{s_{n + 1}} \),
against \( \bar{x} \in B \setminus \bigcup_{i} \Nbhd_{s_i} \), and if \( \bar{s} \subset s_{n + 1} \)
this would go against the minimality of \( \lh ( s_{n + 1} ) \), hence either way
a contradiction is reached.

\bibliographystyle{alpha}
\bibliography{density}

\end{document}